\documentclass[11pt]{amsart}

\usepackage[margin=1.25in]{geometry}
\usepackage{setspace}
\usepackage[parfill]{parskip}
\usepackage{amssymb}
\usepackage{graphicx}
\usepackage{hyperref}

\usepackage{calrsfs}
\DeclareMathAlphabet{\pazocal}{OMS}{zplm}{m}{n}

\usepackage{tikz}
\usetikzlibrary{decorations.markings}

\usepackage{ytableau}

\theoremstyle{plain}
\newtheorem{thm}{Theorem}[section]
\newtheorem{cor}[thm]{Corollary}
\newtheorem{lemma}[thm]{Lemma}
\newtheorem{prop}[thm]{Proposition}
\newtheorem{definition}[thm]{Definition}

\theoremstyle{remark}
\newtheorem{rem}[thm]{Remark}

\numberwithin{equation}{section}

\newcommand{\eps}{\epsilon}
\newcommand{\C}{\mathbb{C}}
\newcommand{\F}{\mathcal{F}}
\newcommand{\Lb}{\pazocal{L}}
\newcommand{\R}{\mathbb{R}}
\newcommand{\T}{\pazocal{T}}
\newcommand{\Tb}{\mathbf{T}}
\newcommand{\X}{\pazocal{X}}

\newcommand{\Z}{\mathbb{Z}}
\newcommand{\lm}{\lambda}
\newcommand{\sst}{\Delta_n}
\newcommand{\st}{\Delta_{\infty}}

\newcommand{\pr}[1]{\mathbb{P}\left [ #1 \right]}
\newcommand{\E}[1]{\mathbb{E}\left [ #1 \right ]}
\newcommand{\ind}[1]{\mathbf{1}_{\{#1\}}}
\newcommand{\G}[1]{\Gamma \left(#1 \right)}
\newcommand{\dint}[2]{\frac{1}{(2\pi \mathbf{i})^2} \oint \limits_{C_z #1} dz \oint \limits_{C_w #2} dw\,}

\title[Local statistics of RSN]{Random sorting networks: local statistics via random matrix laws}

\author{Vadim Gorin}

\address{Department of Mathematics,
MIT,  Cambridge, MA, USA and Institute for
Information Transmission Problems of Russian Academy of Sciences, Moscow, Russia.}
\email{vadicgor@gmail.com}

\author{Mustazee Rahman}

\address{Department of Mathematics,
MIT, Cambridge, MA, USA.}
\email{mustazee@gmail.com}

\date{}
\keywords{Sorting network, reduced decomposition, Gaudin-Mehta law, GUE corners, Young tableau, determinantal point process}

\begin{document}

\begin{abstract}
This paper finds the bulk local limit of the swap process of uniformly random
sorting networks. The limit object is defined through a deterministic procedure, a
local version of the Edelman-Greene algorithm, applied to a two dimensional
determinantal point process with explicit kernel. The latter describes the
asymptotic joint law near $0$ of the eigenvalues of the corners in the antisymmetric
Gaussian Unitary Ensemble. In particular, the limiting law of the first time a given
swap appears in a random sorting network is identified with the limiting
distribution of the closest to $0$ eigenvalue in the antisymmetric GUE. Moreover,
the asymptotic gap, in the bulk, between appearances of a given swap is the
Gaudin-Mehta law -- the limiting universal distribution for gaps between eigenvalues of
real symmetric random matrices.

The proofs rely on the determinantal structure and a double contour integral
representation for the kernel of random Poissonized Young tableaux of arbitrary shape.
\end{abstract}

\maketitle

\section{Introduction} \label{sec:intro}

\subsection{Overview}

The main object of this article is the uniformly random sorting network, as
introduced by Angel, Holroyd, Romik, and Vir\'{a}g in \cite{AHRV}. Let $\mathfrak{S}_n$ denote
the symmetric group and $\tau_i$ denote the transposition between $i$ and $i+1$ for
$1 \leq i \leq n-1$. The $\tau_i$ are called \emph{adjacent swaps}. Let
$\mathrm{rev} = n, n-1, \ldots, 1$ denote the reverse permutation of $\mathfrak{S}_n$.
A sorting network of $\mathfrak{S}_n$ is a sequence of permutations $\sigma_0 = \mathrm{id},
\sigma_1, \ldots, \sigma_N = \mathrm{rev}$ of shortest length with the property that for
every $k$, $$\sigma_{k+1} = \sigma_k \circ \tau_i \;\;\text{for some}\; i.$$
In other words, the permutations change by swapping adjacent labels at each step and must go from the
identity to the reverse in the shortest number of swaps. The number of adjacent
swaps required in any sorting network of $\mathfrak{S}_n$ is $\binom{n}{2}$. See Figure
\ref{fig:wiring} for an example of a sorting network in the wiring diagram
representation. We identify a sorting network of $\mathfrak{S}_n$ by its sequence of swaps
$$\Big(s_1,\,\ldots,\,s_{\binom{n}{2}} \Big),$$
where $s_i$ denotes the adjacent swap $(s_i, s_i+1)$.

A \emph{random sorting network} of $\mathfrak{S}_n$ is a sorting network of $\mathfrak{S}_n$ chosen
uniformly at random. Computer simulations were used to conjecture many beautiful asymptotic
properties of random sorting networks. See \cite{AHRV} for an account of these statements and the first
rigorous results, and also \cite{ADHV, AGH, AH, AHR, D_Virag, KV, RVV, Rozinov} for other
asymptotic theorems. The proofs of the conjectures from \cite{AHRV} were recently announced in \cite{D}.

\begin{figure}[t]
\begin{center}
\includegraphics[scale=0.4]{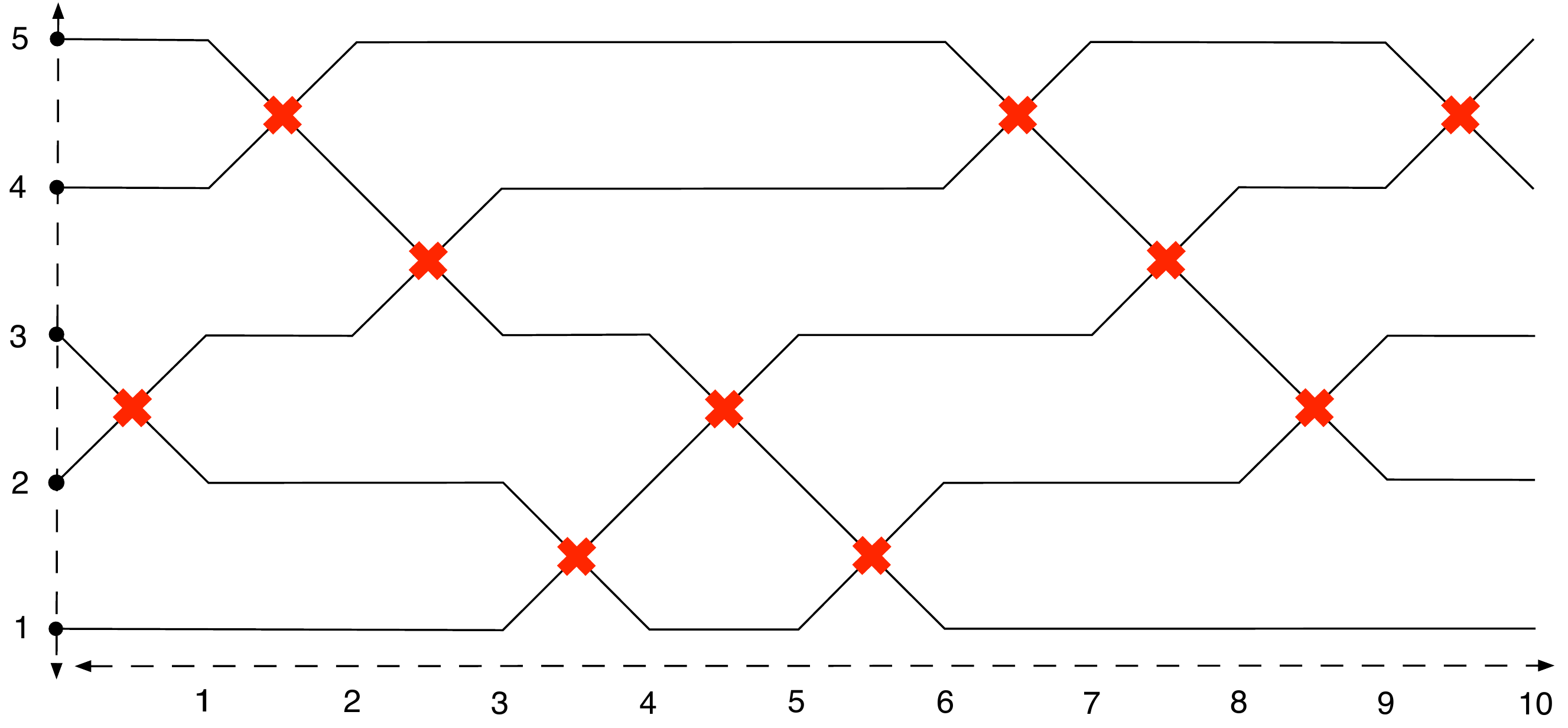}
\caption{\small{Wiring diagram of a sorting network of $\mathfrak{S}_5$ with swap sequence
(2,4,3,1,2,1,4,3,2,4). Intersection of two paths at location $(i-1/2, j+1/2)$
indicates a swap at time $i$ between labels at positions $j$ and $j+1$. The
intersection locations (red crosses) of a random sorting network of $\mathfrak{S}_n$ has a
distributional limit in windows of unit order in the vertical direction and order $n$ in the
horizontal direction.}} \label{fig:wiring}
\end{center}
\end{figure}

In many examples, random combinatorial structures built out of symmetric groups
are known to exhibit the same asymptotic behavior as random matrices. The most
famous result of this sort due to Baik--Deift--Johansson \cite{BDJ} identifies the
fluctuations of longest increasing subsequences of random permutations with the
fluctuations of largest eigenvalue of random Hermitian matrices. Its further
upgrades, \cite{BOO, Johansson_plancherel, Ok} link fluctuations of
several first rows of the Young diagram distributed according to the Plancherel
measure for symmetric groups to those of several largest eigenvalues. A connection
also exists for ``bulk'' (i.e.\ not largest) rows and eigenvalues, see \cite{BOO}.

On the other hand, up to now no such connections were known for random sorting
networks. In the present article we find such a connection. It exists for a sort of local
limit of random sorting networks. Indeed, we find the bulk local limit of random sorting
networks, by proving that it is given by a simple, local, deterministic algorithm
(the local Edelman-Greene algorithm) applied to a specific random point process on
$\Z \times \R_{\geq 0}$. In turn, we describe this point process by showing
that its correlation functions have determinantal form and provide explicit
expressions for the corresponding correlation kernel. The very same point process
appeared in the work of Forrester--Nordenstam \cite{FN} (see also Deffoseux
\cite{Def}) as the hard edge limit of antisymmetric GUE--corners process; it
describes the asymptotic distribution of the closest to 0 eigenvalues of the corners
of large antisymmetric matrix with i.i.d.\ (modulo symmetry) Gaussian entries of
mean $0$.

A corollary of our theorem is that the rescaled, asymptotic distribution of the first time when the
swap between $\lfloor \frac{n(1+\alpha)}{2} \rfloor$ and $\lfloor \frac{n(1+\alpha)}{2} \rfloor + 1$
appears, for $\alpha \in (-1,1)$, is the same as the rescaled, asymptotic distribution of the closest to $0$
eigenvalue of an antisymmetric-GUE random matrix. Another corollary is that within the bulk,
the asymptotic gap between appearances of the aforementioned swap is described by the
Gaudin--Mehta law --- the asymptotic universal distribution of the gap between eigenvalues of
real symmetric random matrices in the bulk. Complete statements are given in the next section.

In an independent and parallel work, Angel, Dauvergne, Holroyd, and Vir\'{a}g \cite{ADHV}
also study the bulk local limit of random sorting networks. Their approach is very
different from ours. We deduce explicit formulas for the prelimit local structure of
random sorting networks, and then analyze the asymptotic of these formulas in the
spirit of \emph{Integrable Probability}, see \cite{BG_review, BP_review} and also \cite{Ro}.
On the other hand, \cite{ADHV} argue probabilistically, analyzing a Markov chain (whose
transition probabilities are expressed through the hook formula for dimensions) for
sampling random Young tableaux. The connection to random matrices remains invisible
in the results of \cite{ADHV}. It would be interesting to match these two approaches,
but it has not been done so far.

\subsection{Bulk limit of random sorting networks.}

We now describe our main result. Informally, we study
the asymptotics of the point process $(s_i,i)$, $i=1,\dots,\binom{n}{2}$, in a window
of finite height and order $n$ width, so that the number of points in the window
remains finite; see Figure \ref{fig:wiring}. Here $(s_1,\dots,s_{\binom{n}{2}})$ are
swaps of a random sorting network of $\mathfrak{S}_n$.

In \cite{AHRV} it is proven that the point process $(s_i,i)$ is stationary with respect to
the second coordinate. Therefore, it suffices to study windows adjacent to $0$
in second coordinate, which we do.

The limiting object $S_{\rm{local}}$ is a point process on $\Z \times \R_{\geq 0}$
defined by a two--step procedure. First, we introduce an auxiliary point
process $\X_{\rm{edge}}$ on $\Z \times \R_{\geq 0}$ through its correlation functions.

\begin{definition} \label{def:Xedge}
$\X_{\rm{edge}}$ is the (unique) determinantal point process on $\Z \times \R_{\geq 0}$
with correlation kernel
 \begin{align*}
&K_{\rm{edge}}(x_1,u_1;x_2,u_2) =
\begin{cases}
\displaystyle \frac{2}{\pi}  \int \limits_0^1 t^{x_2-x_1}
\cos \left( tu_1 + \frac{\pi}{2} x_1\right) \cos \left( tu_2 + \frac{\pi}{2}x_2 \right )\,dt, & \text{if}\;\; x_2 \geq x_1; \\
& \\
\displaystyle - \frac{2}{\pi} \int \limits_1^{\infty} t^{x_2-x_1} \cos \left(
tu_1 + \frac{\pi}{2} x_1\right) \cos \left( tu_2 + \frac{\pi}{2}x_2 \right )\,dt, & \text{if}\;\; x_2 < x_1.
\end{cases}
\end{align*}
\end{definition}
We refer to \cite{Bor} and Section \ref{sec:dpp} for more detailed discussions of
determinantal point processes. We note that the particles of $\X_{\rm edge}$ on
adjacent lines $\{x\} \times \R_{\geq 0}$ and $\{x+1\}\times \R_{\geq 0}$ almost surely
interlace, see Figure \ref{fig:Xedge}. The point process $\X_{\rm{edge}}$ has appeared in
the random matrix literature before in \cite{FN}, \cite{Def}. In more details, let $G$ be an
infinite random matrix with rows and columns indexed by $\Z_{> 0}$, and whose entries are
independent and identically distributed, real-valued, standard Gaussians.
Let $A = \frac{G - G^{T}}{\sqrt{2}}$. The top--left $m\times m$ corner of
$A$ almost surely has $2 \lfloor m/2\rfloor$ non-zero eigenvalues of the form
$$\pm \mathbf i \lm_1^m, \, \pm \mathbf i \lm_2^m , \, \ldots, \, \pm \mathbf i \lm_{\lfloor m/2 \rfloor}^m,$$
where $0 < \lm_1^m < \lm_2^m < \cdots < \lm_{\lfloor m/2 \rfloor}^m$. Forrester and Nordenstam prove
that $\X_{\rm edge}$ is the weak limit of the point process
$\{(j,\sqrt{2M}\lm_i^{2M+j})\}\subset \mathbb Z\times \mathbb R_{\ge 0}$, $i\in
\mathbb Z_{> 0}$, $j\in \mathbb Z$, as $M \to \infty$.

Particle configurations of $S_{\rm{local}}$ are obtained from $\X_{\rm{edge}}$ by a
deterministic procedure, which is a local version of the well-known Edelman--Greene
bijection \cite{EG} between staircase shaped tableaux and sorting networks.
In the following we describe this procedure. A rigorous definition of the procedure
utilizes properties of $\X_{\rm{edge}}$ that are not immediate from Definition \ref{def:Xedge}.
We provide the rigorous construction in Section \ref{sec:localEG} where the description is given
in the language of Young tableau, which is the more standard setup for defining the Edelman-Greene
bijection.

\begin{figure}[htpb]
    \begin{center}
        \includegraphics[scale=0.55]{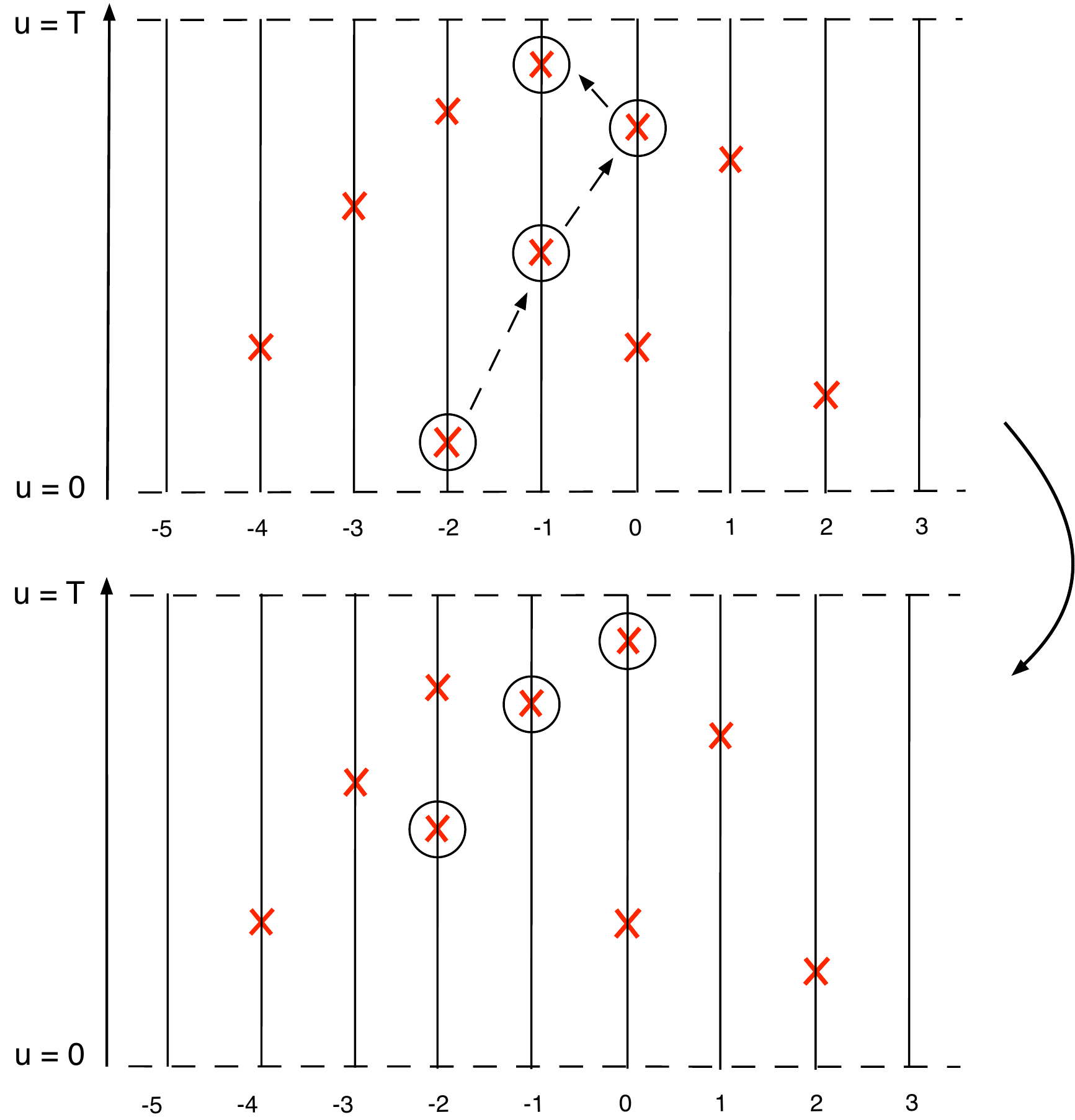}
        \caption{\small{Top: A possible configuration (red crosses) of $\X_{edge}$ restricted to
                $[-5,3] \times [0,T]$ with no particles on $\{-5,3\} \times [0,T]$. Particles on consecutive lines
                interlace. Encircled points represent the sliding path during the first step of the Edelman-Greene algorithm.
                Bottom: The result after the first step of the Edelman-Greene algorithm.}} \label{fig:Xedge}
    \end{center}
\end{figure}

\paragraph{\textbf{Local Edelman-Greene algorithm}}
Fix a configuration $X$ of $\X_{\rm edge}$ and suppose that we
want to define the positions of all particles of $S_{\rm{local}}$ inside the
rectangle $[a,b]\times [0,T]$ with $a<0<b$. Then almost surely there are two integers
$\hat a<2a$ and $\hat b>2b$ such that $X$ has no particles on the segments
$\hat a\times [0,T]$ and $\hat b\times [0,T]$. The particles of $X$ outside
$[\hat a,\hat b]\times[0,T]$ are further ignored.

We now define a particle configuration $Y$ -- the restriction of  $S_{\rm{local}}$
onto $[\hat a,\hat b]\times[0,T]$ -- through an iterative procedure. Start by declaring
$Y=\emptyset$, and setting $\hat X$ to be the restriction of $X$ onto
$[\hat a,\hat b]\times[0,T]$. Repeat the following until $\hat X$ is empty:

\begin{enumerate}
 \item \label{Step_1} Let $(x,u)$ be an element of $\hat X$ with smallest second coordinate.
 The parity of $x$ will be even. Add $(x/2,u)$ to $Y$, i.e., redefine $Y:=Y\cup\{(x/2,u)\}$.
 \item Define the \emph{sliding path} $(x_1,u_1), (x_2,u_2), \ldots$ as a unique collection of
 points in $\hat X$ (of maximal length) such that
  \begin{itemize}
   \item $(x_1,u_1)=(x,u)$,
   \item $u_1<u_2<\cdots<u_k$ and $|x_i-x_{i+1}|=1$ for $i=1,\dots,k-1$,
   \item For each $i=1,\dots,k-1$, the only points of $\hat X$ in the rectangle
   $[x_i-1,x_i+1]\times [u_i,u_{i+1}]$ are $(x_i,u_i)$ and $(x_{i+1},u_{i+1})$.
  \end{itemize}
  In other words, $(x_{i+1},u_{i+1})$ is the point in $[x_i-1,x_i+1] \times (u_i,T]$, which
  is closest to $(x_i,u_i)$. See Figure \ref{fig:Xedge} for an illustration.
 \item Remove the $k$ points $(x_1,u_1)$,\dots, $(x_{k},u_k)$ from $\hat X$ and
 replace them by $k-1$ points $(x_1,u_2)$, $(x_2,u_3)$, \dots, $(x_{k-1},u_k)$.
 \item Go back to Step \eqref{Step_1}, unless $\hat X$ is empty.
\end{enumerate}
The first coordinates of the particles of $Y$ will be integral; this follows from
the interlacing property of the particles of $\hat X$, which is preserved throughout
the steps of the procedure.

One immediate property of the just defined map $\X_{\rm{edge}}\mapsto
S_{\rm{local}}$ is that the position of the first particle of $S_{\rm{local}}$
in the ray $\{a\} \times \R_{\geq 0}$ almost surely coincides with the position
of the first particle of $\X_{\rm{edge}}$ in the ray $\{2a\} \times \R_{\geq 0}$.
Therefore, the joint law of the positions of the first particles of $S_{\rm{local}}$
in the rays $\{a_i\} \times \R_{\geq 0}$, for $i=1,\dots,k$, can be explicitly evaluated as a
Fredholm determinant. See Corollary \ref{cor:firstswap} for the case $k=1$ and
\cite{Bor} for general statements.

We also show that $S_{\rm{local}}$ is invariant under translations and reflections
of the first ($\Z$--valued) coordinate, ergodic with respect to translations of the
first coordinate, and stationary in the second ($\R_{\geq 0}$--valued)
coordinate; see Proposition \ref{thm:localsortingstats}.

We are ready to formulate the main result.
\begin{thm}[Local random sorting network] \label{thm:localsorting}
Fix $\alpha \in (-1,1)$, and let $s_1,s_2,\dots,s_{n\choose 2}$ be swaps of
a random sorting network of $\mathfrak{S}_n$. Define the point process $S_{\alpha,n}$ of
rescaled swaps near the point $(\frac{n(\alpha+1)}{2},0)$ through
$$
S_{\alpha,n}=\left\{\left(s_i-\left\lfloor \frac{n(\alpha+1)}{2} \right\rfloor ,
\sqrt{1-\alpha^2}\cdot \frac{2i}{n}\right)\right\}_{i=1}^{n\choose 2}.
$$
Then as $n\to\infty$, the point process $S_{\alpha,n}$ converges weakly to $S_{\rm local}$.
\end{thm}
It is proven in \cite[Theorem 2]{AHRV} that the global scaling limit of the
space-time swap process of random sorting networks is the product of the semicircle
law and Lebesgue measure. The $\sqrt{1-\alpha^2}$ scaling of Theorem
\ref{thm:localsorting} is consistent with the semicircle result.

We emphasize that Theorem \ref{thm:localsorting} states both that $S_{\alpha,n}$
converges and the limit is obtained by applying the localized Edelman-Greene
algorithm to $\X_{\rm{edge}}$. Theorem \ref{thm:localsorting} does not cover the
case $|\alpha|=1$, where the asymptotic behavior changes. It is plausible that the
methods of the present article can be adapted to this remaining case, but we
not address it here; see \cite{Rozinov} for another approach to $|\alpha|=1$ case.

Theorem \ref{thm:localsorting} implies that the first
swap times in random sorting networks converge to a one--dimensional
marginal of $S_{\rm local}$; the distribution of the latter can be expressed as a
Fredholm determinant. Figure \ref{fig:firstswap} shows the approximate
sample distribution of the rescaled first swap time and \eqref{eqn:Dyson} shows the tail asymptotics.
\begin{cor}[First swap law]\label{cor:firstswap}.
Let $\Tb_{FS, \alpha, n}$ be the first time the swap interchanging $\left\lfloor
\frac{n(\alpha+1)}{2} \right\rfloor$ with $\left\lfloor \frac{n(\alpha+1)}{2}
\right\rfloor+1$ appears in a random sorting network of $\mathfrak{S}_n$. The following
convergence in law holds:
$$ \lim_{n\to\infty}  \, \frac{2 \sqrt{1-\alpha^2} }{n} \, \Tb_{FS, \alpha, n} = \Tb_{\rm{FS}},$$
where
\begin{align} \label{eqn:firstswaplaw}
\pr{ \Tb_{\rm{FS}} > t} &= 1 + \sum_{k=1}^{\infty} \frac{(-1)^k}{k!} \int \limits_{[0,t]^k} \det[ K_{\rm{edge}}(u_i,u_j)] \, du_1 \cdots du_k,
\;\text{and} \\
\nonumber K_{\rm{edge}}(u_1,u_2) &= \frac{\sin(u_1-u_2)}{\pi\,(u_1-u_2)} + \frac{\sin(u_1+u_2)}{\pi\,(u_1+u_2)}\,.
\end{align}
\end{cor}

\paragraph{\textbf{Connection to the Gaudin-Mehta law}}
A further consequence deals with the limiting law of the gap between swaps on the same horizontal line
in random sorting networks. Fix $\beta \in (0,1)$. Given a random sorting network of $\mathfrak{S}_n$,
let $\Tb_{+}$ be the distance between $ \left \lfloor \beta \binom{n}{2} \right \rfloor$ and the closest to its right
swap interchanging $\left\lfloor \frac{n(\alpha+1)}{2} \right\rfloor$ with $\left\lfloor \frac{n(\alpha+1)}{2} \right\rfloor+1$. 
Let $\Tb_{-}$ be the analogous distance to the closest to its left swap.

Due to stationarity of random sorting networks, the joint law $\Tb_{-}$ and $\Tb_{+}$ is given by
\begin{equation} \label{eqn:gaplaw}
\pr{\Tb_{-} > a, \Tb_{+} > b} = \pr{\Tb_{FS, \alpha, n} > a+ b}.
\end{equation}
Indeed, due to stationarity, both sides of \eqref{eqn:gaplaw} give the probability of the event
that there are no swaps in the interval $[-a,b]$ after the appropriate re-centerings.
Equation \ref{eqn:gaplaw} shows that the law of $(\Tb_{+}, \Tb_{-})$, and hence, of the gap $\Tb_{-} +
\Tb_{+}$, is determined by the law of the first swap time $\Tb_{FS, \alpha, n}$. In
particular, their limiting law after rescaling by $\sqrt{1-\alpha^2}/ n$ is uniquely
determined from the distribution function \eqref{eqn:firstswaplaw}.

This is connected to the scaling limit of the point process of eigenvalues of GOE random matrices
in the bulk. The scaling limit of the eigenvalues of GOE random matrices near $0$ is stationary.
Let $-\Lambda_{-}$ and $\Lambda_{+}$ be, respectively, the closest to $0$ negative and closest to
$0$ positive point in the limit process. If the matrices are scaled so that the mean eigenvalue
gap near $0$ is $1$, then \eqref{eqn:firstswaplaw} is the distribution function of
$(\pi/2)\Lambda_{+}$. In other words,  \eqref{eqn:firstswaplaw} is the asymptotic probability to
see no eigenvalues in an interval of length $(2/\pi)t$ for large GOE random matrices, normalized so
that the mean eigenvalue gap around the interval is $1$; see e.g.~\cite{Gaudin}, \cite{Dyson},
\cite[(8.139) and (9.81)]{Forrester}. The gap between points, $\Lambda_{-} + \Lambda_{+}$, has
its law determined from that of $\Lambda_{+}$ according to \eqref{eqn:gaplaw}. This is the
celebrated \emph{Gaudin--Mehta} law, originally put forward by Wigner as a model for the gap
between energy levels in heavy nuclei and later found in numerous systems.
We arrive at the following corollary.

\begin{cor}[Gap law] \label{cor:spacing_GM}
For $\alpha \in (-1,1)$ and $\beta \in (0,1)$, let
$\mathrm{Gap}_{\alpha,\beta,n}$ be the distance in a random sorting network of
$\mathfrak{S}_n$ between the two swaps interchanging $\left\lfloor
\frac{n(\alpha+1)}{2} \right\rfloor$ with $\left\lfloor \frac{n(\alpha+1)}{2}
\right\rfloor+1$: the one closest from the left to time $\beta \binom{n}{2}$
and the one closest from the right to $\beta \binom{n}{2}$.
Then, the distributional limit
$$ \lim_{n\to\infty}  \, \frac{4\sqrt{1-\alpha^2} }{\pi n} \, \mathrm{Gap}_{\alpha,\beta,n} $$
is the Gaudin--Mehta law, i.e.~the asymptotic gap in the bulk between eigenvalues of
real symmetric random matrices with mean gap one.
\end{cor}

\begin{figure}[t]
\begin{center}
\includegraphics[scale=0.5]{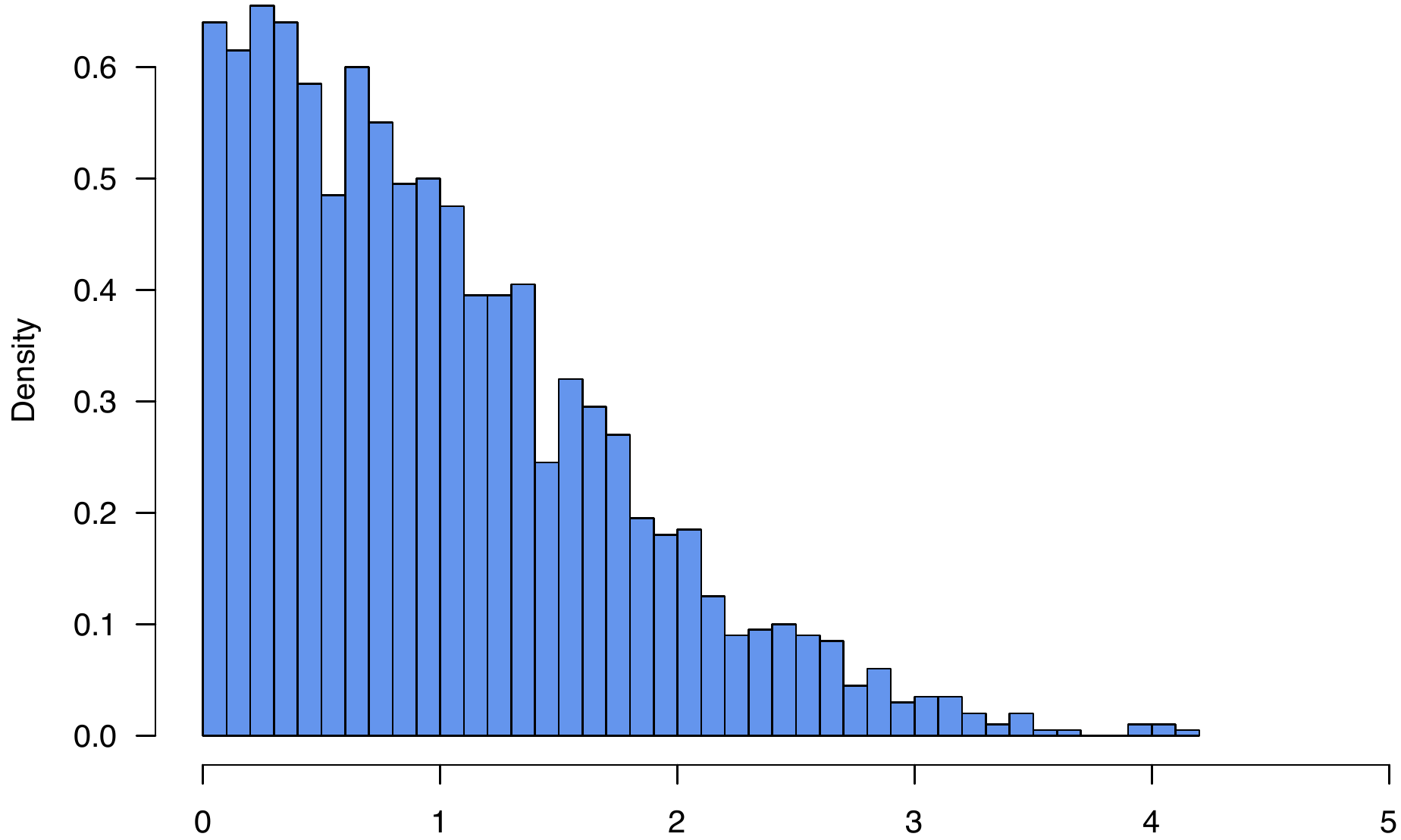}
\caption{\small{Density histogram of the rescaled first swap time
for element $500$ in a 1000 element random sorting network.}} \label{fig:firstswap}
\end{center}
\end{figure}

The proof of Theorem \ref{thm:localsorting} builds upon two ideas. The first one (which
is also used in most of the rigorous results on sorting networks) is to reduce
the study of random sorting networks to uniformly random staircase shaped
standard Young tableaux via the Edelman--Greene bijection \cite{EG} (see also \cite{HY}).
Our observation is that if we \emph{Poissonize} uniformly random standard Young
tableaux (of arbitrary shape!), then the result can be described by a determinantal
point process with an explicit correlation kernel written as a double contour
integral. We further show that the Poissonization does not change the local statistics,
and therefore, the limit theorem is reduced to the asymptotic analysis of the
aforementioned double contour integral, which we perform.

Our results on the correlations and limiting behavior of random standard Young tableaux
might be of independent interest, and so we present them in the next section.

\subsection{Random Standard Young Tableaux}

A partition $\lambda$ is a sequence of non-negative integers $\lambda_1 \geq
\lambda_2 \geq \cdots \geq 0$ such that $|\lambda|:=\sum_{i=1}^{\infty}\lambda_i<\infty$.
The \emph{length} of $\lambda$, denoted $\ell(\lambda)$, is the number of positive $\lm_i$
and the \emph{size} of $\lm$ is $|\lm|$.

We identify a partition with a \emph{Young diagram} (YD), which is the set of lattice points
$$\{ (i,j) \in \Z^2:\, i\ge 1,\, 1 \leq j \leq \lm_i\}.$$
The points of the Young diagram $\lm$ are its \emph{cells} and we say
the Young diagram has \emph{shape} $\lm$.
Given a pair of YDs $\lm$ and $\mu$, we write $\lm \preceq \mu$
if the cells of $\lm$ are contained within the cells of $\mu$. If the containment
is strict then $\lm \prec \mu$. If $\lm \preceq \mu$ then $\mu \setminus \lm$
denotes the cells of $\mu$ that are not in $\lm$. A \emph{standard Young tableau} (SYT)
of shape $\lm$ is an insertion of the numbers $1,2, \ldots, |\lm|$ into the cells of $\lm$
such that they strictly increase along the rows (from left to right) and also along the
columns (from bottom to top). The numbers within a SYT are its \emph{entries}.
The set of SYTs of shape $\lm$ is in bijection with the set of increasing sequences of YDs
\begin{equation} \label{eqn:YDbijection}
\emptyset = \lm^{(0)} \prec \lm^{(1)} \prec \lm^{(2)} \prec \cdots \prec
\lm^{(|\lm|)} = \lm
\end{equation}
such that the entry $k$ is inserted into the singleton cell of $\lm^{(k)} \setminus
\lm^{(k-1)}$.

A staircase shaped SYT of length $n-1$ (or also $n-1$ rows) is a SYT of shape
$(n-1,n-2,\ldots,2,1)$, which we denote $\sst$.
The Edelman--Greene bijection \cite{EG} gives a one-to-one correspondence
between staircase shaped SYTs and sorting networks; see Section \ref{sec:EG}
for the details. This is the reason for our interest in SYTs.

A Poissonized Young tableau (PYT) of shape $\lambda$ is an insertion of distinct
real numbers from the interval $(0,1)$ into the cells of $\lm$ such that they
strictly increase along the rows and  along the columns. Note that if we replace the
entries of a PYT by their relative ranks then we get a SYT. The set of PYTs of
shape $\lm$ is in bijection with the set of increasing sequences of YDs indicating the times of jumps:
\begin{equation} \label{eqn:PYDbijection}
\emptyset = \lm^{(0)} \stackrel{t_1}{\prec} \lm^{(1)} \stackrel{t_2}{\prec}
\lm^{(2)} \stackrel{t_{3}}{\prec} \cdots \stackrel{t_{|\lm|}}{\prec} \lm^{(|\lm|)} = \lm
\end{equation}
such that the entry $t_k$ is inserted in the singleton cell of $\lm^{(k)} \setminus
\lm^{(k-1)}$. These increasing sequences of Young diagrams
with labels were discussed in \cite{BO_YB} in the connection to the \emph{Young
bouquet}; see also \cite{Nica}.

We would like to identify a PYT with a collection of non-intersecting
paths. For that we first map a Young diagram $\lambda$ to a countable particle
configuration $\{\lambda_i-i+1/2\}_{i=1,2,\dots}\subset \mathbb Z+1/2$. This
procedure can be viewed as projecting the boundary of the Young diagram in
\emph{Russian notation} onto a horizontal line, see Figure \ref{fig:YD_projection}.
The empty Young diagram $\emptyset$ corresponds to $\{-1/2,-3/2,-5/2,\dots\}$.

\begin{figure} [htpb]
\begin{center}
\includegraphics[scale=0.8]{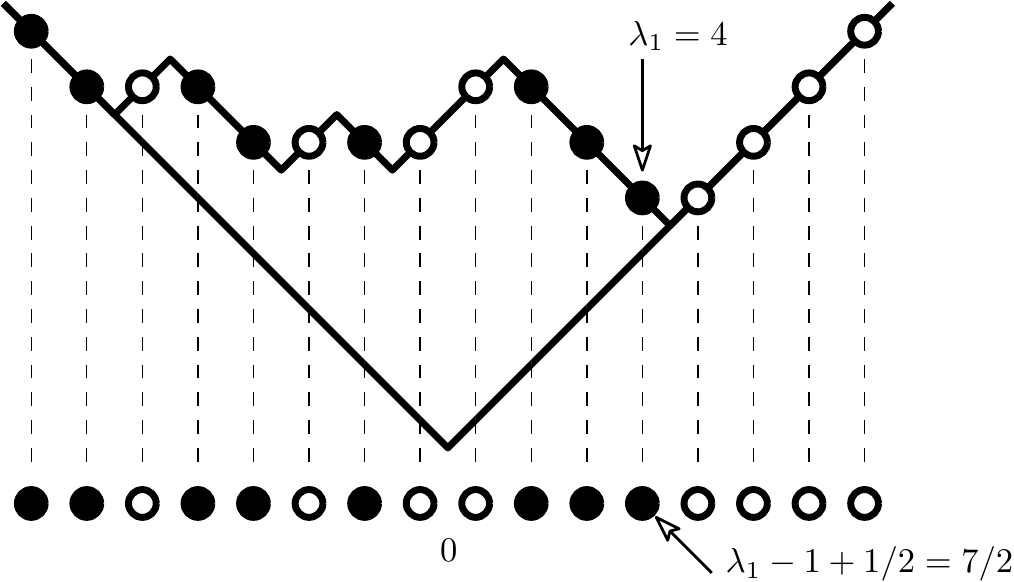}
\caption{\small{Young diagram $(4,4,4,2,1,1)$ and corresponding particle configuration
$(7/2, 5/2, 3/2, -3/2, -7/2, -9/2, -13/2, -15/2,\dots)$.}} \label{fig:YD_projection}
\end{center}
\end{figure}

Give a PYT, for each $t$ consider the countable particle configuration corresponding
to the Young diagram filled with the entries $\leq t$ in the PYT. The trajectories of
particles then form a collection of paths, making jumps to the right at the times
indexed by the entries $t_k$ of the tableau (equivalently, labels in
\eqref{eqn:PYDbijection}). Let us draw a cross at a point $(x,t)$, $x\in\mathbb Z$,
$0<t<1$, if a particle jumps from $(x-1/2)$ to $(x+1/2)$ at time $t$; see Figure
\ref{fig:interlace}. Although there are infinitely many particles, the only ones that move
are the $\ell(\lm)$ particles that correspond to the rows of $\lm$ with positive size.

\begin{figure}[t]
\begin{center}
\includegraphics[scale=0.7]{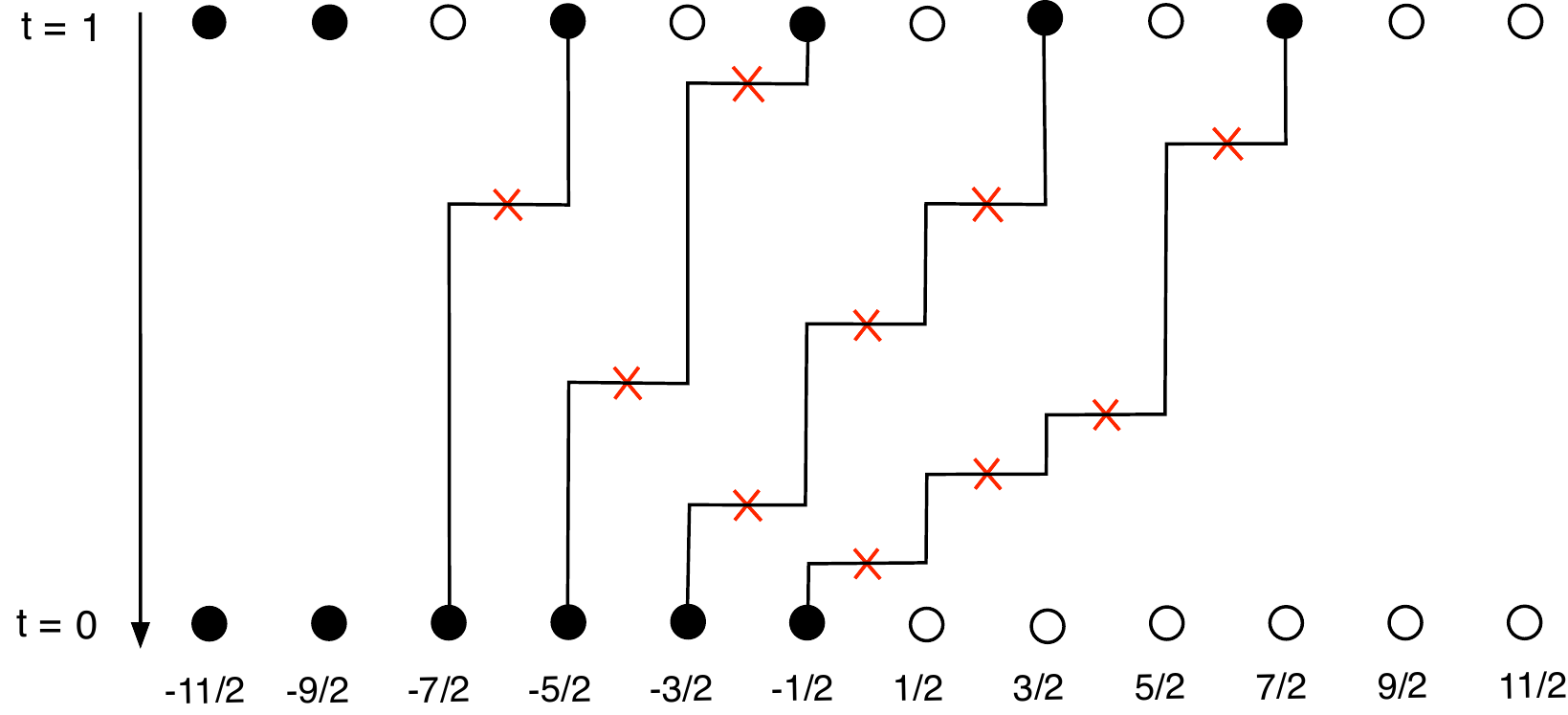}
\caption{\small{Particle system associated to a staircase shaped PYT of size 10.
Particles move along non-intersecting paths. The jumps (red crosses) of a random
PYT of fixed shape is a determinantal point process. The local window of a PYT
of shape $\sst$ consists of the trajectories of a group of $L$ successive particles
traced down from $t=1$ to $t = 1-(u/n)$.}} \label{fig:interlace}
\end{center}
\end{figure}

\begin{thm}[Poissonized tableaux]  \label{thm:pyt}
Given a finite Young diagram $\lambda$, consider the point process
$\X_{\lm}$ of jumps of a uniformly random Poissonized Young tableau of
shape $\lambda$. $\X_{\lm}$ is a determinantal point process on $\Z \times [0,1]$
with correlation kernel $K_{\lm}(x_1,t_1; x_2,t_2)$ as follows.
For $x_1,x_2 \in \Z$ and $t_1, t_2 \in [0,1]$,
\begin{align*}
&K_{\lm}(x_1,t_1; x_2;t_2) = \ind{t_2 > t_1,\, x_1 > x_2}\,\frac{(t_1-t_2)^{x_1-x_2-1}}{(x_1-x_2-1)!} \,+\\
& \dint{[0,\lm_1-x_2)}{[0,n+x_1)} \frac{\G{-w} }{\G{z+1}} \cdot
\frac{G_{\lm}(z+x_2)}{G_{\lm}(x_1-1-w)}  \cdot \frac{(1-t_2)^z (1-t_1)^w}{w+z+x_2-x_1+1},\\
& \text{where}\;\; G_{\lm}(u) = \G{u+1} \prod_{i=1}^{\infty}
\frac{u+i}{u-\lambda_i+i}= \frac{\G{u+1+n}}{\prod_{i=1}^n (u-\lm_i +i)}, \quad n\ge
\ell(\lambda).
\end{align*}
The contours $C_z[0,\lm_1-x_2)$ and $C_w[0,n+x_1)$ are as shown in Figure
\ref{fig:contour1}. Both are counter-clockwise, encloses only the integers in the
respective half open intervals $[0,\lm_1-x_2)$ and $[0,n+x_1)$, and arranged such
that $w+z+x_2-x_1+1$ remains uniformly bounded away from 0 along the contours.
\end{thm}

\begin{rem}
When $t_1$ or $t_2$ equals 1, $K_{\lm}$ is to be understood in the limit as
$t_1$ or $t_2$ tends to 1. The contours $C_z[0,\lm_1-x_2)$ and $C_w[0,n-x_2)$
may also be replaced by unbounded contours $C_z[0,\infty)$ and $C_w[0,\infty)$
with bounded imaginary parts, respectively.
\end{rem}

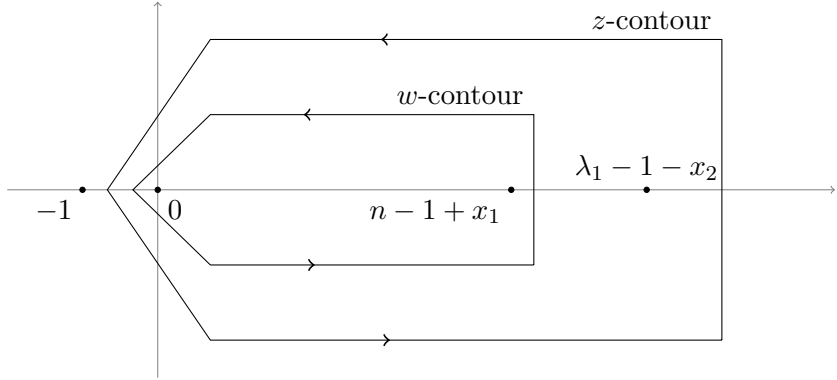
\begin{figure}
\begin{center}
\begin{tikzpicture}
\draw[help lines,->] (-2,0) -- (9,0) coordinate (xaxis); \draw[help lines,->]
(0,-2.5) -- (0,2.5) coordinate (yaxis);

\draw [decoration={markings, mark=at position 4.5cm with {\arrow[line
width=0.75pt]{<}}, mark=at position 14cm with {\arrow[line width=0.75pt]{<}}},
postaction={decorate}] (7.5,-2) -- (0.7,-2) -- (-0.67,0) -- (0.7,2) -- (7.5,2)
node[above left] {$z$-contour} -- (7.5,-2);

\draw [decoration={markings, mark=at position 3cm with {\arrow[line
width=0.75pt]{<}}, mark=at position 8.5cm with {\arrow[line width=0.75pt]{<}}},
postaction={decorate}] (5,-1) -- (0.7,-1) -- (-0.33,0) -- (0.7,1) -- (5,1)
node[above left] {$w$-contour} -- (5,-1);

\draw [black,fill=black] (0,0) circle (1pt) node[below right] {$0$}; \draw
[black,fill=black] (-1,0) circle (1pt) node[below left] {$-1$}; \draw
[black,fill=black] (4.7,0) circle (1pt) node[below left] {$n-1+x_1$}; \draw
[black,fill=black] (6.5,0) circle (1pt) node[above] {$\lambda_1-1-x_2$};
\end{tikzpicture}
\caption{\small{The contours in the statement of Theorem \ref{thm:pyt}.}}
\label{fig:contour1}
\end{center}
\end{figure}

The proof of Theorem \ref{thm:pyt} is through a limit transition in the
correlation kernel of \cite{Petrov} (see also \cite{DM}) for the uniformly random
Gelfand--Tsetlin patterns; the proof is in Section \ref{sec:poissonization}.
Such a limit transition can be viewed as a degeneration of the combinatorial
structures related to the representation theory of the unitary groups $U(N)$ to
those related to the symmetric groups $\mathfrak{S}_n$, see \cite{BO_YB} for a discussion.

Let us emphasize that the very same procedure can be used to identify a uniformly random SYT with
a point process of jumps, however, the resulting process will \emph{not} be
determinantal -- this is why we need to pass from SYTs to PYTs.

\subsection{From Poissonized tableaux to local statistics} \label{sec:introlst}

We close the introduction with an outline of the argument that takes us from Poissonized
tableaux to local statistics of sorting networks. By the nature of the Edelman-Greene
bijection, the swaps of a sorting network of $\mathfrak{S}_n$ near time 0 are determined by
the location of the largest entries of an SYT of shape $\sst$. These entries reside within
unit order distance of the \emph{edge} of $\sst$, which consist of the cells $(i,n-i)$
for $1 \leq i \leq n-1$. As a result, the first step to deriving
local statistics of random sorting networks is to derive the statistics of a uniformly
random PYT of shape $\sst$ near its edge.

Let $\T_{\sst}$ denote a uniformly random PYT of shape $\sst$. We are interested in the statistics
of the entries of $\T_{\sst}$ that lie within the following windows. A window is parameterized by a
center $\alpha \in (-1,1)$ (corresponding to the center $\frac{(1+\alpha)n}{2}$ of the swaps of a
sorting network), a length $L$, and an entry height $u$. The window then consist of entries
$\T_{\sst}(i,j)$ that satisfy $|i-(1+\alpha)n/2| \leq L$ and $\T_{\sst}(i,j) \geq 1 - \frac{u}{n}$.
In other words, roughly the largest $n u$ entries of $\T_{\sst}$ with row indices in the interval
$[(1+\alpha)n/2-L, (1+\alpha)n/2+L]$.

We study the statistics of $\T_{\sst}$ in a window in terms of its associated process
of jumps $\X_{\sst}$, rescaled accordingly. For each integer $n \geq 1$, let $c_n$ be an
integer having the same parity as $n$ and such that $|c_n - \alpha n| = O(1)$ as $n \to \infty$.
Consider the rescaled process of jumps
\begin{equation} \label{eqn:edgeprocess}
\X_{\alpha, n} = \left \{ (x,u) \in \Z \times \R_{\geq 0}: \Big (\, x + c_n, 1 - \frac{u}{n \sqrt{1-\alpha^2}} \,\Big) \in \X_{\sst} \right \}.
\end{equation}
Theorem \ref{thm:main} and Proposition \ref{prop:simplekernel} together imply that $\X_{\alpha,n}$
converges weakly to the point process $\X_{\rm{edge}}$ from Definition \ref{def:Xedge}.
It is the building block for the proof of Theorem \ref{thm:localsorting}.

In Section \ref{sec:localtab} we construct the local staircase shaped tableau $\T_{\rm{edge}}$
by using the local jump process $\X_{\rm{edge}}$. We prove in Theorem \ref{thm:tableaulocallimit}
that it provides the limiting statistics of $\T_{\sst}$ in local windows.
Using a de-Poissonization argument we conclude in Theorem \ref{corr:LST}
that uniformly random staircase shaped SYTs also converge within local windows
to $\T_{\rm{edge}}$. Although Poissonization is not important for the local limit,
it is important for the proof.

In Section \ref{sec:sorting} we prove Theorem \ref{thm:localsorting}.
First, we give a proof of Corollary \ref{cor:firstswap} in Section \ref{sec:firstswap}.
In Section \ref{sec:localEG} we define the local version of the Edelman-Greene algorithm
that maps $\T_{\rm{edge}}$ to $S_{\rm{local}}$. In Section \ref{sec:localsorting}
we complete the proof of Theorem \ref{thm:localsorting} and conclude with some
statistical properties of $S_{\rm{local}}$.

\subsection*{Acknowledgements} We would like to thank Balint Vir\'{a}g, who brought
our attention to the problem of identifying the local limit of sorting networks. We are
grateful to Alexei Borodin, Percy Deift, and Igor Krasovsky for valuable discussions
and references. We would also like to thank an anonymous referee for an exceptionally
careful reading of the paper and some helpful feedback.

V.~Gorin's research was partially supported by NSF grants DMS-1407562, DMS-1664619, by a Sloan
Research Fellowship, by The Foundation Sciences Math\'{e}matiques de Paris, and by NEC
Corporation Fund for Research in Computers and Communications.
M.~Rahman's research was partially supported by an NSERC PDF award.

\section{Preliminaries} \label{sec:prem}

This section presents basic facts about Young tableaux, Poissonization and determinantal point processes.
Some material from the Introduction is repeated for convenience.

\subsection{Gelfand-Tsetlin patterns} \label{sec:gtp}

A \emph{semi-standard Young tableau} of shape $\lm = (\lm_1, \ldots, \lm_M)$, where
$\lm_1 \geq \lm_2 \geq \cdots \geq \lm_M \geq 0$ are integers, is an insertion of
numbers from $\{1,\ldots, M\}$ into the cells of the YD $\lm$ such that the entries
weakly increase along each row and strictly increase along every column. It is
important to emphasize that while for the definitions of Young diagrams and standard
Young tableaux the value of $M$ is not important, here the object
substantially depends on $M$. Semi-standard Young tableaux (SSYTs) are in bijection
with interlacing particle systems, often known as \emph{Gelfand-Tsetlin patterns}
(or \emph{schemes}).

A Gelfand-Tsetlin pattern (GTP) with $M$ rows is a triangular array of non-negative
integers $[a(i,j)]$ with row $i$ containing $i$ entries $a(i,1), \ldots, a(i,i)$.
The array satisfies the following order and interlacing constraints.
\begin{equation*}
\text{Order \& Interlace}: \quad  a(i,j) \geq a(i-1,j) \geq a(i,j+1) \;\;\text{for every}\; i \;\text{and} \; j.
\end{equation*}
Let $a^{(i)} = (a(i,1), \ldots, a(i,i))$ denote the $i$-th row of the GTP. Each row
corresponds to a YD due to the order constraints. The interlacing conditions ensure
that $a^{(i-1)} \preceq a^{(i)}$, and in fact, $a^{(i)} \setminus a^{(i-1)}$ is a
horizontal strip which means that the cells in any row of $a^{(i)} \setminus
a^{(i-1)}$ are to the left of the cells in the previous row. Figure \ref{fig:GTP} provides an example.
\begin{figure}[htpb]
\begin{center}
$$ \setlength{\arraycolsep}{1pt}
    \begin{array}{ccccccc}
       9 &   & 5 &   & 3  &   &  2 \\
         & 8 &   & 5 &    & 2 &    \\
         &   & 6 &   &  3 &   &    \\
         &   &   & 5 &    &   &    \\
    \end{array}$$
\caption{\small{A Gelfand-Tsetlin pattern with 4 rows.}} \label{fig:GTP}
\end{center}
\end{figure}

The set of GTPs with a fixed top row $a^{(M)}$ is in bijection with the set of SSYTs
of shape $\lm = a^{(M)}$. Indeed, given a GTP $[a(i,j)]$ with top row $\lm$, such a
tableaux is obtained by inserting the value $i$ into the cells of $a^{(i)} \setminus
a^{(i-1)}$ for every $1 \leq i \leq M$ (set $a^{(0)} = \emptyset$). If $a^{(i)} \setminus a^{(i-1)}$
is empty then $i$ is not inserted. In the reverse direction, given a SSYT of shape $\lm$,
a GTP with top row $\lm$ is obtained by setting $a^{(i)}$ to be the YD consisting of the cells
of $\lm$ with entries $\leq i$ and removing trailing zero rows to ensure that $a^{(i)}$ has $i$ entries.

A GTP may also be represented as an interlacing particle system on $(\Z +\frac{1}{2})\times \Z$
as follows. Given a GTP $[a(i,j)]$ with $M$ rows, the particle system
$[\nu(i,j)]$ associated to it has $M$ rows of particles, with particles on row
$i$ being placed on the horizontal line $\{y=i\}$ of the plane, and the position of the $j$-th particle on row $i$ is
$$\Big (\nu(i,j),\, i \Big) = \Big (a(i,j) - j + \frac{1}{2}, \, i \Big)\;\; \text{for every}\;\; 1 \leq j \leq i.$$
The transformation $a(i,j) \to a(i,j)-j + \frac{1}{2}$ makes the order constraints strict and the interlacing constraints semi-strict:
$$\nu(i,j) \geq \nu(i-1,j) > \nu(i,j+1).$$
The \emph{jumps} of an interlacing particle system $\nu$ with $M$ rows is a set of
points in $\Z \times \{1, \ldots, M-1\}$, defined inductively from the top row to the bottom as follows.
Given two consecutive rows $[\nu(i,\cdot)]$ and $[\nu(i-1,\cdot)]$, the jumps on row $i$
consist of particles at the positions
$$(k,i-1) \in \Z^2 \;\;\text{for every integer}\; k \in [\nu(i-1,j), \nu(i,j)] \;\;\text{and every}\;\; 1 \leq j \leq i-1.$$
In other words, the jumps of row $i$ are placed on the horizontal line $\{y = i-1\}$
and fill out integers in the intervals $[\nu(i-1,j),\nu(i,j)]$ for every $1 \leq j \leq i-1$.
Note that $\nu$ may determined from its top row and set of jumps.

\subsection{Poissonized Young tableaux} \label{sec:poissonyt}

For a YD $\lm$, let $[0,1]^{\lm}$ denote the set of functions from
the cells of $\lm$ into $[0,1]$. Let $\mathrm{PYT}(\lm)$ denote the set of all
functions $T  \in [0,1]^{\lm}$ that satisfy the following tableau constraints.
\begin{align} \label{eqn:ytc}
(1)&\; T(i,j) \leq T(i,j+1) \;\text{for every}\; (i,j) \;\text{and}\; (i,j+1) \in \lm, \\
\nonumber (2)&\; T(i,j) \leq T(i+1,j) \;\text{for every}\; (i,j) \;\text{and}\; (i+1,j) \in \lm.
\end{align}
The \emph{Poissonized tableau} (PYT) of shape $\lm$ is an element of
$\mathrm{PYT}(\lm)$. The Poissonized staircase shaped tableau of size
$N = \binom{n}{2}$ is an element of $\mathrm{PYT}(\sst)$.

Let $\T_{\lm}$ denote a uniformly random element of $\mathrm{PYT}(\lm)$. Then
$\T_{\lm}$ is related to a uniformly random SYT of shape $\lm$ in the following way.
First, the entries of $\T_{\lm}$ are distinct with probability 1. Given
that, consider the random SYT $\Tb_{\lm}$ obtained by inserting $k$ into the cell
that contains the $k$-th smallest element of $\T_{\lm}$. Then $\Tb_{\lm}$ is a
uniformly random SYT of shape $\lm$. In the other direction, $\T_{\lm}$ can be
generated by first sampling $\Tb_{\lm}$, then independently sampling a uniformly
random $Y \in [0,1]^{\lm}$, and setting $\T_{\lm}(i,j)$ to be the
$\Tb_{\lm}(i,j)$-th smallest entry of $Y$.

Throughout the paper, $\T_{\lm}$ denotes a uniformly random element of $\mathrm{PYT}(\lm)$
and $\Tb_{\lm}$ denotes a uniformly random SYT of shape $\lm$.

\subsection{Jumps of Poissonized tableaux and local limit} \label{sec:limitsetup}

Any $T \in \mathrm{PYT}(\lm)$ can be represented as an interlacing particles system
with a fixed top row in the following manner. Consider $0 \leq t \leq 1$ and let
$$\mathrm{YD}^{(t)} = \{ (i,j) \in \sst : T(i,j) \leq t\}.$$
The tableau constraints (\ref{eqn:ytc}) ensure that $\rm{YD}^{(t)}$ is a YD for
every $t$. Recall that a YD can be made to have an infinite number of rows
by appending rows of size 0 after the last positive row. Encode $\lm$ as
particle configuration on $\Z + \frac{1}{2}$ by placing a particle at position
\begin{equation} \label{eqn:gt} \nu_j = \lm_j - j + \frac{1}{2} \;\;\text{for}\;\; j \geq 1.\end{equation}
This is an infinite particle configuration on $\Z + \frac{1}{2}$ such that $\nu_1 >
\nu_2 > \cdots$ and $\nu_j - \nu_{j+1} = 1$ for $j > \ell(\lm)$ (shown in Figure \ref{fig:YD_projection}).
Let $\nu^{(t)}$ be the particle configuration associated to $\rm{YD}^{(t)}$ via \eqref{eqn:gt} and let
$\nu = (\nu^{(t)}; 0 \leq t \leq 1)$ be the particle system on $(\Z+ \frac{1}{2})
\times [0,1]$ with a particle at position $(x,t)$ if and only if $x \in \nu^{(t)}$.

The particle system $\nu^{(t)}$ viewed in reverse time, i.e., from $t=1$ to
$t=0$, can be interpreted as an ensemble of non-intersecting and non-increasing
paths $p(i,u)$, for $1 \leq i \leq \ell(\lm)$. Let $p(i,u)$ be the $(\Z+
\frac{1}{2})$-valued path starting from $p(i,0) = \nu^{(0)}_i$ and decreasing an
integer unit at the times $1-T(i,\lm_i), 1- T(i,\lm_i-1), \ldots, 1-T(i,1)$. If some
of the entries are equal then $p(i,u)$ decreases by the number of consecutive equal
entries. The paths should be left continuous so that the jumps occur immediately
after the jump times. The path $p(i,u)$ decreases by $\lm_i$ units with final position
$p(i,1) = -i + \frac{1}{2}$. Due to the columns of $T$ being non-decreasing --
condition (2) of (\ref{eqn:ytc}) -- the paths are \emph{non-intersecting}: $p(i,u) >
p(i+1,u)$ for every $i$ and $u$. Figure \ref{fig:interlace} shows the paths associated
to a staircase shaped PYT.

The \emph{jumps} of $p(i,u)$ consists of points $(x,t) \in \Z \times [0,1]$ such that
\begin{enumerate}
\item $1-t$ is a discontinuity point of $p(i,u)$, i.e., $t$ equals some entry of $T$ on row $i$.
\item $x$ in an integer in the interval $\big [p(i,(1-t)+), \, p(i,1-t) \big]$, where $p(i,u+) = \lim_{s \downarrow u} p(i,s)$.
\end{enumerate}
The paths can be reconstructed from their jumps and initial positions.
The \emph{jumps} of $\nu$, and also of $T$, is the (possibly) multiset of $\Z \times [0,1]$ defined by
\begin{equation} \label{eqn:sytjumps}
X = \{ (x,t) : (x,t) \;\;\text{is a jump of some path}\;\; p(i,u)\}.
\end{equation}
$X$ may be a multiset because two adjacent paths may jump at the same time
by amounts that causes some of their jumps to coincide. The coinciding jumps has
be counted with multiplicity. However, if the entries of $T$ are distinct then $X$ is
a simple set. The tableau can be reconstructed from its jumps and the initial position
of the paths.

Let $\X_{\lm}$ denote the jumps of a uniformly random element $\T_{\lm}$
of $\mathrm{PYT}(\lm)$. $\X_{\lm}$ is simple almost surely since $\T_{\lm}$
has distinct entries almost surely. Theorem \ref{thm:pyt} asserts
that $\X_{\lm}$ is a determinantal point process on $\Z \times [0,1]$.

\subsection{Determinantal point processes} \label{sec:dpp}

We describe some basic notions about point processes; for a thorough
introduction see \cite{Bor, DV}. Let $S$ be a locally compact Polish space.
A \emph{discrete subset} $X$ of $S$ is a countable multiset of $S$ with no
accumulation points. By identifying $X$ with the measure $\sum_{x \in X} \delta_{x}$,
the space of discrete subsets can be given the topology of weak convergence
of Borel measures on $S$. This means that $X_n \to X_{\infty}$ if for every compact subset
$C \subset S$, $\limsup_n \, \# (C \cap X_n) \leq \#(C \cap X_{\infty})$, where cardinality is taken
with multiplicity.


A discrete set is simple if every point in it has multiplicity one. A point process on $S$
is a Borel-measurable random discrete set of $S$.
All point processes considered in this paper will be simple almost surely.

Throughout the paper we denote $\#_{\Z}$ to be counting measure on $\Z$ and $\Lb(A)$
to be Lebesgue measure on a measurable subset $A \subset \R$. Also, $\mu_1 \otimes \mu_2$
denotes the product of measures $\mu_1$ and $\mu_2$, and $\mu^{\otimes k}$
denotes the $k$-fold product of $\mu$.

A \emph{determinantal point process} $\X$ on $S$ is a simple point process
for which there is a correlation kernel $K : S \times S \to \R$, and a Radon
measure $\mu$ on $S$, called the reference measure, with the following property.
For every continuous $f : S^k \to \R$ of compact support,
\begin{equation} \label{eqn:dpp}
\mathbb{E} \,\Bigg [\sum_{\substack{(x_1,\ldots, x_k) \in \X^k \\ \,x_1, \ldots, x_k \; \text{distinct}}} f(x_1, \ldots, x_k)\Bigg] =
\int \limits_{S^k} \det \left [ K(x_i,x_j) \right] f(x_1,\ldots,x_k) \,\mu^{\otimes k}(dx_1,\ldots,dx_k).
\end{equation}
Expectations of the form given by the l.h.s.~of (\ref{eqn:dpp}) determine the law of $\X$ under mild
conditions on $K$ \cite{Lenard}. This will be the case in this paper as the correlation kernels we
consider will be continuous. If $S$ is discrete then it is customary to take the reference measure to be counting measure.
In this case $\X$ is determinantal if for every finite $A \subset S$,
\begin{equation*}
\pr{ A \subset \X} = \det \left [ K(x,y) \right ]_{x,y \in A}.
\end{equation*}

\begin{rem} \label{rem:dpp}
The correlation kernel of a determinantal point process is not unique. If $\X$ is a determinantal point process
with correlation kernel $K$ then $K$ may be replaced by $\frac{g(x)}{g(y)}K(x,y)$, for any non-vanishing
function $g$, without changing determinants on the r.h.s.~of (\ref{eqn:dpp}). Thus the new kernel
determines the same process. This observation will be used multiple times.
\end{rem}

The determinantal point processes that we consider will be on spaces of the form
$S = \Z \times \{1, \ldots, M\}$, or $S = \Z \times [0,1]$, or $S = \Z \times \R_{\geq 0}$,
with reference measures being, respectively, counting measure, $\#_{\Z} \otimes \Lb[0,1]$
and $\#_{\Z} \otimes \Lb(\R_{\geq 0})$. The following lemma records some facts
that will be used in deriving weak limits of determinantal point processes. We do not include the
proof as it is rather standard; see \cite{DV, Lenard}.

\begin{lemma} \label{lem:dpp}
I) Let $\X_M$ be a determinantal point process on $\Z \times \{1, \ldots, M\}$
with correlation kernel $K_M$. For $x_1,x_2 \in \Z$ and $0 \leq t_1,t_2 \leq 1$, let
$$k_M(x_1,t_1; x_2,t_2) = M K_M(x_1, \lceil Mt_1 \rceil; x_2, \lceil Mt_2 \rceil).$$
Suppose that $k_M \to k$ uniformly on compact subsets of $\Z \times (0,1)$.
Then the point process
$$\X_M^{\mathrm{scaled}} = \{ (x, t/M): (x,t) \in \X_M\}$$
restricted to $\Z \times (0,1)$ converges weakly to a determinantal point process $\X$ whose
reference measure is $\#_{\Z} \otimes \Lb(0,1)$ and whose correlation kernel is $k$.
\smallskip

\noindent II) Let $\X_n$ be a determinantal point process on $\Z \times (0,1)$ with reference
measure $\#_{\Z} \otimes \Lb(0,1)$ and correlation kernel $K_n$.
For $c_n \in \Z$ and $\beta > 0$, define a point process on $\Z \times \R_{> 0}$ by
$$\X_n^{\mathrm{scaled}} = \left \{ \big(x- c_n, \,\beta n(1-t) \big): (x,t) \in \X_n \right \}.$$
The correlation kernel of $\X_n^{\mathrm{scaled}}$ with reference measure $\#_{\Z} \otimes \Lb(\R_{> 0})$ is
$$k_n(x_1,u_1; x_2,u_2) = (\beta n)^{-1} K_n \left (x_1+c_n, 1- \frac{u_1}{\beta n}; \,x_2+c_n, 1- \frac{u_2}{\beta n} \right).$$
If $k_n \to k$ uniformly on compact subsets of $\Z \times \R_{> 0}$ then $\X_n^{\mathrm{scaled}}$
converges weakly to a determinantal point process $\X$ with reference measure $\#_{\Z} \otimes \Lb(\R_{> 0})$
and correlation kernel $k$.

Extend $\X$ to a point process on $\Z \times \R_{\geq 0}$ without additional points.
Then $\X_n^{\rm{scaled}}$ converges weakly to $\X$ on $\Z \times \R_{\geq 0}$ if
the points of $\X_n^{\rm{scaled}}$ do not accumulate at the boundary in the sense that for every $x \in \Z$,
$$ \lim_{\eps \to 0} \, \limsup_{n \to \infty} \, \pr{ \X_n^{\rm{scaled}} \cap (\{x\} \times [0,\eps]) \neq  \emptyset} = 0.$$
\end{lemma}

\section{Determinantal representation of Poissonized tableaux} \label{sec:poissonization}

\subsection{Determinantal representation of discrete interlacing particle systems} \label{sec:petrov}

In order to prove Theorem \ref{thm:pyt} we use a determinantal description of discrete
interlacing particle systems due to Petrov. This is the main tool behind the proof.

Let $\nu = (\nu(M,1) + \frac{1}{2} > \cdots > \nu(M,M) + \frac{1}{2})$ be a fixed particle
configuration on $\Z + \frac{1}{2}$. Here we abuse notation from Section \ref{sec:gtp}
to have the $\nu(M,j)$s be integers instead of half-integers. Let $\mathbb{P}_{\nu}$ be the
uniform measure on all interlacing particle systems or, equivalently, GTPs as described in
Section \ref{sec:gtp}, with fixed top row $\nu$. Let $\X_{\nu}$ be the point process of jumps
of an interlacing particle system sampled according to $\mathbb{P}_{\nu}$,
where the jumps are as described in Section \ref{sec:gtp}.

Petrov \cite[Theorem 5.1]{Petrov} proves that $\mathbb{P}_{\nu}$ is
a determinantal point process on  $(\Z+ \frac{1}{2}) \times \{1, \ldots, M\}$ with an explicit correlation kernel.
According to the notation there, particles live on $\Z$ but we have translated particle
systems by $1/2$ so that the jumps are integral. In particular, in the notation of \cite[Theorem 5.1]{Petrov},
one has $N=M$, $x ^{M,j} = \nu(M,j)$ and the variables $x_1, x_2$ take integer values.
In \cite[Section 6.1]{Petrov} it is explained that the point process of jumps,
$\X_{\nu}$, is also determinantal on $\Z \times \{1, \ldots, M-1\}$ and its correlation kernel
is given in terms of the correlation kernel of $\mathbb{P}_{\nu}$ in \cite[Theorem 6.1]{Petrov}
(up to the translation by $1/2$).

In particular, \cite[Theorem 6.1]{Petrov} proves that the correlation kernel of the jumps is
$$K_{\X_{\nu}}(x_1,m_1;x_2,m_2) = (-1)^{x_2-x_1 + m_2-m_1} \, K_{\mathbb{P}_{\nu}}(x_1-1,m_1+1;x_2,m_2),$$
where $K_{\mathbb{P}_{\nu}}$ is the kernel presented in \cite[Theorem 5.1]{Petrov}.
The discussion there is in terms of lozenge tilings of polygonal domains using three types of
lozenges as depicted in Figure \ref{fig:lozenge}. It is proved that the positions
of any of the three types of lozenges in such a uniformly random tiling is a determinantal point process.
The jumps of an interlacing particle system are given by the positions of the lozenges of the rightmost
type from Figure \ref{fig:lozenge}, where as the particles themselves are given by the positions of
lozenges of the leftmost type. Jumps occur when a lozenge of the leftmost type is glued along its
bottom diagonal to a lozenge of the rightmost type; see \cite[Figure 3]{Petrov} for
such a tiling. By Remark \ref{rem:dpp}, $(-1)^{x_2-x_1 + m_2-m_1} K_{\mathbb{P}_{\nu}}(x_1-1,m_1+1;x_2,m_2)$
defines the same determinantal point process as $K_{\mathbb{P}_{\nu}}(x_1-1,m_1+1;x_2,m_2)$,
and we will use the latter kernel.

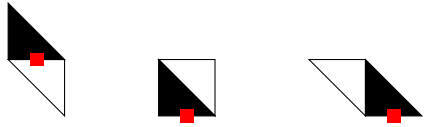
\begin{figure}[hbtp]
\begin{center}
\begin{tikzpicture}
\draw (0,0.75) -- (0,0)   -- (0.75,0) -- cycle;
\draw (0,0) -- (0.75,0)   -- (0.75,-0.75) -- cycle;
\fill[black] (0,0.75) -- (0,0)   -- (0.75,0) -- cycle;
\fill[red] (0.29,-0.08) rectangle (0.47,0.09);

\draw (2,0) -- (2.75,0) -- (2.75,-0.75) -- cycle;
\draw (2,0) -- (2,-0.75) -- (2.75,-0.75) -- cycle;
\fill[black] (2,0) -- (2,-0.75) -- (2.75,-0.75) -- cycle;
\fill[red] (2.29,-0.66) rectangle (2.47,-0.84);

\draw (4,0) -- (4.75,0) -- (4.75,-0.75) -- cycle;
\draw (4.75,0) -- (4.75,-0.75) -- (5.5,-0.75) -- cycle;
\fill[black] (4.75,0) -- (4.75,-0.75) -- (5.5,-0.75) -- cycle;
\fill[red] (5.04,-0.66) rectangle (5.22,-0.84);

\end{tikzpicture}
\caption{\small{The 3 types of lozenges used in tiling polygonal domains that correspond to interlacing particle systems.
The position of a lozenge is given by the midpoint of the horizontal side of the black triangular part (red square).
The positions have integer coordinates. Particles correspond to positions of the leftmost lozenges
but translated by $1/2$ in the $x$-coordinate (in our notation). Jumps correspond to positions of the rightmost lozenges.}}
\label{fig:lozenge}
\end{center}
\end{figure}

Some notation is needed in order to express the kernel for the point process of jumps.
For integers $a$ and $b$, let $C[a,b)$ denote a closed, counter-clockwise contour on $\C$
that encloses only the integers $a, a+1, \ldots, b-1$ if $a > b$, and empty otherwise.
Throughout the paper, all contours intersect the real line at points which have distance at least
$1/10$ from the integers. This ensures that the integrands of all contour integrals will be a
uniform distance away from their poles.

For $z \in \C$ and an integer $m \geq 1$, let
$$(z)_m = z(z+1) \cdots (z-m+1), \quad (z)_0 = 1.$$

\begin{thm}[{\cite[Theorem 5.1]{Petrov}}] \label{thm:petrov}
The process of jumps, $\X_{\nu}$, of a uniformly random interlacing particle system
with fixed top row $\nu = \big (\nu(M,1) + \frac{1}{2} > \nu(M,2) + \frac{1}{2} > \cdots > \nu(M,M) + \frac{1}{2} \big)$
is a determinantal point process on $\Z \times \{1,\ldots,M-1\}$ with correlation kernel $K$ as follows.
For $x_1, x_2 \in \Z$ and $1 \leq m_1, m_2 \leq M-1$,
\begin{align} \label{eqn:petrov}
&K(x_1,m_1;\,x_2,m_2) = - \ind{m_2 \leq m_1, \,x_2 < x_1} \, \frac{(x_1-x_2)_{m_1-m_2}}{(m_1-m_2)!} \;+\\
\nonumber & + \dint{[x_2,\nu(M,1)+1)}{[x_1-M,\nu(M,1)+1)} \Bigg [\frac{(z-x_2+1)_{M-m_2-1}}{(w-x_1+1)_{M-m_1}}
\cdot \frac{(M-m_1-1)!}{(M-m_2-1)!} \times \\
\nonumber & \quad \times \frac{1}{w-z} \cdot \prod_{j=1}^M \frac{w - \nu(M,j)}{z - \nu(M,j)} \Bigg].
\end{align}
The contour $C_w[-M,\nu(M,1)+1)$ contains $C_z[x_2,\nu(M,1)+1)$ without intersecting it.
\end{thm}

\subsection{Proof of Theorem \ref{thm:pyt}} \label{sec:pytlimit}

Let $\lm$ be a Young diagram with at most $n$ rows of positive length, that is,
$\ell(\lm) \leq n$. For $M \geq n$, consider semi-standard Young tableaux of shape
$$\lm_M = (\lm_1, \ldots, \lm_n, \underbrace{0, \ldots, 0}_{M-n\; \text{zeroes}}).$$
The effect of adding $M-n$ zero rows is to allow the entries in the non-zero rows of
$\lm$ to be between $1$ to $M$. The law of a uniformly random PYT of shape $\lm$
is the weak limit of a uniformly random semi-standard Young tableau of shape $\lm_M$, as $M \to \infty$,
after the entries are rescaled onto the interval $[0,1]$. Indeed, the law of a uniformly random PYT
of shape $\lm$ can be approximated by the uniform distribution on points $[T(i,j)] \in [0,1]^{\lm}$
that satisfy the tableau constraints (\ref{eqn:ytc}), with each $T(i,j) = k/M$ for some $1 \leq k \leq M$,
and the column constraints being strict.

The top row of particle systems associated to semi-standard Young tableaux of shape $\lm_M$
under the bijection described in Section \ref{sec:gtp} is
\begin{equation} \label{eqn:discretetoprow}
\nu^{(M)}_{\lm} = \left ((\lm_1-1) + \frac{1}{2}, \ldots, (\lm_n-n)+\frac{1}{2},
-(n+1) + \frac{1}{2},-(n+2) + \frac{1}{2},\ldots, -M + \frac{1}{2} \right ).
\end{equation}
Due to the approximation scheme above, and the bijection between semi-standard Young
tableaux and interlacing particle systems discussed in Section \ref{sec:poissonyt},
the process of jumps, $\X_{\lm}$, of a uniformly random PYT of shape $\lm$ is the weak
limit of the process of jumps, $\X_{\lm_M}$, of a uniformly random interlacing particle
system with top row $\nu^{(M)}$ after these jumps are rescaled onto $\Z \times \{1/(M-1), \ldots, 1\}$.
Thus, we derive a determinantal description of the rescaled jumps of
$\X_{\lm_M}$ in the large $M$ limit. Since $\X_{\lm}$ almost surely contains no jumps on the
boundary $\Z \times \{0, 1\}$, it suffices to derive the determinantal description with $\X_{\lm}$
restricted to $\Z \times (0,1)$.

Let $K_{M}$ denote the kernel from Theorem \ref{thm:petrov} for the process $\mathcal{X}_{\lm_{M}}$.
As explained in Remark \ref{rem:dpp}, the kernel $(M-1)^{x_2-x_1}K_M(x_1,m_1;x_2,m_2)$ determines
the same determinantal point process. By Lemma \ref{lem:dpp}, in order to prove the theorem it suffices to show that
$$(M-1)^{x_2-x_1+1} K_{M}(x_1, \lceil (M-1)t_1 \rceil; x_2, \lceil (M-1)t_2 \rceil) \; \longrightarrow\; K_{\lm}, $$
uniformly over compacts subsets of $x_1,x_2 \in \Z$ and $t_1, t_2 \in (0,1)$.

We begin by deforming the contours in the double contour integral from (\ref{eqn:petrov}) that defines $K_M$.
This will simplify the representation of $K_M$ for taking the large $M$ limit.
Deform the $w$-contour, $C_w[x_1-M,\nu_{\lm_M}(M,1)+1)$, by pushing it leftward past the $z$-contour,
$C_z[x_2,\nu_{\lm_M}(M,1)+1)$, so that it encloses the consecutive integers $\min \{x_1,x_2\} -1, \ldots, x_1-M$.
The deformation results in picking up residues at $w = z$ and also at the consecutive integers $w = x_1-1, \ldots, x_2$, if $x_2 < x_1$.

Let $J_M(w, x_1,m_1; z, x_2, m_2)$ denote the integrand of the double contour integral from (\ref{eqn:petrov})
but without the factor of $1/(w-z)$. Note that $\nu_{\lm_M}(M,1) +1= \lm_1$. Calculating the residues at $w=z$
and leaving the remaining residues as a contour integral provides the following representation of $K_M$.
\begin{align} \label{eqn:KM}
&K_M(x_1,m_1;x_2,m_2) = - \ind{m_2 \leq m_1, \,x_2 < x_1} \, \frac{(x_1-x_2)_{m_1-m_2}}{(m_1-m_2)!} \; + & (Ia)\\
\nonumber &+ \frac{1}{2 \pi \mathbf{i}} \oint \limits_{C_z[x_2, \lm_1)} dz\, \frac{(z-x_2+1)_{M-m_2-1}}{(z-x_1+1)_{M-m_1}}
\cdot \frac{(M-m_1-1)!}{(M-m_2-1)!} \;+ & (Ib)\\
\nonumber &+ \frac{1}{(2 \pi \mathbf{i})^2} \oint \limits_{C_z[x_2, \lm_1)} dz
\oint \limits_{C_w[x_1-M, \min \{x_1,x_2\}) \cup C_w[x_2, x_1)} dw\, \frac{J_M(w,x_1,m_1; z, x_2, m_2)}{w-z} & (II).
\end{align}

The following lemma simplifies $(Ia) + (Ib)$.

\begin{lemma} \label{lem:hypergeometric}
For $0 \leq m_1, m_2 \leq M-1$,
\begin{align*}
&\frac{1}{2 \pi i} \oint \limits_{C_z[x_2, \lm_1)} dz\, \frac{(z-x_2+1)_{M-m_2-1}}{(z-x_1+1)_{M-m_1}}
\cdot \frac{(M-m_1-1)!}{(M-m_2-1)!} \; - \ind{m_2 \leq m_1, \,x_2 < x_1} \, \frac{(x_1-x_2)_{m_1-m_2}}{(m_1-m_2)!} \\
& = \ind{m_2 > m_1, \, x_1 > x_2} \frac{(m_1-m_2+1)_{x_1-x_2-1}}{(x_1-x_2-1)!}.
\end{align*}
\end{lemma}

\begin{proof}
The integral $(Ib)$ is evaluated in \cite[Lemma 6.2]{Petrov} by summing over residues at $z = x_2, \ldots, x_1-1$
and evaluating the resulting sum in closed form via a hypergeometric identity. We have
$$(Ib) = \ind{x_1 > x_2} \frac{(m_1-m_2+1)_{x_1-x_2-1}}{(x_1-x_2-1)!}.$$
Therefore,
$$(Ia) + (Ib) = \ind{x_1 > x_2} \left [ \frac{(m_1-m_2+1)_{x_1-x_2-1}}{(x_1-x_2-1)!}
- \ind{m_2 \leq m_1} \, \frac{(x_1-x_2)_{m_1-m_2}}{(m_1-m_2)!} \right].$$
If $m_2 \leq m_1$ then the above is 0 because both terms in the difference are
equal to $\frac{(m_1-m_2+x_1-x_2-1)!}{(m_1-m_2)!(x_1-x_2-1)!}$. Hence, $(Ia) + (Ib)$
in non zero only if $x_1 > x_2$ and $m_2 > m_1$ and equals what is given in the statement of the lemma.
\end{proof}

Now consider the expression (II) from (\ref{eqn:KM}). Observe that the zeroes
of $(w-x_1+1)_{M-m_1}$ are at consecutive integers $x_1-1, x_1-2, \ldots, x_1-(M-m_1)$.
On the other hand, the polynomial $\prod_j (w-\nu_{\lm_M}(M,j))$ also has zeroes at $\nu_{\lm_M}(M,n+i) = -(n+i)$
for $1 \leq i \leq M-n$. Therefore, the only poles of $J_M$ in the $w$ variable are at the integers
$x_1-1, x_1-2, \ldots, -n$ so long as $x_1+m_1 \geq 0$. For fixed $x_1$ and $m_1 \geq \sqrt{M}$, say,
the condition $x_1+m_1 \geq 0$ is satisfied for all large $M$. Thus, the contour integral over $C_w[x_1-M, \min \{x_1,x_2\})$
may be shortened to $C_{w}[-n,\min \{x_1,x_2\})$ for all large $M$ if $x_1$ remains fixed and $m_1 \geq \sqrt{M}$.

Having done so, (II) becomes the following integral after changing variables $z \mapsto z+x_2$ and $w \mapsto -w + x_1-1$:
\begin{equation} \label{eqn:II}
(II) = \dint{[0,\lm_1-x_2)}{[0,n+x_1)} \frac{(z+1)_{M-m_2-1}\,(M-m_1-1)!
\prod \limits_{j=1}^M \frac{-w+x_1-1+\nu_{\lm_M}(M,j)}{z+x_2-\nu_{\lm_M}(M,j)}}
{(-w)_{M-m_1}\,(M-m_2-1)! \, (w+z+x_2-x_1+1)}. \end{equation}
By a slight abuse of notation, let $J_M$ henceforth denote the integrand of (\ref{eqn:II}) without the factor $1/(w+z+x_2-x_1+1)$.
The following lemma provides the asymptotic form of $J_M$.

\begin{lemma} \label{lem:Jasymptotic}
Fix integers $x_1$ and $x_2$. Suppose $0 < t_1,t_2 < 1$ and $w$ and $z$ remain bounded and have distance at
least $1/10$ from the integers. Then for $m_1 = \lceil t_1(M-1)\rceil$ and $m_2 = \lceil t_2(M-1)\rceil$,
\begin{equation*}
(M-1)^{x_1-x_2-1}J_M(w,x_1,m_1;z,x_2,m_2) =
(1-t_1)^{w} (1-t_2)^{z} \, \frac{\G{-w} G_{\lm}(z+x_2)}{ \G{z+1} G_{\lm}(x_1-1-w)} \; (1 + O(M^{-1})).
\end{equation*}
The big O term is uniform over $z,w$ so long as the stated assumptions hold and $t_1,t_2$
remain in compact subsets of $(0,1)$. The function $G_{\lm}$ is as stated in Theorem \ref{thm:pyt}.
\end{lemma}

\begin{proof}
We will use the following identity:
\begin{equation} \label{eqn:GammaI}
(y)_m = \frac{\G{y+m}}{\G{y}}, \;\; y \notin \{0,-1,-2, \ldots \}.
\end{equation}
We will also use Stirling's approximation of the Gamma function in the following form:
\begin{equation}
\label{eqn:GammaII}
\frac{\G{y+m}}{(m-1)!} = m^y \,(1+O(m^{-1})), \;\; m \geq 1.
\end{equation}
The big O term is uniform in $m$ so long as $y$ is bounded and bounded away from negative integers.
Using these two properties, if $ m_1 =\lceil t_1(M-1)\rceil$ and $m_2 =\lceil t_2(M-1)\rceil$ then
\begin{equation} \label{eqn:pochhammerratio}
\frac{(z+1)_{M-m_2-1}}{(-w)_{M-m_1}} \cdot \frac{(M-m_1-1)!}{(M-m_2-1)!} =
(1-t_1)^w (1-t_2)^z \frac{\G{-w}}{\G{z+1}} \, (M-1)^{z+w} \;(1+ O(M^{-1}))
\end{equation}

Now consider the term
\begin{align*}
\prod_{j=1}^M \big (w-\nu_{\lm_M}(M,j) \big) &= \prod_{j=1}^n (w-\lm_j+j) \cdot (w+n+1)_{M-n}\\
&= \prod_{j=1}^n (w - \lm_j+j) \cdot \frac{\G{w+M+1}}{\G{w+n+1}}.
\end{align*}
Applying (\ref{eqn:GammaII}) to $\G{w+M+1}$ and $\G{z+M+1}$ gives
\begin{equation*}
\prod_j \, \frac{w-\nu_{\lm_M}(M,j)}{z-\nu_{\lm_M}(M,j)} =
\frac{\G{z+n+1}}{\G{w+n+1}} \cdot \prod_{j=1}^n \frac{w-\lm_j+j}{z- \lm_j+j} \cdot M^{w-z} \, (1+O(M^{-1})).
\end{equation*}
Substituting in $z+x_2$ and $-w + x_1-1$ then gives
\begin{equation} \label{eqn:Gratio}
\prod_j \frac{-w+x_1-1-\nu_{\lm_M}(M,j)}{z+x_2-\nu_{\lm_M}(M,j)} = \frac{G_{\lm}(z+x_2)}{G_{\lm}(x_1-1-w)} \, (M-1)^{-(w+z) +x_1-x_2-1} (1+O(M^{-1})).
\end{equation}
Combining (\ref{eqn:pochhammerratio}) with (\ref{eqn:Gratio}) provides the desired conclusion of the lemma.
\end{proof}

We now prove that $(M-1)^{x_2-x_1+1}K_{M}$ converges to $K_{\lm}$.
Suppose $x_i$ are fixed and $m_i = \lceil t_i(M-1) \rceil$ for $i = 1,2$ and $t_i \in [\delta, 1-\delta]$ for some $\delta > 0$.
Recall that $K_M$ is given in (\ref{eqn:KM}). For all sufficiently large values of $M$,
Lemma \ref{lem:hypergeometric} and then the identity (\ref{eqn:GammaI}) followed by the estimate
(\ref{eqn:GammaII}) show that $(Ia) + (Ib)$ of (\ref{eqn:KM}) equals
\begin{align} \label{eqn:firstterm}
(Ia) + (Ib) & = \ind{m_2 > m_1, x_1 > x_2} \frac{(m_1-m_2+1)_{x_1-x_2-1}}{(x_1-x_2-1)!} \\
\nonumber & = \ind{t_2 > t_1, x_1 < x_2} \frac{(t_1-t_2)^{x_1-x_2-1}}{(x_1-x_2-1)!} \cdot (M-1)^{x_1-x_2-1} (1+ O(M^{-1})).
\end{align}

Now we consider (II) in the form given in (\ref{eqn:II}) following the change of variables.
Lemma \ref{lem:Jasymptotic} implies that as $M \to \infty$,
\begin{equation} \label{eqn:Jconverge}
(M-1)^{x_2-x_1+1}\,J_M(w,x_1,t_1; z, x_2,t_2) \to (1-t_1)^{w} (1-t_2)^{z} \frac{\G{-w} G_{\lm}(z+x_2)}{ \G{z+1} G_{\lm}(x_1-1-w)}.
\end{equation}
The convergence is uniform over compact subsets of $w$ and $z$ so long as $w$ and $z$ are
uniformly bounded away from the integers. For all large values of $M$ the contours of integration of
$(II)$ become free of $M$, namely, $z \in C_z[0,\lm_1-x_2)$ and $w \in C_w[0,n+x_1)$.
The contours may also be arranged such that they remain bounded away from the
integers and $|w+z+x_2-x_1+1| \geq 1/10$ throughout, say. This implies that as $M \to \infty$,
$(M-1)^{x_2-x_1+1} \cdot (II)$ converges to
\begin{align*}
\dint{[0,\lm_1-x_2)}{[0,n+x_1)} (1-t_1)^{w} (1-t_2)^{z} & \frac{\G{-w} G_{\lm}(z+x_2)}{ \G{z+1} G_{\lm}(x_1-1-w)} \, \times \\
  &\times \, \frac{1}{(w+z+x_2-x_1+1)}.
\end{align*}

We have thus concluded that $(M-1)^{x_2-x_1+1} K_M$ converges to the kernel $K_{\lm}$
given in Theorem \ref{thm:pyt}. Moreover, the estimates show that the converge
is uniform over compact subsets of $\Z \times (0,1)$. Indeed, so long at $t_1,t_2 \in [\delta,1-\delta]$
and $|x_1|, |x_2| \leq B$, the error term in the convergence is of order $O_{B,\delta}(M^{-1})$,
by Lemma \ref{lem:Jasymptotic}, because the double contours eventually become free of $M$
and the integrand converges uniformly over the contours. Part (I) of Lemma \ref{lem:dpp} now implies
that the rescaled process of jumps, $\X^{\mathrm{scaled}}_{\lm_M}$, converges weakly on $\Z \times (0,1)$
to a determinantal point process with kernel as given in Theorem \ref{thm:pyt}.
\qed

\section{Bulk local limit of the jumps of Poissonized staircase shaped tableaux} \label{sec:Xedge}

In this section we prove that the point process $X_{\alpha,n}$ from (\ref{eqn:edgeprocess})
converges weakly to the point process $\X_{\rm{edge}}$ from Definition \ref{def:Xedge}.
This is done in a two-step procedure. First, we prove in Theorem \ref{thm:main} that the limit of
$X_{\alpha, n}$ is a determinantal point process whose kernel is given in terms of a double contour integral.
Second, we identifying this kernel with the one from Definition \ref{def:Xedge} in Proposition \ref{prop:simplekernel}.

\begin{thm} \label{thm:main}
The point process $\X_{\alpha, n}$ from (\ref{eqn:edgeprocess}) converges weakly to a
limiting determinantal point process $\X_{\rm{edge}}$ on $\Z \times \R_{\geq 0}$.
The correlation kernel of $\X_{\rm{edge}}$ with respect to reference measure
$\#_{\Z} \otimes \Lb(\R_{\geq 0})$ is given as follows. For $u_1,u_2 \in \R_{\geq 0}$ and $x_1, x_2 \in \Z$,
\begin{align}
\nonumber & K_{\rm{edge}}(x_1,u_1;x_2,u_2) = \ind{u_2 < u_1, x_2 < x_1}\, \frac{(u_2-u_1)^{x_1-x_2-1}}{(x_1-x_2-1)!} \,+\\
 \nonumber & \dint{[0,\infty)}{[0,\infty)} \frac{\pi}{2} \cdot \frac{G(w; x_1,u_1) \, G(z; x_2,u_2)}{w+z +x_2-x_1+1}, \;\;\text{where}\\
\label{eqn:mainG} & G(z; x,u) = \frac{u^z}{\G{z+1} \sin(\frac{\pi}{2}(z+x_2))}.
\end{align}
The contours $C_z[0,\infty)$ and $C_w[0,\infty)$ are unbounded, contain the non-negative integers
but remain uniformly bounded away from them and are arranged such that $w+z+x_2-x_1+1$ remains
uniformly bounded away from 0. They may also be arranged such that their imaginary parts remain
bounded and $C_z$ contains $C_w$.

The value of $K_{\rm{edge}}$ when $u_1$ or $u_2$ equals 0 is to be understood in
the sense of the limit as $u_1$ or $u_2$ tends to 0.
\end{thm}

\begin{proof}
The proof proceeds in two steps, each verifying the conditions of part (II) of Lemma \ref{lem:dpp}.
First, we will show that the correlation kernel of $\X_{\alpha,n}$ converges to $K_{\rm{edge}}$
uniformly on compact subsets of $x_1,x_2 \in \Z$ and $u_1,u_2 \in \R_{> 0}$. Then we
will argue that points of $\X_{\alpha,n}$ do not accumulate on the boundary $\Z \times \{0\}$
as $n \to \infty$.

Let $\beta = \sqrt{1-\alpha^2}$. Part (II) of Lemma \ref{lem:dpp} and Theorem \ref{thm:pyt}
imply that the correlation kernel of $\X_{\alpha,n}$ with reference measure $\#_{\Z} \otimes \Lb(\R_{\geq 0})$ is
$$K_n(x_1,u_1;x_2,u_2) = (\beta n)^{-1} K_{\sst} \left(x_1 + c_n, 1- \frac{u_1}{\beta n};
x_1 + c_n, 1- \frac{u_1}{\beta n} \right ).$$
The kernel $(\beta n)^{x_1-x_2} K_n$ determines the same point process by Remark \ref{rem:dpp}.
Using part (II) of Lemma \ref{lem:dpp} it suffices to show that $(\beta n)^{x_1-x_2} K_n$
converges uniformly over compact subsets of $\Z \times \R_{> 0}$ to $K_{\rm{edge}}$ as $n \to \infty$
in order to deduce convergence of $\X_{\alpha,n}$ to $\X_{\rm{edge}}$ on $\Z \times \R_{>0}$.

Let $G_n$ denote the function $2^{n-1}G_{\sst}$, where $G_{\sst}$ is as in Theorem \ref{thm:pyt}. Then
\begin{equation} \label{eqn:Gsst}
G_n(z) = \frac{2^{n-1} \G{z+n+1}}{\prod_{j=1}^{n-1} (z-n + 2j)} = \frac{\G{z+n+1}}{(\frac{z-n}{2}+1)_{n-1}}
= \frac{\G{z+n+1} \G{\frac{z-n+2}{2}} }{\G{\frac{z+n}{2}}}.
\end{equation}
Substitute in $x_i +c_n$ for the variables $x_i$ and $1 - (\beta n)^{-1}u_i$ for the variables $t_i$ in $K_{\sst}$.
Then,
\begin{align} \label{eqn:thmmain0}
&K_n(x_1,u_1;x_2,u_2) = \ind{u_2 < u_1, x_2 < x_1}\,\frac{(u_2-u_1)^{x_1-x_2-1}}{(x_1-x_2-1)!} \,(\beta n)^{x_2-x_1}\,+\\
\nonumber & \dint{[0,n-1-c_n-x_2)}{[0,n-1+c_n+x_1)} \frac{\G{-w} G_n(z+x_2+c_n) \, u_1^w \, u_2^z \, (\beta n)^{-w-z-1}}
{\G{z+1} G_n(x_1-1+c_n-w) (w+z+x_2-x_1+1)}.
\end{align}
Using the formula for $G_n$ from (\ref{eqn:Gsst}) and applying the identity
\begin{align} \label{eqn:thmmain1}
\G{1-y}\G{y} &= \frac{\pi}{\sin(\pi y)}, \;y \notin \{0,-1,-2,\ldots\}, \;\text{to}\;\; y = \frac{n-c_n-z}{2}\;\; \text{gives} \\
\label{eqn:thmmain2}
G_n(z+c_n) &= \frac{\G{z+n+c_n+1}}{\G{\frac{n+c_n+z}{2}}\G{\frac{n-c_n-z}{2}}} \cdot \frac{\pi} {\sin \left (\frac{\pi}{2}(n-c_n-z) \right)}.
\end{align}

The estimate for $\G{z}$ from (\ref{eqn:GammaII}) along with the observation that
$n \pm c_n = (1 \pm \alpha)n  + O(1)  \to +\infty$ implies the following asymptotic bevaviour as $n \to \infty$.
The symbol $\sim$ denotes a multiplicative term $1+O(n^{-1})$ where the big O error is uniform over
$z$ in compact subsets of $\C \setminus \Z$.
\begin{align} \label{eqn:thmmain3}
\G{\frac{n \pm(c_n+z)}{2}} & \sim \G{\frac{n \pm c_n}{2}} \cdot \left( \frac{n \pm c_n}{2}\right)^{\pm \frac{z}{2}}, \\
\label{eqn:thmmain4} \G{n+c_n+z+1} & \sim \G{n+c_n} \cdot (n+c_n)^{z+1}.
\end{align}
Therefore,
\begin{align*}
\frac{\G{n+c_n+z+1}}{\G{\frac{n+c_n+z}{2}}\G{\frac{n-c_n-z}{2}}} \cdot
\frac{\G{\frac{n+c_n+w}{2}}\G{\frac{n-c_n-w}{2}}}{\G{n+c_n+w+1}} & \sim
\left(\frac{n+c_n}{n-c_n}\right)^{\frac{w-z}{2}} (1+\alpha)^{z-w} \; n^{z-w} \\
& \sim \left( \frac{1-\alpha}{1+\alpha} \right)^{\frac{z-w}{2}} (1+\alpha)^{z-w} \; n^{z-w} \\
& = (\beta n)^{z-w}.
\end{align*}
Substituting in $z+x_2$ for $z$ and $-w+x_1-1$ for $w$ in this estimate gives
\begin{equation} \label{eqn:Gnestimate}
\frac{G_n(z+x_2+c_n)}{G_n(-w+x_1-1+c_n)} \sim
 \frac{\sin \left (\frac{\pi}{2}(-w+x_1-1-n-c_n) \right)}{\sin \left (\frac{\pi}{2}(z+x_2-n-c_n) \right)}
 \left (\beta n \right)^{w+z + x_2-x_1+1}.
\end{equation}
The error in the above estimate vanishes as $n \to \infty$ so long as $w, z$ and the $x_i$ lie
in compact subsets of their respective domains. We also have that
$$\frac{\sin \left( \frac{\pi}{2}(n-c_n-w) \right)}{\sin \left(\frac{\pi}{2}(n-c_n-w)\right)} =
\frac{\sin \left(\frac{\pi}{2} w \right)}{\sin \left(\frac{\pi}{2} z \right)} \quad \text{if}\;\; n-c_n \;\;\text{is even}.$$
Since $n-c_n$ is assumed to be even we conclude from (\ref{eqn:Gnestimate}) that so long
$w$ and $z$ are bounded and remain uniformly bounded away from the integers then
\begin{equation} \label{eqn:thmmain5}
\lim_{n \to \infty} \; (\beta n)^{x_1-x_2+1} \, \frac{G_n(z+x_2+c_n) (\beta n)^{-w-z-1}}{G_n(x_1+c_n-1-w)} =
\frac{-\sin \left (\frac{\pi}{2} (w-x_1+1) \right)}{\sin \left (\frac{\pi}{2} (z+x_2)\right)}.
\end{equation}

The above displays the pointwise limit of the part of the integrand from (\ref{eqn:thmmain0}) that depends on $n$.
In order to interchange the pointwise limit with the contour integral we must show
that the integrand is bounded uniformly over $n$ by a function that is integrable over the
contours $z \in C_z[0,n-1-c_n-x_2)$ and $w \in C_w[0,n-1+c_n+x_1)$, also uniformly over $n$.
Then we may apply the dominated convergence theorem.

Towards this end suppose $z$ is such that (1) $|\Re(z)| \leq 2n$, (2) $|\Im(z)|$ is uniformly
bounded over $n$, say by 100, and (3) $z$ remains bounded away from $\Z$ by distance at least
$1/10$. In this case Stirling approximation to the Gamma function implies that modulus of
the ratio of the l.h.s.~of (\ref{eqn:thmmain3}) to its r.h.s.~is bounded above and below by exponential factors
of uniform rate in $|\Re(z)|$. The same holds for the ratio of the l.h.s.~of (\ref{eqn:thmmain4}) to its r.h.s.
That is, for some constant $C$,
\begin{align*}
e^{-C \,(|\Re(z)| + 1)} \leq \left | \dfrac{\G{\frac{n \pm(c_n+z)}{2}}}
{\G{\frac{n \pm c_n}{2}} \cdot \left( \frac{n \pm c_n}{2}\right)^{\pm \frac{z}{2}}} \right| \leq e^{C \,(|\Re(z)| + 1)}, \\
\\
e^{-C \,(|\Re(z)| + 1)} \leq \left | \dfrac{\G{n+c_n+z+1}}{\G{n+c_n} \cdot (n+c_n)^{z+1}} \right |\leq e^{C \,(|\Re(z)| + 1)}.
\end{align*}

Throughout the following $C$ denotes a constant that is free of $n$ but its value may change
from line to line. We combine the estimates above with the equation for $G_n(z+c_n)$ from (\ref{eqn:thmmain2})
and observe that there is a $C$ such that $1/C \leq |\sin(n-c_n-z)| \leq C$ due to the assumptions
on $z$. This in turn implies that there is a $C$ such that
\begin{equation*}
e^{-C \, (|\Re(z)| +1)} \leq \left | \frac{G_n(z+c_n)}{(\beta n)^z} \right| \leq e^{C \, (|\Re(z)| + 1)}.
\end{equation*}

The contours $C_z[0,n-1-c_n-x_2)$ and $C_w[0,n-1+c_n+x_1)$ can certainly be arranged
such that for fixed $x_1$ and $x_2$ the variables $z+x_2$ and $-w+x_1-1$ satisfy the
aforementioned assumptions (1)--(3) uniformly over $n$. Thus, we get the following
uniform estimate over $n$ with $z \in C_z[0,n-1-c_n-x_2)$ and $w \in C_w[0,n-1+c_n+x_1)$:
\begin{equation*}
\left | \frac{G_n(z+x_2 + c_n)}{G_n(-w+x_1-1+c_n)} \right| \leq (\beta n)^{\Re(z+w) +x_2-x_1+1}\, e^{C \,(|\Re(z)| + |\Re(w)| +1)}.
\end{equation*}
The contours may also be arranged such that $|w+z+x_2-x_1+1| \geq 1/10$, say. Then the modulus
of the integrand of the double contour integral from (\ref{eqn:thmmain0}) satisfies
\begin{align} \label{eqn:thmmain6}
& (\beta n)^{x_1-x_2} \left | \frac{\G{-w} G_n(z+x_2+c_n) \, u_1^w \, u_2^z \, (\beta n)^{-w-z-1}}
{\G{z+1} G_n(x_1-1+c_n-w) \,(w+z+x_2-x_1+1)} \right| \leq \\
\nonumber & \qquad \qquad \, \left |\frac{\G{-w}}{\G{z+1}} \right | \, |u_1^w| \, |u_2^z| \, e^{C \, (|\Re(z)| + |\Re(w)| +1)}.
\end{align}

Stirling's approximation implies that if $\Re(z) \geq 1/10$ and $|\Im(z)|$ remains bounded then
\begin{equation} \label{eqn:thmmain7}
\frac{|u^z|}{|\G{z+1}|} = e^{-\Re(z) \log \Re(z) + \Re(z)(\log u + O(1))}.
\end{equation}
Applying (\ref{eqn:thmmain1}) with $y = -w$ also gives $\G{-w} = - \pi [\G{w+1} \sin(\pi w)]^{-1}$.
Note that $|\sin (\pi w)|^{-1} \leq C$ so long as $w$ remains uniformly bounded away from the integers.
Combining this with (\ref{eqn:thmmain7}) shows that the r.h.s.~of (\ref{eqn:thmmain6}) is
integrable over unbounded double contours $C_z[0,\infty) \ni z$ and $C_w[0,\infty) \ni w$
as long as the contours are arranged such that $z,w$ remain uniformly bounded away from the
integers, have uniformly bounded imaginary parts, and $z+w+x_2-x_1+1$ remains uniformly bounded
away from 0. Thus, the limit (\ref{eqn:thmmain5}), upper bound (\ref{eqn:thmmain6}) and the
dominated convergence theorem implies that as $n \to \infty$
\begin{align*}
& (\beta n)^{x_1-x_2}\, K_n(x_1,u_1; x_2,u_2) \to
\ind{u_2 < u_1, x_2 < x_1}\,\frac{(u_2-u_1)^{x_1-x_2-1}}{(x_1-x_2-1)!} \,+ \\
& \dint{[0,\infty)}{[0,\infty)} u_1^w u_2^z \cdot \frac{-\G{-w} \sin \left (\frac{\pi}{2} (w-x_1+1) \right)}
{\G{z+1} \sin \left (\frac{\pi}{2} (z+x_2) \right)} \cdot \frac{1}{w+z+x_2-x_1+1}.
\end{align*}

Furthermore, our estimates show that the convergence is uniform over $x_i$ in compact subsets
of $\Z$ and $u_i$ in compact subsets of $\R_{> 0}$. (In fact, when some $u_i \to 0$ the integral
contributes only through residues at the origin. Lemma \ref{lem:tlimit} computes the limit as $u_i \to 0$.)
Comparing the limit integrand with the one presented in (\ref{eqn:mainG})
we observe that the proof of the kernel convergence will be complete once it is shown that
$$- \G{-w} \sin \left (\frac{\pi}{2} (w-x_1+1) \right) = \frac{(\pi/2)}{\G{w+1} \sin \left (\frac{\pi}{2} (w+x_1) \right) }.$$
From (\ref{eqn:thmmain1}), $-\G{-w} = \pi [\G{w+1} \sin(\pi w)]^{-1}$. Also, $\sin ( \frac{\pi}{2}(w-x+1)) = \cos(\frac{\pi}{2}(w-x))$.
Finally, double angle trigonometric formulae imply $\sin(\pi w) = 2 \sin( \frac{\pi}{2}(w+x)) \cos( \frac{\pi}{2}(w-x))$.
Substituting these equations into the l.h.s.~of the above verifies the equality with the r.h.s.

To complete the proof of convergence of $\X_{\alpha,n}$ to $\X_{\rm{edge}}$ on $\Z \times \R_{\geq 0}$
it is enough to show, using Lemma \ref{lem:dpp}, that for every $x \in \Z$,
\begin{equation} \label{eqn:thmmain8}
\lim_{\eps \to 0}\, \limsup_{n \to \infty}\, \E{ \# \, \X_{\alpha,n} \cap (\{x\}\times [0,\eps])} = 0.
\end{equation}
In other words, points do not accumulate at the boundary in the limit.
From relation (\ref{eqn:dpp}) for determinantal point processes and (\ref{eqn:thmmain0}) we get
\begin{align*}
&\E{ \# \, \X_{\alpha,n} \cap (\{x\} \times [0,\eps])} = \int_0^{\eps} K_n(x,t;x,t) \,dt \\
& = \dint{[0,n-1-c_n-x)}{[0,n-1+c_n+x)} \frac{\G{-w} G_n(z+x+c_n) \, \big (\int_0^{\eps} t^{w+z} \, dt \big) \, (\beta n)^{-w-z-1}}
{\G{z+1} G_n(x-1+c_n-w) (w+z+1)} \\
&= \dint{[0,n-1-c_n-x)}{[0,n-1+c_n+x)} \frac{\G{-w} G_n(z+x+c_n) \, \eps^{w+z+1} \, (\beta n)^{-w-z-1}}
{\G{z+1} G_n(x-1+c_n-w) (w+z+1)^2}.
\end{align*}

The quantity above is of the form $\eps I_{\eps,n}$. Arguing exactly as in the derivation of the limit kernel,
$I_{\eps,n} \to I_{\eps}$, where $I_{\eps}$ is given by the double contour integral in the definition of $K_{\rm{edge}}(x,\eps;x,\eps)$
from (\ref{eqn:mainG}) but with an additional factor of $w+z+1$ in the denominator of the integrand. Indeed,
$\eps I_{\eps} = \int_0^{\eps} K_{\rm{edge}}(x,t;x,t)\,dt$, which is the expected number of points of $\X_{\rm{edge}}$
on $\{x\} \times [0,\eps]$. The quantity $I_{\eps}$ remains uniformly bounded near $\eps = 0$ since, as
$\eps \to 0$, the  contribution to the integral that defines $I_{\eps}$ comes from the residues at $w,z = 0$
and these residues do not depend on $\eps$. (See Lemma \ref{lem:tlimit} where $\lim_{\eps \to 0} K_{\rm{edge}}(x,\eps;x,\eps)$
is derived analogously.) Consequently, $\eps I_{\eps} \to 0$ as $\eps \to 0$ and
thus the condition from (\ref{eqn:thmmain8}) holds.
\end{proof}

\subsection{Integral representation of the edge kernel} \label{sec:integralrep}
We begin with an auxiliary lemma.

\begin{lemma} \label{lem:tlimit}
Let $C_w [0,\infty)$ and $C_z [0,\infty)$ be contours as in the statement of
Theorem \ref{thm:main} and let $G(z;x,u)$ be as in (\ref{eqn:mainG}). For $t > 0$ let
$$I(t) = \dint{[0,\infty)}{[0,\infty)} \frac{\pi}{2} \cdot \frac{G(w; x_1,tu_1) \, G(z; x_2,tu_2))}{w+z +x_2-x_1+1}.$$
Then,
\begin{equation*}\lim_{t \to 0} I(t) = \begin{cases}
\frac{2}{\pi}\, \frac{\cos (\frac{\pi}{2}x_1) \cos(\frac{\pi}{2} x_2) }{x_2-x_1+1} & \text{if}\;\; x_2 \neq x_1-1\\
-\ind{x_1 \,\text{even}} & \text{if}\;\; x_2 = x_1-1.
\end{cases}
\end{equation*}
Moreover, $I$ is continuously differentiable on $\R_{> 0}$ and can be differentiated by
interchanging differentiation with the contour integration.
\end{lemma}

\begin{proof}
The integrand of $I$ is continuously differentiable in $t$. Observe that the contours
of integration contain no singularities of the integrand, and in fact, are
arranged to be a positive distance from all zeroes of $\sin(\frac{\pi}{2}(z+x_2))$ and
$\sin(\frac{\pi}{2}(w+x_1))$ in the denominator. The estimate for $|u^z| / |\G{z+1}|$
from (\ref{eqn:thmmain7}) shows that the derivative of the integrand in the variable $t$ is
absolutely integrable over the contours as long as $t$ lies in a compact subset of $\R_{\geq 0}$.
Consequently, by the dominated convergence theorem, $I$ is continuously differentiable
and the derivative may be interchanged with integration.

Let us now consider the limiting value of $I(t)$ as $t \to 0$. Decomposing
$C_w$ as $C_w[0,1) \cup C_w[1,\infty)$, and similarly for $C_z$, gives
\begin{equation*}
\oint \limits_{C_z} \oint \limits_{C_w} = \oint \limits_{C_z[0]} \oint \limits_{C_w[0]} +
\oint \limits_{C_z[0]} \oint \limits_{C_w[1,\infty]}  + \oint \limits_{C_z[1,\infty]} \oint \limits_{C_w[0]} +
\oint \limits_{C_z[1, \infty]} \oint \limits_{C_w[1,\infty]}.
\end{equation*}
These four contours may also be arranged such that $\Re(w+z) > 0$ unless both
$w \in C_w[0]$ and $z \in C_z[0]$. Recall that
$$G(z; x,u) = \frac{u^z}{\G{z+1} \sin (\frac{\pi}{2}(z+x))}.$$
Thus, the integrand of $I(t)$ converges to 0 as $t \to 0$ so long as $w \notin C_w[0]$ and $z \notin C_z[0]$.
So each of the double contour integrals above except for the first has a limit
value of 0 as $t \to 0$ (the limit operation may be interchanged with integration as argued above).
To complete the proof it suffices to calculate
\begin{equation} \label{eqn:tlimit}
\lim_{t \to 0}  \, \dint{[0]}{[0]} \frac{\pi}{2} \cdot \frac{G(w; x_1,tu_1) \, G(z; x_2,tu_2))}{w+z +x_2-x_1+1}.
\end{equation}

The integral above is evaluated via residues at $w=0$ and $z=0$.
If $x_2 \neq x_1-1$ then (\ref{eqn:tlimit}) equals
\begin{equation*}
\mathrm{Res}_{z=0} \left( \mathrm{Res}_{w=0} \left ( \frac{(\pi/2)\, G(w;x_1,tu_1)\,G(z;x_2,tu_2)}{w+z+x_2-x_1+1}\right) \right)
= \frac{2}{\pi} \frac{\cos (\frac{\pi}{2}x_2)\cos (\frac{\pi}{2}x_1)}{x_2-x_1+1}.
\end{equation*}
This is the limit value of $I(t)$ in the statement of the lemma for $x_2 \neq x_1-1$.

Now consider the limit (\ref{eqn:tlimit}) in the case $x_2 = x_1-1$.
As the contour $C_{w}[0]$ can be arranged to be contained inside $C_z[0]$, the integral in $w$ equals the
residue of the integrand at the only possible pole at $w = 0$. This equals
$$ \mathrm{Res}_{w=0} \left (\frac{(\pi/2) \, G(w;x_1,tu_1)\,G(z;x_2,tu_2)}{w+z} \right) =
\frac{\cos(\frac{\pi}{2} x_1) (tu_2)^z}{z\G{z+1} \sin(\frac{\pi}{2}(z+x_2))}.$$
If $x_1$ is odd then the above equals 0. Otherwise, $\cos( (\pi/2)x_1) = (-1)^{x_1/2}$ and
the integral of the above over $C_z[0]$ is given by its residue at the pole $z=0$
(note that $\sin(\frac{\pi}{2} x_2) \neq 0$ since $x_2 = x_1-1$ is odd). The residue equals
$$\mathrm{Res}_{z=0} \left (\frac{\cos(\frac{\pi}{2} x_1) (tu_2)^z}{z\G{z+1} \sin(\frac{\pi}{2}(z+x_2))} \right) =
\frac{\cos(\frac{\pi}{2}x_1)}{\sin(\frac{\pi}{2}x_2)} = -1.$$
Thus, if $x_2 = x_2-1$ then (\ref{eqn:tlimit}) equals $- \ind{x_1 \,\text{even}}$ and this completes the proof.
\end{proof}

\begin{prop} \label{prop:simplekernel}
The kernel $K_{\rm{edge}}$ has the following form.
\begin{align*}
&K_{\rm{edge}}(x_1,u_1;x_2,u_2) =
\begin{cases}
\displaystyle \frac{2}{\pi}  \int \limits_1^0 t^{x_2-x_1}
\cos \left( tu_1 + \frac{\pi}{2} x_1\right) \cos \left( tu_2 + \frac{\pi}{2}x_2 \right )\,dt, & \text{if}\;\; x_2 \geq x_1; \\
\displaystyle - \frac{2}{\pi} \int \limits_1^{\infty} t^{x_2-x_1}
\cos \left( tu_1 + \frac{\pi}{2} x_1\right) \cos \left( tu_2 + \frac{\pi}{2}x_2 \right )\,dt, & \text{if}\;\; x_2 < x_1.
\end{cases}
\end{align*}
\end{prop}

\begin{proof}
For $t > 0$ let
$$f(t) = K_{\rm{edge}} (x_1, tu_1; x_2, tu_2) = t^{x_1-x_2-1} \ind{u_1 > u_2, x_1 > x_2}
\frac{(u_2-u_1)^{x_1-x_2-1}}{(x_1-x_2-1)!} + I(t),$$
where $I(t)$ is as defined in Lemma \ref{lem:tlimit}. By Lemma \ref{lem:tlimit}, $f$
is continuous differentiable on $\R_{>0}$ and the function $t^{x_2-x_1+1}f(t)$ may
be differentiated by interchanging differentiation with integration.
Differentiating $t^{x_2-x_1+1}f(t)$ and clearing common powers of $t$ gives
\begin{align*}
(x_2-x_1+1)f + tf' &= \dint{[0,\infty)}{[0,\infty)} \frac{\pi}{2} \cdot G(w; x_1,tu_1) \, G(z; x_2,tu_2)) \\
&= \frac{\pi}{2} \cdot \frac{1}{2 \pi \mathbf{i}} \oint \limits_{C_w[0,\infty)} dw\, G(w; x_1,tu_1) \cdot
\frac{1}{2 \pi \mathbf{i}} \oint \limits_{C_z[0,\infty)} dz\, G(z; x_2,tu_2)).
\end{align*}

The contour integrals can be evaluated by summing over residues of $G$. Inside the contour
$C_z[0,\infty]$, the function $G(z,x,u)$ has simple poles at integers $z$ such that $\sin (\pi(z+x)/2) = 0$.
In other words, $z$ has to have the same parity as $x$. Let $\bar{x} = \ind{ x \,\text{odd}}$.
The residues of the integral come from integers $z = 2k + \bar{x}$ for $k \geq 0$.
Since $\mathrm{Res}_{y=2k}(\sin(\pi y/2)^{-1}) = (2/\pi)(-1)^k$, summing over the residues gives
\begin{align*}
\frac{1}{2 \pi \mathbf{i}} \oint \limits_{C_z[0,\infty)} dw\, \frac{(tu)^z}{\G{z+1} \sin (\frac{\pi}{2}(z+x))} &= \frac{2(-1)^{\frac{x+\bar{x}}{2}}}{\pi}
\sum_{k \geq 0}  (-1)^k \frac{(tu)^{2k + \bar{x}}}{(2k + \bar{x})!} \\
& = \frac{2(-1)^{\frac{x+\bar{x}}{2}}}{\pi} \cos \left( ut -\frac{\pi}{2} \bar{x} \right)\\
& = \frac{2}{\pi} \cos \left( ut +\frac{\pi}{2} x\right).
\end{align*}
Consequently,
\begin{equation} \label{eqn:prop2}
(x_2-x_1+1)f(t) + tf'(t) = \frac{2}{\pi} \cos \left( u_1t +\frac{\pi}{2}x_1 \right) \cos \left ( u_2t + \frac{\pi}{2}x_2 \right).
\end{equation}
Multiplying (\ref{eqn:prop2}) by $t^{x_2-x_1}$ then implies that
\begin{equation} \label{eqn:prop3}
[t^{x_2-x_1+1}f]' = \frac{2}{\pi} \,t^{x_2-x_1} \cos \left( u_1t +\frac{\pi}{2}x_1\right) \cos \left ( u_2t + \frac{\pi}{2}x_2 \right).
\end{equation}

For $x_2 \geq x_1$, the r.h.s.~of (\ref{eqn:prop3}) is integrable over $t$ in $[0,1]$.
Moreover, $t^{x_2-x_1+1}f(t) \to 0$ as $t \to 0$ because $\lim_{t \to 0} f(t) = \lim_{t \to 0} I(t)$,
and the latter limit is finite whereas $t^{x_2-x_1+1} \to 0$ as $t \to 0$. Therefore, (\ref{eqn:prop3}) implies
$$ f(1) = \frac{2}{\pi} \int_{0}^1
t^{x_2-x_1} \cos \left( u_1t +\frac{\pi}{2}x_1 \right) \cos \left ( u_2t + \frac{\pi}{2}x_2 \right) \,dt.$$

Next, consider the case $x_2 < x_1-1$. Now the relation from (\ref{eqn:prop3}) should be integrated
from 1 to $\infty$, which is convergent since $x_2-x_1 \leq -2$. The formula follows
so long as $\lim_{t \to \infty} t^{x_2-x_1+1}f(t) = 0$. Rather than derive this limit we take a
slightly indirect approach by considering the limit of $f(t)$ near $t=0$.
For $t > 0$ define
\begin{align} \label{eqn:prop4}
g(t) &= - \frac{2}{\pi} \int_1^{\infty} s^{x_2-x_1} \cos ( u_1st +\frac{\pi}{2}x_1) \cos ( u_2st + \frac{\pi}{2}x_2) \,ds \\
\nonumber &= - \frac{2}{\pi}\, t^{x_1-x_2-1} \int_t^{\infty} s^{x_2-x_1} \cos ( u_1s +\frac{\pi}{2}x_1) \cos ( u_2s + \frac{\pi}{2}x_2) \,ds.
\end{align}
Upon differentiating $g$ it follows readily that $g$ satisfies the same differential equation
as $f$ displayed in (\ref{eqn:prop2}). Therefore, $f(t) = g(t) + C$ for some constant $C$.
In order to identify $C$ as zero it suffices to show that $\lim_{t \to 0} f(t)-g(t) = 0$.
Since $t^{x_1-x_2-1} \to 0$ as $t \to 0$, due to $x_2 < x_1-1$, both $f(t)$ and $I(t)$
have the same limit as $t \to 0$. Thus, utilizing Lemma \ref{lem:tlimit}, showing $C=0$ amounts to proving
$$\lim_{t \to 0} \,g(t) = \frac{2}{\pi}\, \frac{\cos (\frac{\pi}{2}x_1) \cos(\frac{\pi}{2} x_2) }{x_2-x_1+1}.$$

The limit of $g(t)$ can be found using L'H\^{o}spital's rule, which shows that
\begin{align*}
\lim_{t \to 0} \,g(t) &= \lim_{t \to 0}
\frac{- \frac{2}{\pi} \int_t^{\infty} s^{x_2-x_1} \cos ( u_1s +\frac{\pi}{2}x_1)
\cos ( u_2s + \frac{\pi}{2}x_2)\,ds}{t^{x_2-x_1+1}} \\
&= \lim_{t \to 0} \; \frac{\frac{2}{\pi}\, t^{x_2-x_1} \cos ( u_1t +\frac{\pi}{2}x_1)
\cos ( u_2t + \frac{\pi}{2}x_2)}{(x_2-x_1+1) t^{x_2-x_1}}\\
& = \frac{2}{\pi}\, \frac{\cos(\frac{\pi}{2}x_1) \cos(\frac{\pi}{2}x_2) }{x_2-x_1+1}.
\end{align*}
We conclude that $g(t) = f(t)$ for $t \in \R_{> 0}$, and in particular that $f(1) = g(1)$, as required.

Finally, consider the case $x_2 = x_1-1$. The r.h.s.~of (\ref{eqn:prop3}) is continuous for
$t$ in $[0,1]$ because one of $\cos ( u_1t +\frac{\pi}{2}x_1)$ or $\cos ( u_2t + \frac{\pi}{2}x_2)$
has a zero at $t=0$ depending upon the parity of $x_1$. From Lemma \ref{lem:tlimit},
$\lim_{t \to 0} f(t) = \ind{u_1 > u_2} - \ind{ x_1 \,\text{even}}$. Therefore, (\ref{eqn:prop3}) implies that
\begin{equation} \label{eqn:1diff}
f(1) = \ind{u_1 > u_2} - \ind{ x_1 \,\text{even}} +
\frac{2}{\pi} \int_0^1 t^{-1}\cos ( u_1t +\frac{\pi}{2}x_1) \cos ( u_2t + \frac{\pi}{2}x_2)\,dt.\end{equation}

We now express (\ref{eqn:1diff}) as an integral over $t \in [1,\infty)$ as given in the proposition.
First, note $K_{\rm{edge}}$ may be modified on the measure zero set
consisting of $(x_1,u_1;x_2,u_2)$ such that $u_1=u_2$ without changing determinants
in (\ref{eqn:dpp}), and thus, this does not affect the law of $\X_{\rm{edge}}$. We will modify
the kernel on this zero set after the following calculations to get the form given in the proposition.

Observe that $\cos(u_2t + \frac{\pi}{2} x_2) = \sin(u_2t + \frac{\pi}{2}x_1)$ if $x_2=x_1-1$.
Using trigonometric formulae the integrand of (\ref{eqn:1diff}) becomes
\begin{align} \label{eqn:1diffintegrand}
&\frac{2 \cos ( u_1t +\frac{\pi}{2}x_1) \sin ( u_2t + \frac{\pi}{2}x_1)}{\pi t} = \begin{cases}
\frac{\sin((u_1+u_2)t) + \sin((u_2-u_1)t)}{\pi t}, & x_1 \;\text{even} \\
 \frac{-\sin((u_1+u_2)t) + \sin((u_2-u_1)t)}{\pi t}, & x_1 \;\text{odd}.
\end{cases}
\end{align}
Using the fact that $\int_0^{\infty} \frac{\sin t}{t} = \frac{\pi}{2}$, we get that for $a \in \R$,
\begin{equation} \label{eqn:si}
\int_0^1 \frac{\sin (at)}{\pi t}\,dt = \mathrm{sgn}(a) \int_0^{|a|} \frac{\sin t}{\pi t} \,dt=
\frac{\mathrm{sgn}(a)}{2} - \int_1^{\infty} \frac{\sin(at)}{\pi t}\,dt.
\end{equation}

Using (\ref{eqn:si}) and the representation of the integrand in (\ref{eqn:1diffintegrand}) we infer that if $x_1$ is even then
\begin{align*}
& \ind{u_1 > u_2} - \ind{ x_1 \,\text{even}} + \frac{2}{\pi} \int_0^1 t^{-1}\cos ( u_1t +\frac{\pi}{2}x_1) \sin ( u_2t + \frac{\pi}{2}x_1)\,dt = \\
& \ind{u_1 > u_2} - 1 + \frac{\mathrm{sgn}(u_1+u_2) + \mathrm{sgn}(u_2-u_1)}{2}
- \int_1^{\infty} \frac{\sin((u_1+u_2)t) + \sin((u_2-u_1)t)}{\pi t} \,dt = \\
& - \ind{u_1=u_2} (\frac{1 + \ind{u_1=0}}{2}) - \frac{2}{\pi} \int_1^{\infty} t^{-1} \cos ( u_1t +\frac{\pi}{2}x_1) \sin ( u_2t + \frac{\pi}{2}x_1)\,dt.
\end{align*}
This shows that (\ref{eqn:1diff}) equals the expression given in the statement of the proposition for
$x_2 = x_1-1$ and $x_1$ even except for the additive term $- \ind{u_1=u_2} (\frac{1 + \ind{u_1=0}}{2})$.
By modifying $K_{\rm{edge}}$ on the zero set $\{u_1=u_2, x_2=x_1-1, x_1 \; \text{even}\}$ we may ignore this term.

For $x_1$ being odd we argue in the same manner as above to infer that (\ref{eqn:1diff}) equals
\begin{align*}
& \ind{u_1 > u_2} - \ind{ x_1 \,\text{even}} + \frac{2}{\pi} \int_0^1 t^{-1}\cos ( u_1t +\frac{\pi}{2}x_1) \sin ( u_2t + \frac{\pi}{2}x_1)\,dt = \\
& \ind{u_1 > u_2}  - \frac{\mathrm{sgn}(u_1+u_2) + \mathrm{sgn}(u_1-u_2)}{2} +
\int_1^{\infty} \frac{\sin((u_1+u_2)t) + \sin((u_1-u_2)t)}{\pi t} \,dt= \\
& = -\frac{\ind{u_1=u_2 > 0}}{2} - \frac{2}{\pi} \int_1^{\infty} t^{-1} \cos ( u_1t +\frac{\pi}{2}x_1) \sin ( u_2t + \frac{\pi}{2}x_1)\,dt.
\end{align*}
Once again, we modify $K_{\rm{edge}}$ on the zero set $\{u_1 = u_2, x_2=x_1-1, x_1 \;\text{odd}\}$
to ignore the additive term $-\frac{1}{2} \ind{u_1=u_2 > 0}$ and get the expression of the kernel
given in the proposition.
\end{proof}

\subsection{Statistical properties of $\X_{\rm{edge}}$} \label{sec:edgeproperties}

This section derives certain properties of $\X_{\rm{edge}}$, namely, Proposition \ref{prop:Kedge},
Proposition \ref{prop:tail}, Lemma \ref{lem:intensity} and Lemma \ref{lem:EGhelper},
that will be used to derive the local limit of staircase shaped tableaux and of sorting networks.

\begin{prop} \label{prop:Kedge}
The process $\X_{\rm{edge}}$ has the following statistical properties.

I) Translation and reflection invariance: For any integer $h$ the translated process
$$ \X_{\rm{edge}} + (2h,0) = \{ (x+2h,u) : (x,u) \in \X_{\rm{edge}} \}$$
and the reflected process
$$(-1,1) * \X_{\rm{edge}} = \{(-x,u): (x,u) \in \X_{\rm{edge}} \}$$
have the same law as $\X_{\rm{edge}}$.

II) One dimensional marginals: For any $x \in \Z$ and $u_1, u_2 \in \R_{\geq 0}$,
\begin{equation} \label{eqn:diagonal}
K_{\mathrm{edge}}(x,u_1; x,u_2) = \frac{\sin(u_1-u_2)}{\pi \, (u_1-u_2)} + (-1)^x \, \frac{\sin(u_1+u_2)}{\pi \, (u_1+u_2)}.
\end{equation}
Therefore, $\X_{\rm{edge}} \cap (\{x\}\times \R_{\geq 0})$ is a determinantal point process
with reference measure $\Lb(\R_{\geq 0})$ and correlation kernel (\ref{eqn:diagonal}).
\end{prop}

\begin{proof}
Part I) From Lemma \ref{lem:dpp} the correlation kernel of $\X_{\rm{edge}} + (2h,0)$
equals $K_{\mathrm{edge}}(x_1-2h,u_1; x_2-2h,u_2)$. The integral representation of
$K_{\rm{edge}}$ in Proposition \ref{prop:simplekernel} implies that
$$K_{\mathrm{edge}}(x_1-2h,u_1; x_2-2h,u_2) = K_{\mathrm{edge}}(x_1,u_1; x_2,u_2)$$
upon observing that $\cos (x + \pi h) = (-1)^h \cos(x)$, which implies that the integrands do not
change after the kernel is transformed. Consequently, the translated point process has the same law
as the original. Similarly, the correlation kernel for the reflected process is
$K_{\mathrm{edge}}(-x_1,u_1;-x_2;u_2) = (-1)^{x_1-x_2} K_{\mathrm{edge}}(x_2,u_2;x_1,u_1)$.
The latter kernel defines the same determinantal point process as $\X_{\rm{edge}}$ in law.
\vskip 0.1in
Part II) Proposition \ref{prop:simplekernel} gives that
\begin{align*}
K_{\mathrm{edge}}(x,u_1;x,u_2) &= \frac{2}{\pi} \int_0^1 \cos(tu_1 + \frac{\pi}{2}x) \cos(tu_2 + \frac{\pi}{2}x)\,dt \\
&= \frac{1}{\pi}  \left(\frac{\sin(tu_1-tu_2)}{u_1-u_2} + \frac{\sin(tu_1 + tu_2 + \pi x)}{u_1+u_2} \right )
\Bigg |_{t=0}^{t=1}\\
&= \frac{1}{\pi} \left ( \frac{\sin(u_1-u_2)}{u_1-u_2} + (-1)^x \frac{\sin(u_1+u_2)}{u_1+u_2}\right).
\end{align*}
The fact hat $\X_{\rm{edge}} \cap (\{x\}\times \R_{\geq 0})$ is determinantal with kernel as stipulated
follows from the relation (\ref{eqn:dpp}) for determinantal point processes.
\end{proof}

\begin{lemma} \label{lem:correlation}
There is a universal constant $C$ such that for $x_1,x_2 \in \Z$ and $u_1, u_2 \in \R_{\geq 0}$,
\begin{equation} \label{eqn:correlation1}
| K_{\rm{edge}}(x_1,u_1;x_2,u_2) | \leq \frac{C}{\max \{ |x_1-x_2|, |u_1-u_2| \} + 1}.
\end{equation}
\end{lemma}

\begin{proof}
Throughout this argument $C$ denotes a universal constant whose value may change from line to line.
We begin with the case $x_2 \neq x_1-1$. From the integral representation of $K_{\rm{edge}}$ we see
that if $x_2 \geq x_1$ then
\begin{align} \label{eqn:Kxdecay1}
|K_{\rm{edge}}(x_1,u_1;x_2,u_2)| &= \frac{2}{\pi} \left | \int_0^1 t^{x_2-x_1}
\cos(tu_1+\frac{\pi}{2}x_1)\cos(tu_2+\frac{\pi}{2}x_2) \, dt \right | \\
\nonumber & \leq \frac{2}{\pi} \int_0^1 t^{x_2-x_1} \,dt \leq  \frac{C}{|x_2-x_1|+1}.
\end{align}
Similarly, if $x_2 < x_1$ then $x_2 \leq x_1-2$ and
\begin{align} \label{eqn:Kxdecay2}
|K_{\rm{edge}}(x_1,u_1;x_2,u_2)| &= \frac{2}{\pi} \left | \int_1^{\infty} t^{x_2-x_1}
\cos(tu_1+ \frac{\pi}{2}x_1)\cos(tu_2+\frac{\pi}{2}x_2) \, dt \right | \\
\nonumber & \leq \frac{C}{|x_2-x_1|+1}.
\end{align}
Combining these bounds we deduce that if $x_2 \neq x_1 -1$ then
\begin{equation} \label{eqn:Kxdecay}
| K_{\rm{edge}}(x_1,u_1;x_2,u_2)| \leq \frac{C}{|x_1-x_2|+1}.
\end{equation}

Now we consider decay in the $u$-variables, assuming that $x_2 \neq x_1-1$.
Define $v(t)$ as
\begin{equation*} v(t) = \begin{cases}
\displaystyle \frac{\sin (tu_1-tu_2 + \frac{\pi}{2} (x_1-x_2))}{\pi(u_1-u_2)} +
\frac{\sin (tu_1+tu_2 + \frac{\pi}{2} (x_1+x_2))}{\pi(u_1+u_2)} & \text{if}\; u_1 \neq u_2 \\ \\
\displaystyle \frac{t \cos(\frac{\pi}{2}(x_1-x_2))}{\pi} +
\frac{\sin (tu_1+tu_2 + \frac{\pi}{2} (x_1+x_2))}{\pi(u_1+u_2)} & \text{if}\; u_1 = u_2.
\end{cases}\end{equation*}
Then $v'(t) = \frac{2}{\pi} \cos(tu_1 + \frac{\pi}{2}x_1)\cos(tu_2 + \frac{\pi}{2}x_2)$.
Using that $|\sin(y)/y| \leq 1$, and the formula for $v(t)$, we observe that there is a $C$ such that
\begin{equation} \label{eqn:vdecay}
|v(t)| \leq \frac{C}{|u_1-u_2|+1} \;\;\text{if}\;\; |u_1-u_2| \geq 1.\end{equation}

Applying integration by parts to the integral form of $K_{\mathrm{edge}}$ gives, for $x_2 \geq x_1$,
\begin{equation*}
K_{\mathrm{edge}}(x_1,u_1;x_2,u_2) = v(1) -v(0)\ind{x_1=x_2} - \int_0^1 (x_2-x_1)\, t^{x_2-x_1-1} \,v(t) \, dt.
\end{equation*}
Now the triangle inequality and (\ref{eqn:vdecay}) imply that if $|u_1-u_2| \geq 1$ then
\begin{align*}
|K_{\mathrm{edge}}(x_1,u_1;x_2,u_2)| &\leq \frac{C}{|u_1-u_2|+1} + \frac{C \, |x_2-x_1|}{|u_1-u_2|+1} \; \int_0^1 t^{x_2-x_1-1}\, dt \\
& \leq \frac{2C}{|u_1-u_2|+1}.
\end{align*}
If $|u_1-u_2| < 1$, then we use the bound (\ref{eqn:Kxdecay}) to reach the same conclusion as above.
An entirely analogous bound holds when $x_2 < x_1$ because then $x_2 \leq x_1 -2$,
and $t^{x_2-x_1}$ is integrable over $t \in [1,\infty)$. Therefore, for $x_2 \neq x_1-1$,
\begin{equation} \label{eqn:Kudecay}
| K_{\rm{edge}}(x_1,u_1;x_2,u_2)| \leq \frac{C}{|u_1-u_2|+1}.
\end{equation}
Combining (\ref{eqn:Kxdecay}) with (\ref{eqn:Kudecay}) implies the
required inequality (\ref{eqn:correlation1}) for $x_2 \neq x_1-1$.

The case $x_2 = x_1-1 $ requires some care. The representation (\ref{eqn:1diffintegrand}) for
the integrand of $K_{\rm{edge}}(x_1,u_1;x_1-1,u_2)$ gives
\begin{align} \label{eqn:correlation5}
K_{\rm{edge}}(x_1,u_1;x_1-1,u_2) &=
- \frac{2}{\pi} \int_1^{\infty} t^{-1} \cos( tu_1 + \frac{\pi}{2} x_1) \sin( tu_2 + \frac{\pi}{2}x_1)\, dt \\
\nonumber & = - \int_1^{\infty} \frac{(-1)^{x_1} \sin((u_1+u_2)t) + \sin((u_2-u_1)t)}{\pi t} \, dt.
\end{align}

Integration by parts and the triangle inequality imply that for $a \geq 1$,
$$ \left | \int_1^{\infty} dt\, \frac{\sin(a t)}{t} \right | = \left | \frac{\cos(a)}{a} - \int_1^{\infty} dt\, \frac{\cos(at)}{at^2}\right| \leq \frac{C}{a} .$$
For $0 \leq a \leq 1$, we have
$$\left | \int_1^{\infty} dt\, \frac{\sin(a t)}{t} \right | = \left | \frac{\pi}{2} - \int_0^1 dt\, \frac{\sin(at)}{t} \right | \leq C + \int_0^1 dt\, a \leq C.$$
Together, these bounds imply that for $a \in \R$,
\begin{equation} \label{eqn:correlation6}
\left | \int_1^{\infty} dt\, \frac{\sin(a t)}{t} \right | \leq \frac{C}{|a|+1}.
\end{equation}

Separating (\ref{eqn:correlation5}) naturally into two integrals
and applying (\ref{eqn:correlation6}) implies that
$$ |K_{\rm{edge}}(x_1,u_1;x_1-1,u_2)| \leq \frac{C}{|u_1-u_2|+1}.$$
This establishes (\ref{eqn:correlation1}) for the case $x_2 = x_1-1$ and completes the proof.
\end{proof}

\paragraph{\textbf{Spatial ergodicity of $\X_{\rm{edge}}$}}
For $h \in \Z$, denote by $\tau^h$ the translation that maps $(x,u) \mapsto (x+2h,u)$ for $(x,u) \in \Z \times \R_{\geq 0}$.
So $\X_{\rm{edge}}$ is invariant under the action of every $\tau^h$ by Proposition \ref{prop:Kedge}.
An event $E$ associated to $\X_{\rm{edge}}$ is invariant if for every $h \in \Z$,
$E = \tau^h E$, where $\tau^hE = \{\tau^h(\omega): \omega \in E\}$ and $\tau^h(\omega)$ is the
action of $\tau^h$ on a sample outcome $\omega$ of $\X_{\rm{edge}}$. The invariant sigma-algebra
of $\X_{\rm{edge}}$ is the sigma-algebra $\F_{\rm{inv}}$ consisting of all the invariant events.

\begin{prop} \label{prop:tail}
$\X_{\rm{edge}}$ is ergodic w.r.t.~spatial translations in that if $E \in \F_{\rm{inv}}$ then $\pr{E} \in \{0,1\}$.
\end{prop}

\begin{proof}
For $A \subset \Z \times \R_{\geq 0}$, let $\F(A) = \sigma (\X_{\rm{edge}} \cap A)$ be the sigma-algebra
generated by the points of $\X_{\rm{edge}}$ restricted to $A$.  For $A, B \subset \Z \times \R_{\geq 0}$, let
$$\mathrm{dist}(A,B) = \inf \,\big \{ \max\{|x-y|, |u-v|\}: (x,u) \in A, (y,v) \in B \big \}.$$
For $k \geq 1$, suppose $f: (\Z \times \R_{\geq 0})^k \to \R$ is continuous and compactly supported. Let
$$N(f) = \sum_{\substack{(x_1,u_1), \ldots, (x_k,u_k) \in \X_\mathrm{edge}\\ (x_i,u_i)\;\text{all distinct}}} f(x_1,u_1;\cdots; x_k,u_k).$$
Now suppose $f, g : (\Z \times \R_{\geq 0})^k \to \R$ are continuous and compactly supported such that
there are disjoint subsets $A, B \subset \Z \times \R_{\geq 0}$ with
$\mathrm{support}(f) \subset A^k$ and $\mathrm{support}(g) \subset B^k$. This implies that if
$(x_1,u_1;\cdots; x_k,u_k) \in \mathrm{support}(f)$ and $(x_{k+1},u_{k+1};\cdots; x_{2k}, u_{2k}) \in \mathrm{support}(g)$,
then $(x_i,u_i) \neq (x_{k+j},u_{k+j})$ for every $1 \leq i,j \leq k$. We first show that in this case
\begin{equation} \label{eqn:tail1}
\left | \E{N(f)N(g)} - \E{N(f)}\E{N(g)} \right | \leq \frac{(2k)! C^{2k}}{\mathrm{dist}(A,B)^2 +1}\, ||f||_1 \, ||g||_1,
\end{equation}
where $C$ is the universal constant from Lemma \ref{lem:correlation} and $||f||_1$
is the $L^1$-norm of $f$ with respect to $(\#_{\Z} \otimes \Lb(\R_{\geq 0}))^{\otimes k}$.

Indeed, the assumption on the supports of $f$ and $g$ imply from \eqref{eqn:dpp} that
\begin{align*} \E{N(f)N(g)} = \int \limits_{(\Z \times \R_{\geq 0})^{2k}} & \det [ K_{\mathrm{edge}}(x_i,u_i;x_j,u_j)]_{1 \leq i,j \leq 2k}\,
f(x_1,u_1;\cdots; x_k,u_k) \times \\
 & g(x_{k+1},u_{k+1};\cdots; x_{2k},u_{2k})\, d \big (\#_{\Z} \otimes \Lb(\R_{\geq 0}) \big)^{\otimes 2k}.
\end{align*}
Let us expand the determinant of the $(2k) \times (2k)$ matrix above as a sum over all permutations.
We break up the permutations into two types: permutations that map the subsets $\{1,\ldots, k\}$ and $\{k+1,\ldots, 2k\}$
into themselves and those that do not. When summed over permutations of the first type the integral above
equals $\E{N(f)} \E{N(g)}$. For a permutation $\sigma$ of the second type, observe that there are two indices $i$ and $j$, with
$i \leq k$ and $j > k$, such that $\sigma(i) > k$ and $\sigma(j) < k$. Then for $\ell \in \{i,j\}$, Lemma \ref{lem:correlation} gives
$$ |K_{\mathrm{edge}}(x_{\ell},u_{\ell}; x_{\sigma(\ell)}, u_{\sigma(\ell)}) | \leq
\frac{C}{\max \{ |x_{\ell} - x_{\sigma(\ell)}|, |u_{\ell}-u_{\sigma(\ell)}|\} + 1}
\leq \frac{C}{1 + \mathrm{dist}(A,B)}.$$
For all other indices $\ell$ we have $|K_{\mathrm{edge}}(x_{\ell},u_{\ell}; x_{\sigma(\ell)}, u_{\sigma(\ell)})| \leq C$.
Consequently, the term involving $\sigma$ contributes at most $C^{2k-2} (1 + \mathrm{dist}(A,B))^{-2}$ in absolute value to the determinant
above for every $(x_1,u_1; \cdots; x_k,u_k) \in \mathrm{support}(f)$ and $(x_{k+1},u_{k+1};\cdots;x_{2k}, u_{2k}) \in \mathrm{support}(g)$.
Since there are $(2k)! -(k!)^2$ such permutations $\sigma$, we conclude that
\begin{align*}
\left | \E{N(f)N(g)} - \E{N(f)}\E{N(g)} \right | & \leq \frac{(2k)!C^{2k}}{\mathrm{dist}(A,B)^2+1}
\int \limits_{(\Z \times \R_{\geq 0})^{2k}} |f g| \, d \big (\#_{\Z} \otimes \Lb(\R_{\geq 0}) \big)^{\otimes 2k} \\
& = \frac{(2k)!C^{2k}}{\mathrm{dist}(A,B)^2+1}\, ||f||_1 ||g||_1.
\end{align*}

Let $\F^k(A)$ be the sigma-algebra generated by the random variables $N(f)$, where $f : A^k \to \R$
is continuous and compactly supported. The bound \eqref{eqn:tail1} implies that if $A$ and $B$ are disjoint,
$X$ is $\F^k(A)$-measurable and $Y$ is $\F^k(B)$-measurable, then,
\begin{equation} \label{eqn:tail2}
| \E{XY} - \E{X}\E{Y}| \leq \frac{(2k)! C^{2k}}{\mathrm{dist}(A,B)^2 + 1}\, \E{|X|} \E{|Y|}\,.
\end{equation}

The bound in \eqref{eqn:tail2} implies ergodicity of $\X_{\rm{edge}}$ as follows.
Let $E \in \F_{\rm{inv}}$. Given $0 < \eps < 1$, we may choose an event $E' \in \F^k([-n,n] \times \R_{\geq 0})$,
for some $k$ and $n$, such that $\pr{E \Delta E'} < \eps$. Since $\X_{\rm{edge}}$ is invariant
under $\tau^h$, we have that $\pr{ \tau^h E \Delta \tau^h E'} = \pr{E \Delta E'}$ for every $h$.
Therefore by the triangle inequality,
$$ | \pr{E' \cap \tau^h E'} - \pr{E \cap \tau^h E}| \leq \pr{E' \Delta E} + \pr{\tau^h E \Delta \tau^h E'} \leq 2 \eps.$$
Due to invariance of $E$ this implies that $|\pr{E' \cap \tau^h E'} - \pr{E}| \leq 2 \eps$.

Set $h = n + m$ for an integer $m \geq 1$. Then $\tau^h E' \in \F^k([n+2m,3n+2m] \times \R_{\geq 0})$.
We now apply \eqref{eqn:tail2} with $A = [-n,n] \times \R_{\geq 0}$ and $B = [n+2m,3n+2m] \times \R_{\geq 0}$,
observing that $\mathrm{dist}(A,B) = 2m$. Since $\pr{\tau^h E'} = \pr{E'}$ by translation invaraince, we infer that
$$ \left | \pr{E' \cap \tau^h E'} - \pr{E'}^2 \right | \leq \frac{(2k)!C^{2k}}{4m^2}.$$
Since $| \pr{E'} - \pr{E}| \leq \eps$, we conclude that
$$|\pr{E} - \pr{E}^2| \leq  \frac{(2k)!C^{2k}}{4m^2} + 5 \eps.$$
Letting $m \to \infty$ followed by $\eps \to 0$ shows that $\pr{E} = \pr{E}^2$, as required.
\end{proof}

\begin{rem} \label{rem:ergodic}
The proof above may be used to deduce that $\X_{\rm{edge}}$ is in fact space-time mixing.
\end{rem}

\begin{lemma} \label{lem:intensity}
Almost surely, $\X_{\rm{edge}}$ has an unbounded collection of points on every line $\{x\} \times \R_{\geq 0}$.
In fact, the following holds. Let $N_x(t) = \# \big (\X_{\mathrm{edge}} \cap (\{x\}\times [0,t]) \big)$.
For every $x$, the sequence $N_x(t)/t \to 1/\pi$ in probability as $t \to \infty$.
\end{lemma}

\begin{proof}
Fix $x \in \Z$. Using part (II) of Proposition \ref{prop:Kedge} we see that for any interval $[a,b] \subset \R_{\geq 0}$,
\begin{equation} \label{eqn:expdensity}
\E{ N_x([a,b])} = \int_a^b K_{\rm{edge}}(x,u,x,u)\,du
= \frac{b-a}{\pi} + \frac{(-1)^x}{2 \pi} \int_{2a}^{2b} du\, \frac{\sin u}{u}\,.
\end{equation}
Observe from (\ref{eqn:expdensity}) that $\E{N_x(t)/t} = 1/\pi + O(1/t)$ as $t \to \infty$.

From (\ref{eqn:diagonal}) we see that $K_{\rm{edge}}(x,u_1;x,u_2)$ is symmetric in the variables $u_1$ and $u_2$. Thus,
$$\rho(u_1,u_2) := K_{\rm{edge}}(x,u_1;x,u_2) K_{\mathrm{edge}}(x,u_2;x,u_1) = K_{\rm{edge}}(x,u_1;x,u_2)^2 \geq 0.$$
From the relation (\ref{eqn:dpp}) for determinantal point processes we have that
\begin{align*}
\E{N_x(t) \cdot (N_x(t)-1)} &= \int_0^t \int_0^t \det [ K_{\mathrm{edge}}(x,u_i;x,u_j)]_{i,j=1,2} \, du_1 du_2\\
&= \left (\int_0^t K_{\rm{edge}}(x,u;x,u) \,du \right)^2 - \int_0^t \int_0^t \rho(u_1,u_2) \, du_1 du_2 \\
&\leq  \E{N_x(t)}^2.
\end{align*}
This inequality implies that $\mathrm{Var}(N_x(t)) \leq \E{N_x(t)}$.
Since $\E{N_x(t)} = (t/\pi) + O(1)$, Chebyshev's inequality implies that for any $\eps > 0$,
$$ \pr{ \left | \frac{N_x(t)}{t} - \frac{1}{\pi} \right | > \eps} =  \frac{O(1)}{\eps^2 \, t}.$$
This provides the claimed convergence in probability.

Convergence in probability implies that there is a sequence of times $t_k \to \infty$
such that $N_x(t_k)/t_k \to 1/\pi$ almost surely as $k \to \infty$.
This in turn implies that there is an unbounded collection of points of $\X_{\rm{edge}}$ on
$\{x\} \times \R_{\geq 0}$ almost surely. An union bound over $x$ provides the claim in the lemma.
\end{proof}

\begin{lemma} \label{lem:EGhelper}
The following event occurs almost surely. For every
$t > 0$ there exists a doubly infinite sequence of integers $x_i$, $i \in \Z$,
such that $\X_{\rm{edge}}$ contains no points on each of the segments $\{2x_i\} \times [0,t]$.
\end{lemma}

\begin{proof}
By monotocity and an union (or rather intersection) bound over rational values of $t$,
it suffices to show that the event occurs almost surely for every fixed $t > 0$. Given a fixed
$t$, let $X_i$ be the indicator of the event that $\X_{\rm{edge}}$ has no points on
$\{2i\} \times [0,t]$. It suffices to show that almost surely infinitely many of the $X_i$s equal 1
for $i \geq 0$. Then, reflection invariance of $\X_{\rm{edge}}$ and another union bound
imply that almost surely a doubly infinite collection of the $X_i$s are equal to 1, as required.

Due to translation invariance of $\X_{\rm{edge}}$ the sequence $X_i$, $i = 0,1,2,\ldots$,
is stationary in that $(X_0,X_1, \ldots)$ has the same law as $(X_1,X_2, \ldots)$. It is
also ergodic by Proposition \ref{prop:tail}. Therefore, by the Ergodic Theorem,
$$ \lim_{n \to \infty}\, \frac{1}{n} \sum_{i=0}^{n-1} X_i \; =\; \pr{X_0 = 1},\;\; \text{almost surely}.$$
The probability that $X_0 = 1$ is the probability that $\X_{\rm{edge}}$ has no points in $\{0\} \times [0,t]$.
This is strictly positive by (\ref{eqn:Dyson}) below. As a result, an infinite number of the $X_i$s equal
1 whenever the limit in the above holds.
\end{proof}

\section{The local staircase shaped tableau} \label{sec:localtab}

The local staircase shaped tableau, henceforth, local tableau, is a random function on
\begin{equation} \label{eqn:infinitesst}
\st = \left \{ (x,y) \in \Z^2: y \geq 0, \, x \equiv y \;(\text{mod} \; 2) \right \}.
\end{equation}
Figure \ref{fig:st} provides an illustration. The \emph{rows} and \emph{columns} of $\st$ are
given by the diagonal lines
\begin{equation*}
\text{row} \; 2x = \{(2x-k,k): k \geq 0\}, \quad \text{column} \; 2x = \{(2x+k,k): k \geq 0\}.
\end{equation*}

\begin{figure}[htpb]
\begin{center}
\includegraphics[scale=0.5]{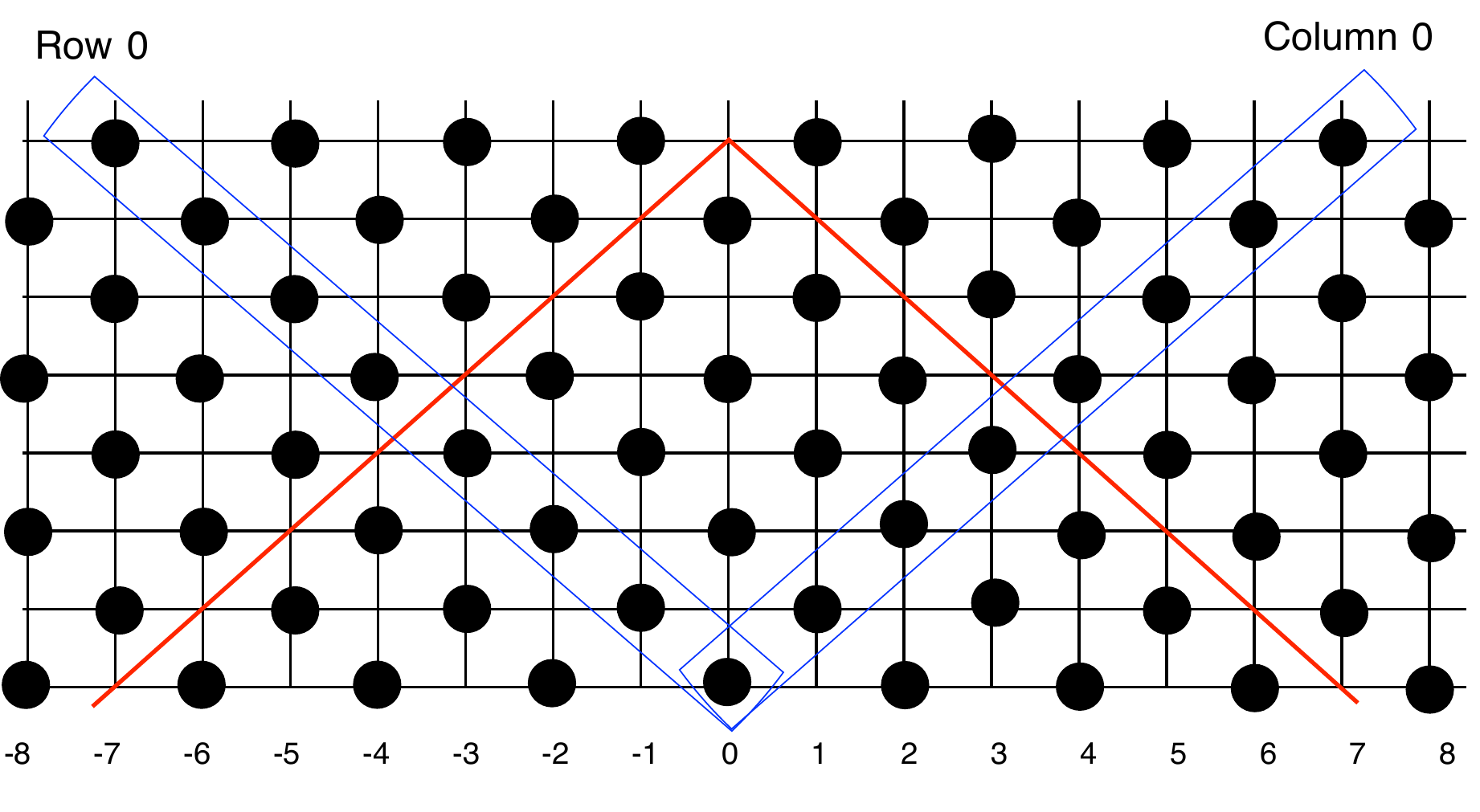}
\caption{\small{Black dots depicts cells of $\st$. Rows and columns are lines that start at even integers
and go in the directions $(-1,1)$ and $(1,1)$, respectively. The region bounded by the red lines is the
embedding of $\Delta_8$ into $\st$.}}
\label{fig:st}
\end{center}
\end{figure}

In order to define tableaux on $\st$ and their convergence, we first explain the topology on $[0,\infty]$
since tableau entries will take values there (we allow the value $\infty$). The topology on $[0,\infty]$ is the
usual topology on $\R_{\geq 0}$ extended in the natural way by stipulating that a sequence
converges to $\infty$ if its values diverge to $\infty$, possibly stabilizing to the value $\infty$. In this case
we will say that the sequence \emph{grows to} $\infty$.  For example, $1,2,3,4, \ldots$ grows to $\infty$,
as does $1,\infty,2,\infty, \ldots$, as well as $1,\infty,\infty,\infty,\ldots$.

A \emph{tableau} is a function $T: \st \to [0,\infty]$  such that it satisfies the \emph{tableau constraints}
\begin{align} \label{eqn:tabconstraint}
\text{I)} &\;T(x,y) \leq \min \left \{ T(x-1,y+1), \, T(x+1,y+1) \right \} \;\; \text{for every}\;\; (x,y) \in \st. \\ \nonumber
\text{II)} &\; \text{Along every row and column of}\;\; T\;\; \text{the entries grow to }\infty.
\end{align}

The YD $\sst$ embeds into $\st$ via $(i,j) \in \sst \mapsto (j-i - \ind{n \, \text{odd}}, n-i-j) \in \st$.
This is a rotation that puts row $r$ of $\sst$ on row $2(\lfloor n/2 \rfloor - r)$ of $\st$; see Figure \ref{fig:st}.
In this manner any PYT $T$ of shape $\sst$ embeds as a tableau
$F_T : \st \to [0,\infty]$ by setting
\begin{equation} \label{eqn:pytembed}
F_T(x,y) = \begin{cases}
n \, \Big (1-T \big ( \lfloor \frac{n}{2} \rfloor - \frac{x+y}{2}, \lfloor \frac{n}{2} \rfloor + \frac{x-y}{2} \big) \Big),
& \text{if}\;\; \big (\lfloor \frac{n}{2} \rfloor - \frac{x+y}{2}, \lfloor \frac{n}{2} \rfloor + \frac{x-y}{2} \big) \in \sst; \\
\infty, & \text{otherwise}.
\end{cases}
\end{equation}
By an abuse of notation we denote $F_T$ by $T$.

We say that a sequence of tableaux $T_n$ converges if there is a tableau $T$ such
that, in the aforementioned topology on $[0,\infty]$, $T_n(x,y) \to T(x,y)$ for every $(x,y) \in \Delta_{\infty}$.
Note we stipulate that a limit of tableaux remain a tableau.

A random tableau $\mathbf{T} : \st \to [0,\infty]$ is a Borel probability measure on tableaux
with respect to the topology above. Convergence of a sequence of random tableaux means
weak convergence with respect to this topology.

\subsection{Bulk local limit of staircase shaped tableaux} \label{sec:LST}

Section \ref{sec:limitsetup} describes how PYTs of a given shape are in
bijection with ensembles of non-increasing and non-intersecting paths whose initial
positions are given in terms of the shape. We describe the bijection explicitly for
tableaux defined on $\st$ as it will be useful in the proof of the local limit
theorem.

Consider an ensemble of paths $\{ p(2x,u) \}$, for $x \in \Z$ and $u \in \R_{\geq 0}$, that
satisfy the following.
\begin{align} \label{eqn:pathconstraints}
\text{I)} & \; p(x,\cdot) : \R_{\geq 0} \to \Z + \frac{1}{2} \;\;\text{is left continuous, non-increasing with}\;\; p(x,0)=2x + \frac{1}{2}. \\
\nonumber
\text{II)} & \; p(x,\cdot) \;\;\text{are non-intersecting:}\;\; p(x,u) > p(x-1,u) \;\;\text{for every}\;\; x \in \Z, u \in \R_{\geq 0}. \\
\nonumber
\text{III)} &\; \text{The jumps of the paths as defined by}\;\; \eqref{eqn:sytjumps}\;\; \text{is a discrete subset of}\;\ \Z \times \R_{\geq 0}.
\end{align}

For a tableau $T : \st \to [0,\infty]$, paths satisfying (\ref{eqn:pathconstraints})
are obtained by setting $p(x,u) = 2x + \frac{1}{2}$ for $0 \leq u \leq T(2x,0)$, and for $k \geq 1$,
$$p(x,u) = p(x,0) - k \;\;\text{if}\;\; T(2x-k+1,k-1) < u \leq T(2x-k,k).$$
In other words, $p(x,\cdot)$ is left continuous and decreases by integer units at times indexed by
row $2x$ of $T$. The paths are non-intersecting due to the columns of $T$ being non-decreasing.
Indeed, $p(x,u) - p(x-1,u) = 2 +N_{x-1}(u) - N_x(u)$, where $N_x(u)$ is the number of entries of
$T$ on row $2x$ with value at most $u$. Due to the columns being non-decreasing,
$N_{x-1}(u) \geq N_x(u) -1$, and thus, $p(x,u) - p(x-1,u) \geq 1$.
The jumps of the paths form a discrete set due to the rows and columns
of $T$ growing to $\infty$. When a row entry equals $\infty$ then the corresponding
path jumps only a finite number of times.

Let $X$ denote the jumps for the ensemble of paths associated to a tableau $T$ on $\st$.
The jumps can be read off from $T$ in the following manner. For every $x \in \Z$, the
jumps on the line $\{x\} \times \Z_{\geq 0}$ are the entries of $T$ whose cells have
first coordinate $x$ in $\st$. More precisely, if $u$ is the $k$-th smallest point of $X$
on $\{x\} \times \R_{\geq 0}$ then
\begin{equation} \label{eqn:kthjump}
u =T \big(x, \, 2k-1-\ind{x \, \text{even}} \big).
\end{equation}
If there are less than $k$ points on $\{x\} \times \R_{\geq 0}$, there is no such
$u$ and the tableau entry above equals $\infty$.

To see this, observe that $u$ is the time when the path starting at initial position
$(x+ \frac{1}{2})+2k-1 - \ind{x\,\text{even}}$ jumps for the $(2k-\ind{x\,\text{even}})$-th time.
Indeed, this is the $k$-th path starting at or to the right of position $x+ \frac{1}{2}$
and it hits position $x - \frac{1}{2}$ after jump number $2k-\ind{x\,\text{even}}$.
The first $k$ jumps on $\{x\} \times \R_{\geq 0}$ are the times when the first $k$
paths starting at or to the right of position $x + \frac{1}{2}$ hits position $x - \frac{1}{2}$.
Also, when there is no such $u$ it means that the path starting from
$(x+ \frac{1}{2})+2k-1 - \ind{x\,\text{even}}$ has exhausted its jumps and it does not
get to position $x - \frac{1}{2}$.

Let $M_{T \to X}$ denote the map from tableaux defined on $\st$ to jumps
of paths satisfying (\ref{eqn:pathconstraints}). This map is invertible
with the inverse given by the relation \eqref{eqn:kthjump}. Namely, $T(x,y)$
is the $\big[(y+1+\ind{x\,even})/2 \big]$--th smallest jump of $X$ on $\{x\} \times \R_{\geq 0}$
with the convention that $T(x,y) = \infty$ if no such jump exists. Let $M_{X \to T}$
denote the inverse map.

\begin{lemma} \label{lem:homeo}
The set of jumps of paths satisfying (\ref{eqn:pathconstraints}) is closed in the topology
on discrete subsets of $\Z \times \R_{\geq 0}$. The map $M_{T \to X}$ is a homeomorphism
from the set of tableaux on $\st$ to the set of jumps of paths satisfying
(\ref{eqn:pathconstraints}), with inverse given by $M_{X \to T}$.
\end{lemma}

\begin{proof}
The maps $M_{T \to X}$ and $M_{X \to T}$ are inverses by design.
We must show that they are continuous. We begin with continuity of $M_{T \to X}$.
Let $T_n$ be a sequence of tableaux such that $T_n$ converges to a tableau $T_{\infty}$.
Let $X_n = M_{T \to X}(T_n)$ and $X_{\infty} = M_{T \to X}(T_{\infty})$.
Recall from Section \ref{sec:dpp} that convergence of $X_n$ to $X_{\infty}$
requires that for every $x \in \Z$ and $k \geq 1$, the $k$-th smallest point of $X_n$
on $\{x\} \times \R_{\geq 0}$ must converge to the corresponding point of $X_{\infty}$,
while accounting for the case that there may be less than $k$ points.

Let $y = 2k-1-\ind{x \, \text{even}}$. Then, by \eqref{eqn:kthjump}, $T_n(x,y)$ is the $k$-th smallest point
of $X_n$ on $\{x\} \times \R_{\geq 0}$ and similarly for $T_{\infty}(x,y)$.
There are two cases: $T_{\infty}(x,y) < \infty$ or $T_{\infty}(x,y) = \infty$.
In the former case, the $k$-th smallest point of $X_n$ on $\{x\} \times \R_{\geq 0}$
is eventually finite and the same for $X_{\infty}$. Moreover, we have convergence
of these points since $T_n(x,y) \to T_{\infty}(x,y)$. In the latter case, given any
bounded subset of $\{x\} \times \R_{\geq 0}$, the $k$-th smallest point of $X_n$
eventually escapes the set or is non-existent due to $T_n(x,y) \to \infty$.
This is as required since $X_{\infty}$ has no $k$-th smallest point on $\{x\} \times \R_{\geq 0}$.
This proves that $M_{T \to X}$ is continuous.

Now we show that the set of jumps of paths satisfying (\ref{eqn:pathconstraints}) is
a closed set in the space of discrete subsets of $\Z \times \R_{\geq 0}$, as well as that
that $M_{X \to T}$ is continuous. Suppose $X_n$ is a sequence of such jumps sets and that
it converges to a discrete subset $X_{\infty}$.

Let $T_n = M_{X \to T}(X_n)$. First, we show that $T_n(x,y)$ converges for every $(x,y) \in \st$.
Indeed, $T_n(x,y)$ is the $k$-th smallest point of $X_n$ on $\{x\} \times \R_{\geq 0}$ for
$k = (y+1+\ind{x\,\text{even}})/2$. Therefore, convergence of $X_n$ to $X_{\infty}$
implies that $T_n$ must converge to some function $T_{\infty}: \st \to [0,\infty]$.
Note that $T_{\infty}(x,y) = \infty$ if and only if $X_{\infty}$ has less than $k$ points
on $\{x\} \times \R_{\geq 0}$.

The function $T_{\infty}$ is a tableau because the tableau inequalities from (\ref{eqn:tabconstraint})
continue to hold in the entry-wise limit, and the rows and columns will grow to $\infty$ due to
$X_{\infty}$ being a discrete set. Thus, consider $\hat{X}_{\infty} = M_{T \to X}(T_{\infty})$.
By the first part of the proof, $X_n \to \hat{X}_{\infty}$. But then, $\hat{X}_{\infty} = X_{\infty}$
because limits of discrete subsets of $\Z \times \R_{\geq 0}$ are unique. This shows both
the closure property of sets of jumps for paths satisfying (\ref{eqn:pathconstraints}) and
the continuity of $M_{X \to T}$.
\end{proof}

We now state the local limit theorem for Poissonized staircase shaped tableaux in the bulk.
For $\alpha \in (-1,1)$, let $c_n = 2 ( \lfloor \frac{n}{2} \rfloor - \lfloor \frac{(1+\alpha)n}{2} \rfloor)$.
For a PYT $T$ having shape $\sst$, embed it as a tableau on $\st$ according to \eqref{eqn:pytembed}
and consider the rescaled tableau $T_{\alpha,n} : \st \to [0,\infty]$ defined by
\begin{equation} \label{eqn:embeddedsyt}
T_{\alpha,n}(x,y) = \sqrt{1-\alpha^2} \; T(x+c_n, y).
\end{equation}
Let $\T_{\sst}$ be a uniformly random PYT of shape $\sst$ and denote by $\T_{\alpha,n}$
the random tableau associated to $\T_{\sst}$ by (\ref{eqn:embeddedsyt}).

\begin{thm} \label{thm:tableaulocallimit}
The sequence of random Poissonized tableaux $\T_{\alpha,n}$ converges weakly to
a random tableau $\T_{\rm{edge}}$. Moreover, the law of $\T_{\rm{edge}}$
is $M_{X \to T}( \X_{\rm{edge}})$.
\end{thm}

\begin{proof}
Observe that $|c_n - \alpha n| \leq 2$ for every $n$. With this choice of $c_n$,
the jump process associated to $\T_{\alpha,n}$ has law $\X_{\alpha,n}$ from (\ref{eqn:edgeprocess})
because these jumps are simply the jumps of $\T_{\sst}$ rescaled onto $\Z \times \R_{\geq 0}$
as in (\ref{eqn:edgeprocess}). Theorem \ref{thm:main} asserts that $\X_{\alpha,n}$ converges
weakly to $\X_{\rm{edge}}$. Due to being a weak limit of the jumps of ensembles of paths satisfying
(\ref{eqn:pathconstraints}), $\X_{\rm{edge}}$ is also almost surely the jumps of such an ensemble
of paths by the closure property given in Lemma \ref{lem:homeo}. The continuity of $M_{X \to T}$
then implies that $M_{X \to T}(\X_{\alpha,n})$ converges weakly to $M_{X \to T}(\X_{\rm{edge}})$.
Therefore, $\T_{\alpha,n}$ converges weakly to a random tableau
$\T_{\rm{edge}}$ having the law of $M_{X \to T}(\X_{\rm{edge}})$.
\end{proof}

\paragraph{\textbf{Bulk local limit theorem for random staircase shaped SYT}}
Recall $\Tb_{\sst}$ denotes a uniformly random SYT of shape $\sst$ and $N = \binom{n}{2}$.
Consider the rescaled tableau
\begin{equation} \label{eqn:rsctab}
\Tb^{\rm{rsc}}_{\sst}(i,j) = \frac{\mathbf{T}_{\sst}(i,j)}{N+1}.
\end{equation}

\begin{thm} \label{corr:LST}
The random tableau $\Tb^{\rm{rsc}}_{\sst}$ converges to $\T_{\rm{edge}}$ in the bulk local limit,
that is, under the embedding and rescaling from \eqref{eqn:embeddedsyt}.
\end{thm}

\begin{proof}
Consider the following coupling between $\Tb_{\sst}$ and $\T_{\sst}$. Given $\Tb_{\sst}$,
independently sample $P_{(1)} < P_{(2)} < \cdots < P_{(N)}$ according to the
order statistics of $N$ i.i.d.~random variables distributed uniformly on $[0,1]$.
Insert the entry $P_{(k)}$ into the cell of $\sst$ that contains entry $k$ of $\Tb_{\sst}$.
The resulting tableau has the law of $\T_{\sst}$. Using this coupling, and due
to the manner the scaling from \eqref{eqn:embeddedsyt} is defined, it suffices to show
the following in order to conclude that $\Tb^{\rm{rsc}}_{\sst}$ converges to $\T_{\rm{edge}}$
by way of Theorem \ref{thm:tableaulocallimit}.

Fix an $L > 0$ and consider any $(x,y) \in \st$ such that $(x,y) \in [-L,L] \times [0,L]$.
Let $P_{(k)}$ be the entry of $\T_{\sst}$ inside $(x,y)$ under the embedding from \eqref{eqn:embeddedsyt}.
Then as $n \to \infty$, we need to show that
\begin{equation} \label{eqn:deposs}
n \, \Big | P_{(k)} - \frac{k}{N+1} \Big | \;\; \longrightarrow \; 0 \;\; \text{in probability}.
\end{equation}
The number $k$ is random, its distribution depends on $\Tb$ as well as $n$ and $\alpha$.

In order to establish \eqref{eqn:deposs} we will use the following auxiliary fact, which is
a byproduct of \cite[Theorem 11]{AHRV}. There is a number $\delta_n$ of order $o(N)$
as $n \to \infty$, such that with probability tending to 1 as $n \to \infty$, every entry of
$\Tb_{\sst}$ within the cells of $\st \cap [-L,L] \times [0,L]$ under the embedding \eqref{eqn:embeddedsyt}
has value at least $N - \delta_n$. As a consequence, $k \geq N - \delta_n$ with probability tending to 1.
We write
\begin{equation} \label{eqn:deposs2}
\pr{n \, \left| P_{(k)} - \frac{k}{N+1} \right| > \eps} \leq \, \frac{n^2}{\eps^2}\,
\E{ \left| P_{(k)}- \frac{k}{N+1} \right|^2 \, \Big |\, k \geq N- \delta_n} + \pr{k < N-\delta_n}.
\end{equation}

For a fixed deterministic $j$,  $P_{(j)}$ has a Beta distribution with parameters $j$ and $N+1-j$, which
has mean $j / (N+1)$ and variance
$$ \E{ \left| P_{(j)} - \frac{j}{N+1} \right|^2 } = \frac{j (N+1-j)}{(N+1)^2(N+2)}.$$
Since the $P_{(j)}$s are independent of $\Tb$, employing the bound above for $j \geq N- \delta_n$ and
summing over the probabilities of $k$ give
$$ \E{ \left| P_{(k)}- \frac{k}{N+1} \right|^2 \; \Big |\; k \geq N- \delta_n} \leq \frac{\delta_n}{N^2}.$$
The latter quantity is of order $o(1)/N$ as $n \to \infty$. Since $N={ n \choose 2}$, we conclude that both terms on
the right hand side of \eqref{eqn:deposs2} tend to 0 as $n \to \infty$.
\end{proof}

\paragraph{\textbf{Statistical properties of the local staircase shaped tableau}}
The set $\st$ can be made into a directed graph by putting directed
edges from each vertex $(x,y) \in \st$ to the vertices $(x-1,y+1)$ and $(x+1,y+1)$.
The automorphisms of this graph consists of translations $\phi_h$,
for $h \in \Z$, given by $\phi_h(x,y) = (x+2h,y)$, as well as a reflection $\phi_{-}$
given by $\phi_{-}(x,y) = (-x,y)$. Tableaux are preserved by these automorphisms.

A random tableau $\mathbf{T}$ is \emph{translation invariant} if
$\mathbf{T} \circ \phi_h$ has the same law as $\mathbf{T}$ for every translation $\phi_h$.
The random tableau is reflection invariant if $\mathbf{T} \circ \phi_{-}$ has the same law as $\mathbf{T}$.
The translation invariant sigma-algebra of $\mathbf{T}$ is the sigma-algebra of events that
remain invariant under every translation:
$$ \F_{\rm{inv}} = \left \{ \text{Events}\; E \;\text{associated to}\; \mathbf{T}\; \text{s.t.}\; \phi_h E = E\; \text{for every}\; h \in \Z. \right \}.$$
(Recall that $\phi_h E = \{\omega \circ \phi_h: \omega \in E\}$.)
We say $\mathbf{T}$ is ergodic under translations if $\F_{\rm{inv}}$ is the trivial sigma-algebra.

\begin{prop} \label{thm:LSTstatistics}
The local tableau $\T_{\rm{edge}}$ has the following statistical properties.
\begin{enumerate}
\item Almost surely, $\T_{\rm{edge}}(x,y)$ is finite for every $(x,y) \in \st$
and the entries of $\T_{\rm{edge}}$ are all distinct.

\item The law of $\T_{\rm{edge}}$ is both translation and reflection invariant.

\item $\T_{\rm{edge}}$ is ergodic under translations.

\item Almost surely, for every $t > 0$ there are infinitely many positive and negative
$x \in \Z$ such that $T_{\rm{edge}}(2x,0) > t$.
\end{enumerate}
\end{prop}

\begin{proof}
Almost surely, $\X_{\rm{edge}}$ has an infinite and unbounded collection of points
on every line $\{x\} \times \R_{\geq 0}$ by Lemma \ref{lem:intensity}. Also,
almost surely, $\X_{\rm{edge}}$ does not contain two points of the form
$(x,u)$ and $(y,u)$ with $x \neq y$. To see this, observe from the relation (\ref{eqn:dpp})
for determinantal point processes that the expected number of such pairs of points in $\X_{\rm{edge}}$
is 0 due to the set of such pairs having measure zero with respect to the measure
$\big(\# \Z \otimes \Lb(\R_{\geq 0})\big)^{\otimes 2}$. When both these properties hold,
$\T_{\rm{edge}}$ satisfies (1).

The law of $\T_{\rm{edge}}$ is invariant under translations because for
every translation $\phi_h$, the tableau $\T_{\rm{edge}} \circ \phi_h$ is constructed
from the jump process $\X_{\rm{edge}} + (2h,0)$, which has the same law of $\X_{\rm{edge}}$ by
Proposition \ref{prop:Kedge}. Similarly, reflection invariance of $\T_{\rm{edge}}$ follows from reflection
invariance of $\X_{\rm{edge}}$. This establishes (2).

The ergodicty of $\T_{\rm{edge}}$ under translations follows from the ergodicity of $\X_{\rm{edge}}$
under translations (Proposition \ref{prop:tail}). This is because a translation invariant event for
$\T_{\rm{edge}}$ is the image of a translation invariant event for $\X_{\rm edge}$ under the map $M_{X \to T}$.
Finally, (4) is the statement of Lemma \ref{lem:EGhelper}.
\end{proof}

\section{Random sorting networks} \label{sec:sorting}

\subsection{Sorting networks, Young tableaux and Edelman-Greene bijection} \label{sec:EG}

Stanley \cite{Stanley} enumerated the number of sorting networks of $\mathfrak{S}_n$, which equals
\begin{equation*} \label{eqn:stanley}
 \frac{\binom{n}{2}!}{\prod_{j=1}^{n-1} (2n-1-2j)^j}.
\end{equation*}
Following Stanley, Edelman and Greene \cite{EG} provided an explicit bijection between sorting
networks and staircase shaped SYT. An account of further combinatorial developments may be found in \cite{Gar, HY}.
We describe the part of the Edelman-Greene bijection that maps staircase shaped tableaux to sorting networks.
The inverse map is a modification of the RSK algorithm; we do not describe it here since it is not used in the paper.
See \cite{EG} or \cite[Section 4]{AHRV} for a full description of the bijection.

Recall that a sorting network of $\mathfrak{S}_n$ is identified by its sequence of adjacent swaps
$(s_1,\ldots,s_N)$, where $N = \binom{n}{2}$. For the rest of the paper we will use $N$ to denote $\binom{n}{2}$.
For $T \in \rm{SYT}(\sst)$, we adopt the convention that $T(i,j) = - \infty$ if $(i,j) \notin \sst$.
\smallskip

\paragraph{\textbf{The Sch\"{u}tzenberger operator}}
Let $(i_{\max}(T),j_{\max}(T))$ denote the cell containing the maximum entry of a  SYT $T$.
The Sch\"{u}tzenberger operator $\Phi : \mathrm{SYT}(\sst) \to \mathrm{SYT}(\sst)$ is a bijection defined
as follows. Given $T \in \rm{SYT}(\sst)$, construct the \emph{sliding path} of cells
$c_0, c_1,\ldots, c_{d-1} \in \sst$ iteratively in the following manner.
Set $c_0 = (i_{\max}(T), j_{\max}(T))$ and $c_d = (1,1)$. Then set
$$c_{r+1} = \mathrm{argmax} \, \big \{ T(c_r - (1,0)), \, T(c_r - (0,1)) \big \}.$$
Let $\Phi(T) = [\hat{T}(i,j)]$ where $\hat{T}(c_r) = T(c_{r+1}) + 1$ for $0 \leq r \leq d-1$,
$\hat{T}(c_d) = 1$, and $\hat{T}(i,j) = T(i,j)+1$ for all other cells $(i,j) \in \sst \setminus \{c_0,\ldots, c_d\}$.
Figure \ref{fig:schutzenberger} provides an illustration.

\begin{figure}[htpb]
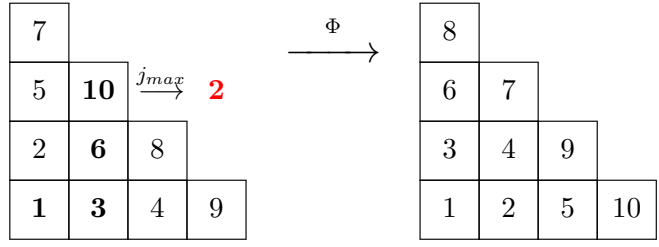

\begin{center}
\ytableausetup{boxsize=2em}
\begin{ytableau}
7 \\
5 & \mathbf{10} & \none[\;\overset{j_{max}}{\longrightarrow}] & \none[\textcolor{red}{\mathbf{2}}] \\
2 & \mathbf{6} & 8 \\
\mathbf{1} & \mathbf{3} & 4 & 9
\end{ytableau}
\quad $\Large \overset{\Phi}{\xrightarrow{\hspace*{1cm}}}$ \quad
\begin{ytableau}
8 \\
6 & 7 \\
3 & 4 & 9 \\
1 & 2 & 5 & 10
\end{ytableau}
\caption{\small{First step of the Edelman-Greene algorithm on a staircase shaped SYT of size 10.
Sliding path is bolded. The associated sorting network of $\mathfrak{S}_5$ has swap
sequence (2,4,3,1,2,1,4,3,2,4).}}
\label{fig:schutzenberger}
\end{center}
\end{figure}

The \emph{Edelman-Greene map} $\rm{EG} : \mathrm{SYT}(\sst) \mapsto \{\text{sorting networks of}\; \mathfrak{S}_n\}$
is defined by
\begin{equation} \label{eqn:EGmap}
\mathrm{EG}(T) = \left ( j_{\max}\big(\Phi^k(T)\big)\right)_{0 \leq k \leq N-1}\,,
\end{equation}
where $\Phi^k$ is the $k$-th iterate of $\Phi$. Edelman and Greene \cite[Theorem 5.4]{EG} proved
that $\mathrm{EG}$ indeed maps to sorting networks and that is has an inverse.

\subsection{First swap times of random sorting networks: proof of Corollary \ref{cor:firstswap}} \label{sec:firstswap}

Let $T_{\rm{FS}}(s)$ be the first time the adjacent swap $(s,s+1)$ appears in a sorting network $\omega$
of $\mathfrak{S}_n$. According to the Edelman-Greene bijection, this time is recorded in
the entry $(n-s,s)$ of $\mathrm{EG}(\omega)$. Thus,
$$T_{\rm{FS}}(s) = N+1 - \mathrm{EG}(\omega)(n-s,s).$$
In terms of the rescaled tableau $\Tb_{\sst}^{\rm{rsc}}$ from \eqref{eqn:rsctab} we have that
$$\Tb_{\mathrm{FS}}(s) \; \overset{law}{=} \; (N+1) \cdot (1 - \Tb_{\sst}^{\rm{rsc}}(n-s,s)).$$
This implies the following for $\Tb_{\mathrm{FS}, \alpha,n}$ -- the first time an adjacent swap between
$\lfloor \frac{n(1+\alpha)}{2} \rfloor$ and  $\lfloor \frac{n(1+\alpha)}{2} \rfloor+1$ appears in a
random sorting network of $\mathfrak{S}_n$:
$$\frac{2 \sqrt{1-\alpha^2}}{n}\, \Tb_{\mathrm{FS}, \alpha,n} \; \overset{law}{=} \;
\frac{2N+2}{n^2} \,\Tb_{\alpha, \sst}^{\rm{rsc}}(0,0).$$
Here, $\Tb_{\alpha, \sst}^{\rm{rsc}}$ is the tableau $\Tb_{\sst}^{\rm{rsc}}$ rescaled and embedded into $\st$
according to (\ref{eqn:embeddedsyt}). Theorem \ref{corr:LST} implies that $\Tb_{\alpha, \sst}^{\rm{rsc}}(0,0)$
converges weakly to $\T_{\rm{edge}}(0,0)$. Since $(2N+2)/n^2 \to 1$, we conclude that the rescaled
$\Tb_{\mathrm{FS}, \alpha,n}$ from above converges weakly to $\T_{\rm{edge}}(0,0)$.
Thus, $\Tb_{\rm{FS}}$ has the law of $\T_{\rm{edge}}(0,0)$.

Now we explain how to get the distribution function of $\Tb_{\rm{FS}}$ given in
(\ref{eqn:firstswaplaw}). Observe that the event $\{ \T_{\rm{edge}}(0,0) > t\}$ is
the event $\{ \X_{\rm{edge}} \cap (\{0\} \times [0,t]) = \emptyset \}$.
The probability of the latter (often known as ``gap probability") has the representation given by
(\ref{eqn:firstswaplaw}), which is the Fredholm determinant of $K_{\rm{edge}}$ over $L^2(\{0\} \times [0,t])$.
This is a well-known property of determinantal point processes under the condition that the kernel
be of trace class \cite{DV}. The kernel $K_{\rm{edge}}$ is of trace class on $L^2(\{0\} \times [0,t])$ simply because
$|K_{\rm{edge}}(0,u_1;0,u_2)| \leq 2/\pi$.

The asymptotic behaviour of the distribution function of $\Tb_{\rm{FS}}$ is well-known:
\begin{equation} \label{eqn:Dyson}
 \log \pr{\Tb_{\rm{FS}} > t} = - \frac{1}{4}t^2 - \frac{1}{2}t - \frac{1}{8} \log t
  +\frac{7}{24} \log 2+ \frac{3}{2}\zeta'(-1) + o(1)\;\;\text{as}\; t \to \infty.
 \end{equation}
The formula \eqref{eqn:Dyson} has a history. In theoretical physics literature,
the leading term in \eqref{eqn:Dyson} was first studied in \cite{dCM}, while the
full expansion was given in \cite{Dyson}. The complete mathematical treatment was
developed in \cite{DIZ,K,E1,DIKZ,E2}; the present form of \eqref{eqn:Dyson} is given
in the last reference.

We will only need the simple corollary of \eqref{eqn:Dyson} that $\pr{\Tb_{\rm{FS}}> t} > 0$ for every $t$.

\subsection{Edelman-Greene algorithm on the local tableau} \label{sec:localEG}

The procedure described here is the same as the one given in the Introduction except
that it is in the language of tableaux instead of their jumps. In order to define the
Edelman-Greene algorithm on the local tableau we first introduce some concepts
that allow us to define Edelman-Greene algorithm on tableaux defined on $\st$.

A \emph{directed path} from $(x,y) \in \st$ to $(x',y') \in \st$ is a sequence of cells
$c_0 = (x,y), c_1, \ldots, c_k = (x',y')$ of $\st$ such that $c_{i+1}-c_i \in \{ (-1,1), (1,1)\}$ for every $i$.
The cells of $\st$ can be partially ordered as follows: $(x,y) \leq (x',y')$ if there is a directed path
from $(x,y)$ to $(x',y')$. Recall that $\st$ is a directed graph with edges from $(x,y)$ to $(x\pm 1, y+1)$.
It can also be thought of as an undirected graph by forgetting the direction of the edges.
A \emph{connected} subset of $\st$ is a connected subgraph of $\st$ in the undirected sense.

A \emph{Young diagram} (YD) of $\st$ is a connected subset $\lm$ that is downward closed in the partial
order, that is, if $(x,y) \in \lm$ and $(x',y') \leq (x,y)$ then $(x',y') \in \lm$. For example, $\sst$ is a YD
of $\st$. The \emph{boundary} of $\lm$, $\partial \lm$, consists of cells $(x,y) \notin \lm$ such that there
is a directed edge from some cell $(x',y') \in \lm$ to $(x,y)$. The \emph{peaks} of $\lm$ consists
of the maximal cells of $\lm$ in the partial order.

Let $T: \st \to [0,\infty]$ be a tableau as in (\ref{eqn:tabconstraint}). A \emph{sub-tableau}
is the restriction of $T$ to a YD $\lm$; we say $\lm$ is the \emph{support} of the sub-tableau.
Let $T^{\rm{finite}} = \{ T(x,y): T(x,y) \neq \infty\}$. We take the support of $T$ to be the support
of $T^{\rm{finite}}$. Observe that $T^{\rm{finite}}$ is a countable disjoint union of sub-tableaux of $T$,
say $T_1, T_2, \ldots$. Indeed, the support of the $T_i$s are the connected components of the
subgraph spanned by cells $(x,y)$ such that $T(x,y) \neq \infty$. We will call the $T_i$s the
\emph{clusters} of $T$. The tableau $T$ is \emph{EG-admissible} if all the entries of $T^{\rm{finite}}$
are distinct and every cluster $T_i$ is supported on a YD of \textbf{finite size}.
\smallskip

\paragraph{\textbf{Edelman-Greene algorithm on a finite tableau}}

Let $\lm$ be a YD of $\st$ of finite size and $T: \lm \to \R_{\geq 0}$ a tableau
such that all its entries are distinct. The Edelman-Greene map $\rm{EG}$ takes
as input $T$ and outputs a triple $(x,t,\hat{T})$, where $x \in \Z$, $t \in \R_{\geq 0}$
and $\hat{T}$ is a sub-tableau.

The sliding path of $T$ is a directed path $c_0, c_1, \ldots, c_k$ defined by
\begin{enumerate}
\item $c_0 = \mathrm{argmin} \, \{ \, T(x,y): (x,y) \in \lm \}$.
\item $c_{i+1} = \mathrm{argmin} \, \{ \, T(c_i + (-1,1)), T(c_i + (1,1)) \,\}$.
\item $c_k = $ peak of $\lm$ obtained when both $c_k + (\pm 1,1)$ belong to $\partial \lm$.
\end{enumerate}
Let $\hat{\lm} = \lm \setminus \{c_k\}$ and define $\hat{T} : \hat{\lm} \to \R_{\geq 0}$ by
\begin{equation*}
\hat{T}(x,y) = \begin{cases}
T(x,y), & \text{if}\;\; (x,y) \in \hat{\lm} \setminus \{c_0, \ldots, c_{k-1}\}; \\
T(c_{i+1}), & \text{if} \;\; (x,y) = c_i\;\; \text{for some}\;\; 0 \leq i \leq k-1.
\end{cases}
\end{equation*}
The cell $c_0$ must be on the bottom level of $\st$ and has the form $(2x,0)$ for some $x \in \Z$.
Set $t = T(c_0)$. The output is $\mathrm{EG}(T) = (x,t,\hat{T})$, and empty if $T$ is the empty tableau.

The Edelman-Greene algorithm on $T$ outputs a discrete subset $S(T) \subset \Z \times \R_{\geq 0}$,
denoted the \emph{swaps} of $T$. Let $(x_j, t_j, \hat{T}_j)$, for $1 \leq j \leq |\lm|$, be defined
iteratively by $(x_1,t_1,\hat{T}_1) = \mathrm{EG}(T)$ and $(x_j,t_j,\hat{T}_j) = \mathrm{EG}(\hat{T}_{j-1})$
for $2 \leq j \leq |\lm|$. Then,
\begin{equation} \label{eqn:swaps}
S(T) = \{ (x_j,t_j): 1 \leq j \leq |\lm| \}.
\end{equation}
If the cell $(x,y) \in \lm$ contains the $k$-th smallest entry of $T$ then its entry is removed
during the $k$-th iteration of the algorithm. We will say that the entry at $(x,y)$ \emph{exits}
at time $t_k$ from row $x_k$. We will also say that $(x_k, t_k)$ \emph{originates} from cell
$(x,y)$.
\smallskip

\paragraph{\textbf{Edelman-Greene algorithm on an admissible tableau}}
Let $T_1, T_2, \ldots$ be the clusters of an EG-admissible tableau $T$.
Observe that for $i \neq j$, the swaps of $T_i$ and $T_j$ exit from mutually
disjoint rows. Thus, the swap sets $S(T_1), S(T_2), \ldots$ are row-wise mutually
disjoint. The swaps of $T$ are defined as
$$S(T) = \bigcup_{i} \, S(T_i).$$

The local tableau $\T_{\rm{edge}}$ is not EG-admissible. In order to define swaps
for the local tableau we cut off large entries so that it becomes EG-admissible,
and then process the tableau in a graded manner. For this to be successful,
the EG algorithm ought to be consistent in the sense that running it on a tableau, and then
restricting to swaps that originate from a sub-tableau, must produce the same
outcome as the algorithm applied to the sub-tableau. This is not always the case
and the following explains when it may be so.

Given two tableaux $T_{\rm{small}}$ and $T_{\rm{big}}$, we say $T_{\rm{small}} \leq T_{\rm{big}}$
if the following criteria hold.
\begin{enumerate}
\item $T_{\rm{small}}(x,y) = T_{\rm{big}}(x,y)$ for every $(x,y) \in \mathrm{support}(T_{\rm{small}})$.
\item For every $(x,y) \in \mathrm{support}(T_{\rm{small}})$, and $(x',y') \in
\mathrm{support}(T_{\rm{big}}) \setminus \mathrm{support}(T_{\rm{small}})$, if $(x,y)$
belongs to the same cluster of $T_{\rm{big}}$ as $(x',y')$ then $T_{\rm{small}}(x,y) < T_{\rm{big}}(x',y')$.
\end{enumerate}

\begin{lemma} \label{lem:EGconsistent}
Let $T_{\rm{small}} \leq T_{\rm{big}}$, and suppose that $T_{\rm{big}}$ is EG-admissible.
Then, applying the EG algorithm to $T_{\rm{big}}$ and restricting to the swaps that originate
from the cells of $T_{\rm{small}}$ produces the same outcome as applying the EG algorithm
to $T_{\rm{small}}$. In particular, $S(T_{\rm{small}}) \subset S(T_{\rm{big}})$.
\end{lemma}

\begin{proof}
Observe that the clusters of $T_{\rm{small}}$ are contained within the clusters of $T_{\rm{big}}$.
The EG algorithm acts independently on each cluster of $T_{\rm{big}}$ is a row-wise
disjoint manner. Fix a particular cluster $T$ of $T_{\rm{big}}$, and suppose
that the clusters of $T_{\rm{small}}$ that are contained inside $T$ are $T_1, \ldots, T_k$.
It suffices to prove that the EG algorithm applied to $T$, and then restricted to
the swaps that originate from $T_1, \ldots, T_k$, produces the same outcome as
the algorithm applied to each individual $T_i$.

Let $\lm = \mathrm{support}(T)$ and $\lm_i = \mathrm{support}(T_i)$. The assumption is that
each entry of $\lm \setminus (\cup_i \lm_i)$ is larger than every entry of $\cup_i \lm_i$.
Therefore, the EG algorithm applied to $T$ will process every entry of $\cup_i \lm_i$
before it ever processes an entry from the complement. When some entry
from $\lm \setminus (\cup_i \lm_i)$ enters a cell of some $\lm_i$ during
the first $\sum_i |\lm_i|$ steps, the algorithm treats that entry as if it were $\infty$.
Since $T_{\rm{small}}$ agrees with $T_{\rm{big}}$ on $\cup_i \lm_i$, the
EG algorithm will output the swaps of $T_1, \ldots, T_k$ during the
first $\sum_i |\lm_i|$ steps, and then output the remaining swaps of
$T \setminus (\cup_i T_i)$. This is what was claimed.
\end{proof}

A tableau $T$ is \emph{graded EG-admissible} if all of its finite-valued entries are distinct
and, if for every $t > 0$, the sub-tableau
$$T^{\leq t} = \{ T(x,y): T(x,y) \leq t\} \;\; \text{is EG-admissible}.$$
Observe that $T^{\leq t_1} \leq T^{\leq t_2}$ whenever $t_1 \leq t_2$.
Lemma \ref{lem:EGconsistent} thus implies that $S(T^{\leq t_1}) \subset S(T^{\leq t_2})$.
Therefore, for a graded EG-admissible tableau $T$, we may define
\begin{equation} \label{eqn:fullswaps}
S(T) = \bigcup_{t \geq 0} \, S(T^{\leq t}).
\end{equation}

\begin{lemma} \label{lem:EGbounded}
Suppose a sequence of tableaux $T_n \to T_{\infty}$, and also that every $T_n$ and $T_{\infty}$
are graded EG-admissible. Then for every integer $x$ and $t \geq 0$, there is a finite
YD $\lm$ that contains the cluster of $(2x,0)$ in $T_n^{\leq t}$ for every $n$.
\end{lemma}

\begin{proof}
This follows from a diagonalization argument, more precisely, K\"{o}nig's infinity lemma, which
states that every infinite connected graph with finite vertex degrees contains an
infinite path.

Suppose for the sake of a contradiction that the conclusion of the lemma fails.
Let $T_{n,x,t}$ denote the cluster of $(2x,0)$ in $T_n^{\leq t}$. Call a cell $(x',y') \in \st$
bad if there is a undirected path in $\st$ from $(2x,0)$ to $(x',y')$ that is contained
in infinitely many of the clusters $T_{n,x,t}$. Consider the connected
component of $(2x,0)$ in $\st$ that is spanned by the subgraph of bad vertices.
If the component is finite then there is a finite YD $\lm$ that contains the component.
This implies that for all sufficiently large $n$, every cell of $\partial \lm$ lies outside $T_{n,x,t}$
because any path from $(2x,0)$ to a cell outside $\lm$ must pass through $\partial \lm$.
Therefore, $T_{n,x,t} \subset \lm$ for all large $n$. Since every $T_n$ is
graded EG-admissible, this means that there is a finite YD that contains
every $T_{n,x,t}$, which is a contradiction.

Therefore, the connected component of $(2x,0)$ spanned by the bad vertices
is infinite. Since every vertex of $\st$ has degree at most 4, K\"{o}nig's lemma
provides an infinite path of (distinct) bad vertices $(x_0,y_0), (x_1,y_1), \ldots$
starting from $(x_0,y_0) = (2x,0)$. By definition of being bad, for every $m$,
there is a path from $(2x,0)$ to $(x_m,y_m)$ that is contained in some infinite subsequence
of the clusters $T_{n^m_i,x,t}$ with $n^m_i \to \infty$ as $i \to \infty$. Let $\ell_m$ be
the length of this path. Observe that $\ell_m \to \infty$ with $m$ because the distance
from $(2x,0)$ to $(x_m,y_m)$ in $\st$ must tend to infinity due to every vertex having degree
at most 4.

The YD $\lm_m$ formed by the cells of $\st$ that are at or below the cells on the path from
$(2x,0)$ to $(x_m,y_m)$ must be contained in every cluster $T_{n^m_i,x,t}$.
Since $T_n$ converges to $T_{\infty}$, this implies that $\lm_m \subset T_{\infty,x,t}$
for every $m$. Since $|\lm_m| \geq \ell_m \to \infty$, we deduce that $T_{\infty,x,t}$ is infinite.
However, this is a contradiction to $T_{\infty}$ being graded EG-admissible.
\end{proof}

\begin{thm} \label{thm:EGcontinuity}
Suppose a sequence of tableaux $T_n \to T_{\infty}$, and that every $T_n$ as well as $T_{\infty}$
is graded EG-admissible. Then $S(T_n) \to S(T_{\infty})$ as discrete subsets of $\Z \times \R_{\geq 0}$.
\end{thm}

\begin{proof}
A compact subset of $\Z \times \R_{\geq 0}$ is a finite, disjoint union of sets of the from
$\{x\} \times C$ for $x \in \Z$ and compact $C \subset \R_{\geq 0}$. Therefore, we must
show that for every such $x$ and $C$,
$$\limsup_n \; \# \,\big [S(T_n) \cap (\{x\} \times C)  \big]  \leq \# \,\big[ S(T_{\infty}) \cap (\{x\} \times C) \big].$$
Fix a $t > 0$ such that $C \subset [0,t]$.

Suppose $T$ is a graded EG-admissible tableau. The swaps of $T$ on $\{x\} \times [0,t]$
are the entries of $T^{\leq t}$ that exit from row $x$. Let $T_{x,t}$ denote the
cluster of $(2x,0)$ in $T^{\leq t}$. By Lemma \ref{lem:EGconsistent}, the swaps
of $T$ on $\{x\} \times [0,t]$ are completely determined by running the EG algorithm
on $T_{x,t}$. We deduce from Lemma \ref{lem:EGbounded} that there is a finite YD $\lm$ such that
$$\mathrm{support}(T_{n,x,t}) \subset \lm \;\;\text{for every}\;\; n\;\; \text{and}\;\; \mathrm{support}(T_{\infty,x,t}) \subset \lm.$$

Since $ \sup_{(x',y') \in \lm} | T_n(x',y') - T_{\infty}(x',y') | \to 0$,
we conclude that the following must occur for all sufficiently large $n$.
\begin{enumerate}
\item The order of the entries of $T_n$ on $\lm$ stabilizes to the order of the entries of $T_{\infty}$ on $\lm$.
\item For every $(x',y') \in \lm$, if $T_{\infty}(x',y') \notin C$ then $T_n(x',y') \notin C$.
\end{enumerate}

Once condition (1) holds then, due to $T_{n,x,t} \subset \lm$, a swap from $S(T_n)$
lies on $\{x\} \times C$ if and only if there is a cell $(x',y') \in \lm$ such that
$T_n(x',y') \in C$ and, when the EG algorithm is applied to $T_{\infty}$ restricted to $\lm$,
the entry at cell $(x',y')$ exits from row $x$. The same conclusion holds for swaps of
$S(T_{\infty})$ on $\{x\} \times C$. This property along with condition (2) implies that
$$S(T_n) \cap (\{x\} \times C) \subset S(T_{\infty}) \cap (\{x\} \times C) \;\;\text{for all large}\;\; n.$$
This completes the proof.
\end{proof}

\subsection{Completing the proof of Theorem \ref{thm:localsorting}} \label{sec:localsorting}
\smallskip

Theorem \ref{thm:localsorting} will follow from Theorem \ref{thm:EGcontinuity} once we prove
that the local tableau $\T_{\rm{edge}}$ is graded EG-admissible almost surely. To this end,
first observe that the entries of $\T_{\rm{edge}}$ are finite and distinct by part (1) of
Proposition \ref{thm:LSTstatistics}. We must show that, almost surely, the clusters
of $\T_{\rm{edge}}^{\leq t}$ are finite for every $t$.

By part (4) of Proposition \ref{thm:LSTstatistics}, the local tableau satisfies the following almost surely:
for every $t$ and $x$, there are integers $a,b \geq 0$ such that $\T_{\rm{edge}}(2x-2a,0) > t$
and $\T_{\rm{edge}}(2x+2b,0) > t$. When this property holds the tableau constraints imply that the
cluster of $\T_{\rm{edge}}^{\leq t}$ containing $(2x,0)$ must be contained within cells whose row
and column indices are both between $2x-2a$ and $2x+2b$. The set of such cells is finite, and so
the cluster of every bottom level cell in $\T_{\rm{edge}}^{\leq t}$ is finite.
Now if $\T_{\rm{edge}}(2x-k,k) \leq t$ then cell $(2x-k,k)$ belongs to the same cluster as
$(2x,0)$ in $\T_{\rm{edge}}^{\leq t}$ since the row entries are non-decreasing.
This implies that, almost surely, $\T_{\rm{edge}}^{\leq t}$ is EG-admissible for every $t$, as required.

Finally, we complete the proof.
The law of $S_{\alpha,n}$ is that of the Edelman-Greene algorithm applied to the
rescaled uniformly random staircase shaped tableau $\Tb^{\rm{rsc}}_{\sst}$ from (\ref{eqn:rsctab}).
Theorem \ref{corr:LST} asserts that $\Tb^{\rm{rsc}}_{\sst}$ converges weakly to
$\T_{\rm{edge}}$ as a tableau embedded in $\st$. By Skorokhod's representation theorem,
there exists random tableaux $\T_n$ and $\T_{\infty}$ defined on a common probability space
such that $\T_n$ has the law of $\Tb^{\rm{rsc}}_{\sst}$, $\T_{\infty}$ has the law of
$\T_{\rm{edge}}$, and $\T_n \to \T_{\infty}$ almost surely.

The tableaux $\T_n$ and $\T_{\infty}$ are graded EG-admissible almost surely.
Theorem \ref{thm:EGcontinuity} then implies that $S(\T_n)$ converges to $S(\T_{\infty})$ almost surely.
This means that $S_{\alpha,n}$, which has the law of $S(\T_n)$, converges weakly to $S(\T_{\rm{edge}})$,
which is the law of $S(\T_{\infty})$.
\qed

We conclude with some statistical properties of the local swap process.
\begin{prop} \label{thm:localsortingstats}
The process $S_{\rm{local}}$ has the following properties.
\begin{enumerate}
\item $S_{\rm{local}}$ is invariant under translations and reflection of the $\Z$-coordinate.
\item $S_{\rm{local}}$ is stationary in time in that for every $t \geq 0$, the process
$S_{\rm{local}}\cap (\Z \times \R_{\geq t})$ has the same law as (shifted) $S_{\rm{local}}$.
\item $S_{\rm{local}}$ is ergodic under translations of the $\Z$-coordinate in that the sigma-algebra
$$\F_{\rm{inv}} = \left \{ \text{Events of}\; S_{\rm{local}}\; \text{that are invariant under every translation}\right \} \;\;\text{is trivial} .$$
\end{enumerate}
\end{prop}
\begin{rem} We believe that $S_{\rm{local}}$ is also ergodic in the time coordinate.
	However, the proof of this is more challenging and, therefore, we leave it as a conjecture.
\end{rem}
\begin{proof}
We have that $S_{\rm{local}} = S(\T_{\rm{edge}})$ in law. Applying a $\Z$-automorphism to
$S_{\rm{local}}$ is the same as first applying its analogue to $\T_{\rm{edge}}$ (the maps $\phi_h$ and $\phi_{-}$),
and then applying the EG algorithm to the resulting tableau. Thus, the invariance of $S_{\rm{local}}$
under $\Z$-automorphisms follows from the corresponding invariance of $\T_{\rm{edge}}$ stated
in Proposition \ref{thm:LSTstatistics}.

Time stationarity of $S_{\rm{local}}$ is a consequence of the stationarity of finite
random sorting networks \cite[Theorem 1(i)]{AHRV}, as we explain. If $(s_1,\ldots, s_N)$
is the sequence of swaps of a random sorting network of $\mathfrak{S}_n$, then
$(s_1,\ldots, s_{N-1})$ has the same law as $(s_2, \ldots, s_N)$.

The ergodicity of $S_{\rm{local}}$ under $\Z$-translations is a consequence of the ergodicty of
$\T_{\rm{edge}}$ under translations (part 3 of Proposition \ref{thm:LSTstatistics}). Indeed,
a translation invariant event for $S_{\rm{local}}$ is the image of a translation invariant event of
$\T_{\rm{edge}}$ under the EG algorithm.
\end{proof}

\end{document}